\documentclass{amsart}

\usepackage[]{fontenc}
\usepackage[latin1, utf8]{inputenc}
\usepackage{geometry}
\usepackage{fancyhdr}
\usepackage{setspace}
\usepackage[all]{xy}
\usepackage{amssymb}
\usepackage{mathtools}

\usepackage{amsfonts}
\usepackage{amsmath}
\usepackage{amsthm}
\usepackage{a4wide}
\usepackage{textcomp}
\usepackage{epsfig}
\usepackage{tikz-cd}

\allowdisplaybreaks

\DeclareMathOperator{\gr}{gr}
\DeclareMathOperator{\Lag}{Lag}
\DeclareMathOperator{\Leg}{Leg}

\DeclareMathOperator{\im}{im}

\DeclareMathOperator{\ind}{ind}

\DeclareMathOperator{\id}{id}
\DeclareMathOperator{\can}{can}
\DeclareMathOperator{\pt}{pt}

\newcommand{\mfKill}[1]{}

\usepackage{aliascnt}
\usepackage[colorlinks, linkcolor=black,citecolor=black, urlcolor=black]{hyperref}

\newtheorem{thm}{Theorem}[section]

\newaliascnt{prop}{thm}
\newtheorem{prop}[prop]{Proposition}
\aliascntresetthe{prop}

\newaliascnt{defi}{thm}

\aliascntresetthe{defi}

\newaliascnt{ex}{thm}

\aliascntresetthe{ex}

\newaliascnt{rmk}{thm}
\newtheorem{rmk}[rmk]{Remark}
\aliascntresetthe{rmk}

\newaliascnt{lemma}{thm}
\newtheorem{lemma}[lemma]{Lemma}
\aliascntresetthe{lemma}

\newaliascnt{cor}{thm}

\aliascntresetthe{cor}

\newaliascnt{conjecture}{thm}

\aliascntresetthe{conjecture}

\usepackage{tikz,pgfplots}
\usetikzlibrary{decorations.markings, intersections, arrows.meta, positioning, calc, math, matrix}
\pgfplotsset{compat=1.18}

\tikzset{middlearrow/.style={
        decoration={markings,
            mark= at position 0.5 with {\arrow{#1}} ,
        },
        postaction={decorate}
    }
}

\title{Contact non-squeezing at large scale via generating functions}
\author{Maia Fraser}
\address{University of Ottawa, Ottawa, ON K1N-6N5, Canada}
\email{mfrase8@uottawa.ca}

\author{Sheila Sandon}
\address{Universit\'e de Strasbourg, CNRS, IRMA UMR 7501, F-67000 Strasbourg, France}
\email{sandon@math.unistra.fr}

\author{Bingyu Zhang}
\address{Centre for Quantum Mathematics, Syddansk Universitet, Campusvej 55, 5230 Odense M, Denmark}
\email{bingyuzhang@imada.sdu.dk}

\begin{document}

\begin{abstract}
\noindent 
Using SFT techniques,
Eliashberg, Kim and Polterovich (2006)
proved that
if $\pi R_2^2 \leq K \leq \pi R_1^2$
for some integer $K$
then there is no contact squeezing in $\mathbb{R}^{2n} \times S^1$
of the prequantization of the ball of radius $R_1$
into the prequantization of the ball of radius $R_2$.
This result was extended
to the case of balls of radius $R_1$ and $R_2$
with $1 \leq \pi R_2^2 \leq \pi R_1^2$
by Chiu (2017) and the first author (2016),
using respectively microlocal sheaves and SFT.
In the present article
we recover this general contact non-squeezing theorem
using generating functions,
a classical method
based on finite dimensional Morse theory.
More precisely,
we develop an equivariant version,
with respect to a certain action of a finite cyclic group,
of the generating function homology
for domains of $\mathbb{R}^{2n} \times S^1$
defined by the second author (2011).
A key role in the construction
is played by translated chains of contactomorphisms,
a generalization of translated points.
\end{abstract}

\maketitle


\section{Introduction}

Consider the Euclidean space
$\mathbb{R}^{2n}$,
with coordinates $(x_1, y_1, \dots, x_n, y_n)$,
equipped with the standard symplectic form
$\omega_0 = \sum_{j = 1}^n dx_j \wedge dy_j$.
The seminal symplectic non-squeezing theorem
by Gromov \cite{Gromov}
says that if $R_1 > R_2$
then there is no symplectic embedding
of the open ball $B^{2n} (R_1)$
of radius $R_1$ in $(\mathbb{R}^{2n}, \omega_0)$
into the open cylinder $B^2 (R_2) \times \mathbb{R}^{2n-2}$.
At first sight it might seem that no similar result
can hold in contact topology.
Indeed, consider the Euclidean space
$\mathbb{R}^{2n+1}$,
with coordinates $(x_1, y_1, \dots, x_n, y_n, \theta)$,
endowed with the standard contact structure
$\xi_0 = \ker \big( d\theta
- \sum_{j = 1}^n \frac{x_j dy_j - y_j dx_j}{2} \big)$;
then contact transformations
of the form
\[
(x_1, y_1, \dots, x_n, y_n, \theta)
\mapsto (c x_1, c y_1, \dots, c x_n, c y_n, c^2 \theta) 
\]
for $c \in \mathbb{R}_{> 0}$ small enough
map any given domain into an arbitrarily small one.
However,
in 2006 Eliashberg, Kim and Polterovich \cite{EKP}
discovered the following surprising
contact non-squeezing phenomenon
in $\mathbb{R}^{2n} \times S^1$,
where $S^1$ is taken to be the quotient $\mathbb{R}/\mathbb{Z}$
and $\mathbb{R}^{2n} \times S^1$ is endowed with the contact structure,
still denoted $\xi_0$,
induced by the standard contact structure on $\mathbb{R}^{2n+1}$.
Let $\mathcal{V}_1$ and $\mathcal{V}_2$
be two open domains in 
$( \mathbb{R}^{2n} \times S^1, \xi_0 )$.
A \emph{contact squeezing}
of $\mathcal{V}_1$ into $\mathcal{V}_2$
is a compactly supported
contact isotopy $\{\phi_t\}$
from the closure of $\mathcal{V}_1$
into $\mathbb{R}^{2n} \times S^1$
such that $\phi_0$ is the inclusion
and $\phi_1$ maps the closure of $\mathcal{V}_1$
into $\mathcal{V}_2$.
If the closure of $\mathcal{V}_1$ is compact,
the existence of a contact squeezing
of $\mathcal{V}_1$ into $\mathcal{V}_2$
is equivalent to the existence
of a compactly supported contactomorphism
of $(\mathbb{R}^{2n} \times S^1, \xi_0)$,
contact isotopic to the identity\footnote{
By a
compactly supported contactomorphism
contact isotopic to the identity
we always mean a contactomorphism
that is isotopic to the identity
through a compactly supported contact isotopy.},
mapping the closure of $\mathcal{V}_1$
into $\mathcal{V}_2$.
As in \cite{Fraser},
we call \emph{coarse contact squeezing}
of $\mathcal{V}_1$ into $\mathcal{V}_2$
a compactly supported contactomorphism
of $(\mathbb{R}^{2n} \times S^1, \xi_0)$,
not necessarily contact isotopic to the identity\footnote{
We do not know whether there exist
compactly supported contactomorphisms
of $(\mathbb{R}^{2n} \times S^1, \xi_0)$
that are not contact isotopic to the identity,
thus whether the notion
of contact squeezing and coarse contact squeezing
are indeed different.},
mapping the closure of $\mathcal{V}_1$ into $\mathcal{V}_2$.
For a domain $\mathcal{U}$ of $\mathbb{R}^{2n}$
we denote by $\widehat{\mathcal{U}}$
the domain $\mathcal{U} \times S^1$
of $\mathbb{R}^{2n} \times S^1$.
The non-squeezing theorem of Eliashberg, Kim and Polterovich \cite{EKP}
states that if $\pi R_2^2 \leq K \leq \pi R_1^2$
for some $K \in \mathbb{Z}$
then there is no coarse contact squeezing
(hence no contact squeezing)
of $\widehat {B^{2n} (R_1)}$
into $\widehat {B^{2n} (R_2)}$.
In contrast,
the same authors also proved that
if $\pi R_1^2 < 1$ and $n > 1$
then there is a contact squeezing
of $\widehat {B^{2n} (R_1)}$ into $\widehat {B^{2n} (R_2)}$
for any $R_2$
(while it is pointed out in \cite{EKP}
that in dimension $3$
the contact shape invariant \cite{Eli91}
obstructs coarse contact squeezing
of $\widehat {B^{2} (R_1)}$
into $\widehat {B^{2} (R_2)}$
for any $R_2 \leq R_1$).

Contact squeezing for $\pi R_2^2 \leq \pi R_1^2 < 1$
(in dimension higher than $3$)
is a manifestation of flexibility,
and is proved in \cite{EKP}
by a geometric construction
that uses a positive contractible loop of contactomorphisms
of the standard contact sphere $(S^{2n-1}, \xi_0)$.
On the other hand,
(coarse) non-squeezing for $\pi R_2^2 \leq K \leq \pi R_1^2$
is a rigidity phenomenon,
which is proved in \cite{EKP}
with holomorphic curves techniques coming from SFT,
and reproved
(in its non-coarse version)
by the second author \cite{San11} with generating functions.
The case $1 < \pi R_2^2 \leq \pi R_1^2$
with no integers between $\pi R_2^2$ and $\pi R_1^2$
was left open in \cite{EKP} and \cite{San11},
and settled by Chiu \cite{Chiu}
(in the non-coarse version)
and the first author \cite{Fraser}
(in the apriori stronger coarse version)
using respectively microlocal sheaves and SFT.
In the present paper we recover
(the non-coarse version of) this result
by a proof that uses generating functions,
in the continuation of \cite{San11}.
More precisely we prove the following theorem.

\begin{thm}[Contact non-squeezing at large scale]
\label{theorem: main}
For any $R_1$ and $R_2$ with $1 \leq \pi R_2^2 \leq \pi R_1^2$
there is no contact squeezing of $\widehat {B^{2n} (R_1)}$
into $\widehat {B^{2n} (R_2)}$.
\end{thm}

A different approach to prove \autoref{theorem: main}
with generating functions
has been outlined by the first author in \cite{Fraser - poster},
and is developed in a work in preparation
with Traynor \cite{FT}.
Following an idea introduced in \cite{Chiu},
and exploited also in \cite{Fraser}
in the context of holomorphic curves,
our proof of \autoref{theorem: main}
uses an equivariant homology
with respect to a certain action
of a cyclic group $\mathbb{Z}_k$.
Specifically,
we develop a $\mathbb{Z}_k$-equivariant version
of the generating function homology
for domains of $(\mathbb{R}^{2n} \times S^1, \xi_0)$
that is defined in \cite{San11}
(while the approach in \cite{Fraser - poster} and \cite{FT}
uses the homology groups of \cite{San11}
without introducing any cyclic action).
In order to explain this
we first recall the idea of the proof in \cite{San11}
of non-squeezing for integers.

The homology groups defined in \cite{San11}
are a generalization to domains
of $(\mathbb{R}^{2n} \times S^1, \xi_0)$
of the symplectic homology groups
for domains of $(\mathbb{R}^{2n}, \omega_0)$
defined by Traynor \cite{Traynor}.
For any compactly supported 
Hamiltonian symplectomorphism
$\varphi$ of $(\mathbb{R}^{2n}, \omega_0)$
and $a \leq b$ in $\mathbb{R} \cup \{\pm \infty\}$
not belonging to the action spectrum of $\varphi$,
Traynor defined $G_{\ast}^{(a, b]} (\varphi)$
to be the relative homology
of the sublevel sets at $a$ and $b$
of any generating function quadratic at infinity of $\varphi$.
She then defined the homology
$G_{\ast}^{(a, b]} (\mathcal{U})$
of a domain $\mathcal{U}$ of $(\mathbb{R}^{2n}, \omega_0)$
by a limit process
over compactly supported
Hamiltonian symplectomorphisms
supported in $\mathcal{U}$.
Symplectic invariance of these groups,
i.e.\ the fact that
$G_{\ast}^{(a, b]} \big(\psi(\mathcal{U})\big)
\cong G_{\ast}^{(a, b]} (\mathcal{U})$
for every compactly supported Hamiltonian symplectomorphism
$\psi$ of $(\mathbb{R}^{2n}, \omega_0)$,
is proved using invariance by conjugation
of the groups associated to Hamiltonian symplectomorphisms,
which follows from the fact that
the critical points of any generating function quadratic at infinity
of a Hamiltonian symplectomorphism $\varphi$
are in 1--1 correspondence with the fixed points of $\varphi$,
with critical values given
by the symplectic action,
and the fact that the action spectrum of a Hamiltonian symplectomorphism
is invariant by conjugation.
Similarly, 
in \cite{San11} the second author defined
the homology $G_{\ast}^{(a, b]} (\phi)$
of a compactly supported contactomorphism $\phi$
of $(\mathbb{R}^{2n} \times S^1, \xi_0)$
contact isotopic to the identity
and, by a limit process,
the homology $G_{\ast}^{(a, b]} (\mathcal{V})$
of a domain $\mathcal{V}$
of $( \mathbb{R}^{2n} \times S^1 , \xi_0 )$.
A crucial difference
with respect to the symplectic case
is that this homology is contact invariant,
i.e.\ $G_{\ast}^{(a, b]} \big(\psi(\mathcal{V})\big) \cong G_{\ast}^{(a, b]} (\mathcal{V})$
for every compactly supported contactomorphism
$\psi$ of $(\mathbb{R}^{2n} \times S^1, \xi_0)$
contact isotopic to the identity,
only if $a$ and $b$ are integers
(or $a, b \in \{\pm \infty\}$).
This difference is due to the fact that 
the critical points of any generating function
of a contactomorphism $\phi$
of $(\mathbb{R}^{2n} \times S^1, \xi_0)$
are in 1--1 correspondence with the translated points of $\phi$,
with critical values given by the contact action,
and the fact that translated points
of contactomorphisms of $(\mathbb{R}^{2n} \times S^1, \xi_0)$
are invariant by conjugation
only if they have integral action.
More precisely,
recall that a translated point
of a contactomorphism $\phi$
of $(\mathbb{R}^{2n} \times S^1, \xi_0)$
with respect to the standard contact form
$\alpha_0 = d\theta
- \sum_{j = 1}^n \frac{x_j dy_j - y_j dx_j}{2}$
is a point $p$ of $\mathbb{R}^{2n} \times S^1$
such that $p$ and $\phi(p)$
are in the same Reeb orbit
and $g(p) = 0$,
where $g$ is the conformal factor of $\phi$,
i.e.\ the function defined by the relation
$\phi^{\ast}\alpha_0 = e^g \alpha_0$.
If $\phi$ is compactly supported
and contact isotopic to the identity
then it can be uniquely lifted
to a compactly supported contactomorphism $\Phi$
of $(\mathbb{R}^{2n+1}, \xi_0)$
that is equivariant by translation by $1$
in the $\theta$-direction.
The length of the Reeb chord
from $P$ to $\Phi(P)$
for any lift $P$ of $p$ to $\mathbb{R}^{2n+1}$
is then called the (contact) action
of the translated point $p$ of $\phi$.
If $K$ is an integer
then there is a 1--1 correspondence
between the translated points of action $K$ of $\phi$
and the translated points of action $K$
of any conjugation $\psi \circ \phi \circ \psi^{-1}$.
Using this,
it is shown in \cite{San11}
that if $a, b \in \mathbb{Z} \cup \{\pm \infty\}$
then $G_{\ast}^{(a, b]} (\psi \circ \phi \circ \psi^{-1})
\cong G_{\ast}^{(a, b]} (\phi)$
for every compactly supported contactomorphism
$\psi$ of $(\mathbb{R}^{2n} \times S^1, \xi_0)$
contact isotopic to the identity,
and in turn this is used to show that
the groups $G_{\ast}^{(a, b]} (\mathcal{V})$
are contact invariants.
It is also proved that for any domain $\mathcal{U}$
of $(\mathbb{R}^{2n}, \omega_0)$
we have $G_{\ast}^{(a, b]} (\widehat{\mathcal{U}} )
\cong
G_{\ast}^{(a, b]} (\mathcal{U}) \otimes H_{\ast} (S^1)$.
The fact that there is no contact squeezing
of $\widehat{B^{2n}(R_1)}$ into $\widehat{B^{2n}(R_2)}$
if $\pi R_2^2 \leq K \leq \pi R_1^2$
for some integer $K$
then follows from the calculation
of the homology groups of balls in $(\mathbb{R}^{2n}, \omega_0)$
done by Traynor \cite{Traynor}.
For a ball $B^{2n} (R)$
and $a \in \mathbb{R}_{>0}$,
Traynor proved that,
with coefficients in $\mathbb{Z}_2$
and any integer $l> 0$,
\[
G_{2nl}^{(a, \infty]} \big(B^{2n}(R)\big)
\cong \begin{cases}
\mathbb{Z}_2 \quad &\text{if }
(l-1) \pi R^2 \leq a < l \pi R^2 \\
0 \quad &\text{otherwise,}
\end{cases}
\]
and moreover that for $R_2 < R_1$ with
$(l-1) \pi R_2^2 \leq a < l \pi R_2^2$
and $(l-1) \pi R_1^2 \leq a < l \pi R_1^2$
the homomorphism
\[
G_{2nl}^{(a, \infty]} \big(B^{2n}(R_1)\big) \rightarrow
G_{2nl}^{(a, \infty]} \big(B^{2n}(R_2)\big)
\]
induced by the inclusion
of $B^{2n}(R_2)$ into $B^{2n}(R_1)$
is an isomorphism.

\begin{figure}[htbp]
    \centering

\begin{tikzpicture}

\draw[help lines, gray!30](0,0) grid[step={($(3/2, 1) - (0, 0)$)}] (7.7,4.2);

\draw[->]    (0,0)--(9,0) node[anchor=north east]  {$a$} ;

\node[anchor=north] at (0,0)   {$0$} ;
 \node[anchor=north] at (3/2,0)   {$\pi R^2$} ;
\foreach \i in {2,3,4,5} 
{ \node[anchor=north] at (3*\i/2,0)   {$\i \pi R^2$} ;}

\foreach \i in {1,2,3,4}
{ \node[anchor=east] at (9.5,\i)   {$l={\i} $} ;}

\draw[thick]  (0,1)--({3/2-0.05},1)     ;
\foreach \i in {2,3,4} 
{ \draw[thick]  ({3*(\i -1)/2},\i)--({3*\i/2-0.05},\i)     ;}

\foreach \i in {2,3,4} 
{ \draw[fill=black] ({3*(\i -1)/2},\i) circle (0.05);}
\foreach \i in {1,2,3,4} 
{ \draw ({3*\i/2},\i) circle (0.05);}

\foreach \i in {1,2,3,4}
{ \node[anchor=south] at ({3*\i/2-3/4},\i)   {$\mathbb{Z}_2$} ;}

\end{tikzpicture}

    \caption{The homology groups $G_{2nl}^{(a,\infty]}(B^{2n}(R))$.}
    \label{fig:non-equivariant barcode}
\end{figure}
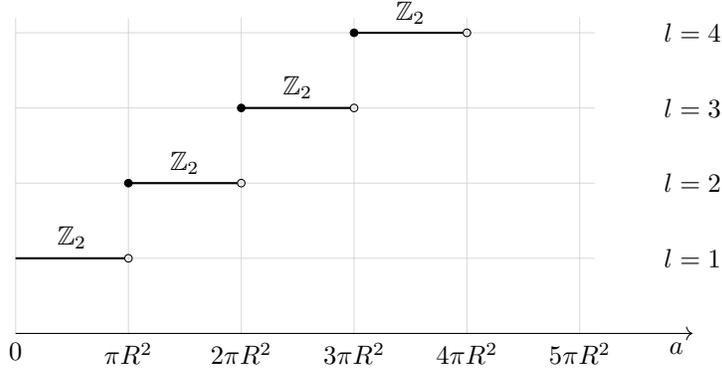

As shown in \cite{San11},
these results imply (the non-coarse version of)
the non-squeezing theorem of \cite{EKP}.
Indeed, let $\pi R_2^2 \leq K \leq \pi R_1^2$
for some integer $K$
and suppose by contradiction that there is
a contact squeezing of $\widehat{B^{2n}(R_1)}$
into $\widehat{B^{2n}(R_2)}$.
This then induces a contact squeezing
of a neighborhood of $\widehat{B^{2n}(R_1)}$
into $\widehat{B^{2n}(R_2)}$,
so without loss of generality we may assume 
$K < \pi R_1^2$.
By the contact isotopy extension theorem
there is a compactly supported contactomorphism $\psi$
of $(\mathbb{R}^{2n} \times S^1, \xi_0)$
contact isotopic to the identity
such that
$\psi \big( \widehat{B^{2n}(R_1)} \big) \subset  \widehat{B^{2n}(R_2)}$.
Take $R_3$ large enough,
so that $\psi \big( \widehat{B^{2n}(R_3)} \big) =  \widehat{B^{2n}(R_3)}$.
We then have a commutative diagram
\[
\begin{tikzcd}
G_{\ast}^{(K, \infty]} \big(\widehat{B^{2n}(R_3)}\big)
\arrow{rr} \arrow{d}{\cong}
&  & G_{\ast}^{(K, \infty]} \big(\widehat{B^{2n}(R_1)}\big)
\arrow{d}{\cong}\\
G_{\ast}^{(K, \infty]} \big(\widehat{B^{2n}(R_3)}\big)
\arrow{r} 
& G_{\ast}^{(K, \infty]} \big(\widehat{B^{2n}(R_2)}\big)
\arrow{r}
& G_{\ast}^{(K, \infty]} \big(\psi(\widehat{B^{2n}(R_1)}) \big)
\end{tikzcd}
\]
where the horizontal arrows
are the homomorphisms induced by the inclusions
and the vertical arrows the isomorphisms induced by $\psi$.
But this gives a contradiction,
because for $\ast = 2n$
the horizontal arrow on the top
is an isomorphism from $\mathbb{Z}_2$ to $\mathbb{Z}_2$,
while $G_{\ast}^{(K, \infty]}
\big(\widehat{B^{2n}(R_2)}\big) \cong 0$.

\begin{rmk}\label{remark: category theoretic construction}
In \cite{Traynor}
(and similarly in \cite{San11}
for the generalization to the contact case)
the homology of a domain is defined
by taking the inverse limit
of the homologies of the Hamiltonian symplectomorphisms
supported in the domain,
which are considered to form an inverse system
over the mentioned set of Hamiltonian symplectomorphisms, 
the latter being a directed set
under the partial order $\leq$
defined by setting $\varphi_0 \leq \varphi_1$
if $\varphi_1 \circ \varphi_0^{-1}$
is the time-1 map of the flow
of a compactly supported non-negative Hamiltonian function. 
The property of being an inverse system
relies on having an induced morphism
from the homology of $\varphi_1$
to the homology of $\varphi_0$
whenever $\varphi_0 \leq \varphi_1$. 
As noted in \cite{Fraser - poster}, however,
well-definedness of this morphism
is not addressed in \cite{Traynor} (nor in \cite{San11}).
One point not discussed explicitly
in the homology theories of \cite{Traynor, San11}
is that the homology group
($\mathbb{Z}_2$-vector space)
associated to a Hamiltonian symplectomorphism
is by construction only determined
up to certain isomorphisms
(since generating functions are only unique
up to stabilization and fibre preserving diffeomorphisms);
it is therefore, formally speaking,
a groupoid of vector spaces,
not a vector space.
However, claimed maps between homologies
are defined in \cite{Traynor, San11}
only on specific vector spaces of each groupoid
(arising from specific generating functions)
but are never shown to pass to morphisms
of the associated groupoids.
Moreover, it is a priori not clear whether
the concrete calculations
done with specific vector spaces and morphisms
in \cite{Traynor, San11}
also carry over to the mentioned groupoids.
In the present article
we thus deviate from \cite{Traynor, San11}
and give a more precise construction
that avoids the above concerns 
by taking as underlying directed set
Hamiltonian isotopies
instead of their time-1 maps,
and by using a category theoretic formulation
that does not require the choice of specific generating functions
and that we believe also provides advantages for future work.
The construction is written directly for equivariant homology,
but can be done also in the non-equivariant case
clarifying thus the definition and functorial properties
of the homology groups defined in \cite{Traynor, San11}.
It follows from the analogue of \autoref{proposition: commuting isomorphism}
that the homology $G_{\ast}^{(a, b]} \big(\{\varphi_t\}\big)$
associated via our construction
to a compactly supported Hamiltonian isotopy
$\{\varphi_t\}_{t \in [0, 1]}$ of $(\mathbb{R}^{2n}, \omega_0)$
is isomorphic to the  homology $G_{\ast}^{(a, b]} (\varphi_1)$
in the sense of \cite{Traynor},
and similarly in the contact case
(although these isomorphisms are not canonical,
since they depend on the choice of generating functions).
As we will see,
the advantage of working with isotopies instead of their time-1 maps
is that the homomorphisms involved in the limit constructions
to define the homology of domains become canonical,
and so the functorial properties easier to prove.
\end{rmk}

In order to prove \autoref{theorem: main}
we develop a $\mathbb{Z}_k$-equivariant version
of the generating function homology
of domains of $(\mathbb{R}^{2n}, \omega_0)$
and $(\mathbb{R}^{2n} \times S^1, \xi_0)$.
Given a compactly supported Hamiltonian symplectomorphism
$\varphi$ of $(\mathbb{R}^{2n}, \omega_0)$,
we consider the function
\[
F^{\sharp k}:
\mathbb{R}^{2nk} \times \mathbb{R}^{Nk} \rightarrow \mathbb{R}
\]
obtained by applying a composition formula
due to Allais \cite{Allais}
(\autoref{proposition: composition formula})
to a generating function quadratic at infinity
$F: \mathbb{R}^{2n} \times \mathbb{R}^N
\rightarrow \mathbb{R}$
of $\varphi$.
As we will see,
the function $F^{\sharp k}$
is invariant by the action of $\mathbb{Z}_k$
that cyclically permutes the coordinates,
and its critical points are in 1--1 correspondence
with the $k$-periodic points of $\varphi$,
i.e.\ the fixed points of $\varphi^k$,
with critical values given by the symplectic action.
Moreover,
under this 1--1 correspondence
the $\mathbb{Z}_k$-action
on the set of critical points of $F^{\sharp k}$
corresponds to the $\mathbb{Z}_k$-action
on the set of $k$-periodic points of $\varphi$
generated by the map that sends
a $k$-periodic point $p$
to the $k$-periodic point $\varphi(p)$.
We will also see that, after a change of coordinates,
$F^{\sharp k}$ can be extended to a function
\[
\overline{F^{\sharp k}}:
S^{2n} \times \mathbb{R}^{2n (k-1)} \times \mathbb{R}^{Nk} \rightarrow \mathbb{R}
\]
that is invariant
by the induced $\mathbb{Z}_k$-action
and, if $F$ is special,
quadratic at infinity for the vector bundle
$S^{2n} \times \mathbb{R}^{2n (k-1)} \times \mathbb{R}^{Nk} \rightarrow S^{2n}$.

\begin{figure}[htbp]
\centering

\begin{tikzpicture}[yscale=0.6, xscale=-1]

\def \k {7}
\pgfmathsetmacro\m{\k-1} 
\pgfmathsetmacro\w{2*\k}
\def \radius {2cm}
\def \radiusphi {1.7cm}
\def \Radius {2.7cm}
\def \RadiusO {3.2cm} 
\def \margin {6} 

\foreach \s in {1,2,...,\k}
{
  \node at ({360/\k * (\s )}:\radius) {$\bullet$};
  \draw[->, >=latex] ({360/\k * (\s )+\margin}:\radius) 
    arc ({360/\k * (\s )+\margin}:{360/\k * (\s+1)-\margin}:\radius);
}
\node at ({0}:\RadiusO) {$p=\varphi^{\k}(p)$};

\foreach \s in {1,2,...,\k}
{
  \node at ({360/\k * (\s )+360/\w}:\radiusphi) {$\varphi$};
  
}
\node at ({0}:\RadiusO) {$p=\varphi^{\k}(p)$};

\foreach \s in {1,...,\m}
{
  \node at ({360/\k * (\s )}:\Radius) {$\varphi^{\s}(p)$};
}

\end{tikzpicture}

\caption{If $p$ is a $7$-periodic point of $\varphi$
then all $\varphi^{j}(p)$, $j = 1,\dots,6$,
are $7$-periodic points of $\varphi$.
The group $\mathbb{Z}_7$ acts on the set
$\{\varphi^{j}(p) \;\lvert\; j=0,\dots,6$\}.}

\end{figure}
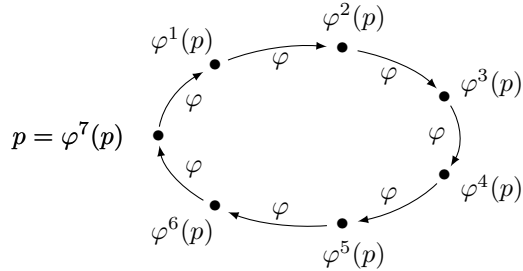

For $a \leq b$ in $\mathbb{R} \cup \{\pm \infty\}$
not belonging to the action spectrum of $\varphi^k$
the equivariant homology
$G_{\mathbb{Z}_k, \ast}^{(a, b]}(F)$
is defined to be the relative $\mathbb{Z}_k$-equivariant homology
(with $\mathbb{Z}_k$-coefficients)
of the sublevel sets of $\overline{F^{\sharp k}}$ at $a$ and $b$.
Although $G_{\mathbb{Z}_k, \ast}^{(a, b]}(F)$
is isomorphic to $G_{\mathbb{Z}_k, \ast}^{(a, b]}(F')$
for any other generating function quadratic at infinity $F'$ of $\varphi$,
as discussed in \autoref{remark: category theoretic construction}
this is not enough to obtain well-defined homology groups for $\varphi$
just by posing
$G_{\mathbb{Z}_k, \ast}^{(a, b]}(\varphi) = G_{\mathbb{Z}_k, \ast}^{(a, b]}(F)$.
In \autoref{section: equivariant symplectic homology}
we thus define the equivariant homology
$G_{\mathbb{Z}_k, \ast}^{(a, b]} \big(\{\varphi_t\}\big)$
of a compactly supported Hamiltonian isotopy
$\{\varphi_t\}_{t\in [0,1]}$ of $(\mathbb{R}^{2n}, \omega_0)$
by a category theoretic construction
that does not require the choice of specific generating functions,
and that ensures that $G_{\mathbb{Z}_k, \ast}^{(a, b]} \big(\{\varphi_t\}\big)$
is isomorphic to $G_{\mathbb{Z}_k, \ast}^{(a, b]} (F)$
for any generating function quadratic at infinity $F$ of $\varphi_1$.
We then define the equivariant generating function homology
$G_{\mathbb{Z}_k, \ast}^{(a, b]}(\mathcal{U})$
of a domain $\mathcal{U}$ of $(\mathbb{R}^{2n}, \omega_0)$
by taking the inverse limit,
in a certain sense,
of the equivariant homologies $G_{\mathbb{Z}_k, \ast}^{(a, b]} \big(\{\varphi_t\}\big)$
over all compactly supported Hamiltonian isotopies $\{\varphi_t\}$
supported in $\mathcal{U}$.

In the contact case
the relevant geometric objects to consider
in order to develop a $\mathbb{Z}_k$-equivariant homology theory
are not the translated points
of $k$-th iterations of contactomorphisms,
on which there is no natural $\mathbb{Z}_k$-action,
but what we call the translated $k$-chains.
A \emph{translated} $k$\emph{-chain}
of (contact) action $tk$
of a contactomorphism $\phi$
of $(\mathbb{R}^{2n + 1}, \xi_0)$
with respect to the contact form $\alpha_0$
is a $k$-tuple of points $(p_1, \dots, p_k)$
such that
\[
g(p_1) + \dots + g (p_k) = 0 \,,
\]
where $g$ denotes the conformal factor of $\phi$,
and
\[
p_{j+1} = \varphi_{-t}^{\alpha_0} \circ \phi \, (p_j)
\]
for all $j$,
with the convention $p_{k+1} = p_1$,
where $\{\varphi_t^{\alpha_0}\}$
denotes the Reeb flow.
In particular,
a translated $1$-chain is a translated point,
and its action coincides
with the action as a translated point.
We say that a $k$-tuple $(p_1, \dots, p_k)$
of points of $\mathbb{R}^{2n} \times S^1$
is a translated $k$-chain of (contact) action $tk$
of a compactly supported contactomorphism $\phi$
of $(\mathbb{R}^{2n} \times S^1, \xi_0)$
contact isotopic to the identity
if there is a $k$-tuple of points of $\mathbb{R}^{2n+1}$
projecting to $(p_1, \dots, p_k)$
that is a translated $k$-chain with action $tk$
of the lift of $\phi$ to $(\mathbb{R}^{2n+1}, \xi_0)$.

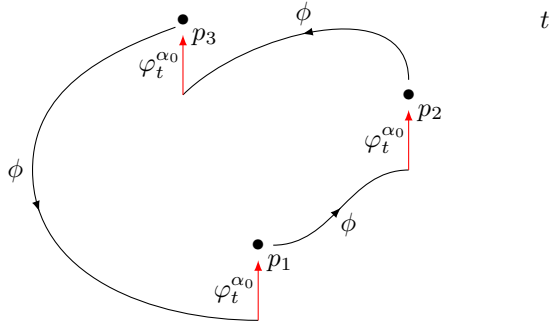
\begin{figure}[htbp]
\centering
\begin{tikzpicture}

\node at (2,1) {$\bullet$};
\node[anchor=north west] at (2,1) {$p_2$};
\draw[middlearrow={latex}]  (1.9,1) to[out=180,in=0]  (0,0);
\draw[->, >=latex,red]  (2,2) to (2,1.1);
\node[anchor=south] at (1.2,0) {$\phi$};
\node[anchor=east] at (2.1,1.5) {$ \varphi_{-t}^{\alpha_0}$};
  
\node at (-1,2) {$\bullet$};
\node[anchor=north west] at (-1,2) {$p_1$};
\draw[middlearrow={latex}]  (-1,2) to[out=55,in=135]  (2,2);
\draw[->, >=latex,red]  (-1,3) to (-1,2.1);
\node[anchor=south] at (0.2,2.7) {$\phi$};
\node[anchor=east] at (-1,2.5) {$ \varphi_{-t}^{\alpha_0}$};

\node at (0,-1) {$\bullet$};
\node[anchor=north west] at (0,-1) {$p_3$};
\draw[middlearrow={latex}]  (-0.1,-1) to[out=180,in=-90] (-3,1) to[out=90,in=190] (-1,3);
\draw[->, >=latex,red]  (0,0) to (0,-0.9);
\node[anchor=east] at (-3,0.2) {$\phi$};
\node[anchor=east] at (0,-0.5) {$ \varphi_{-t}^{\alpha_0}$};

\draw[->, >=latex]  (4,-1) to (4,3);
\node[anchor=east] at (4,3) {$t$};
\end{tikzpicture}

\caption{If the triple $(p_1,p_2,p_3)$
is a translated $3$-chain of $\phi$
then $(p_2, p_3, p_1)$ and $(p_3, p_1, p_2)$
are translated $3$-chains of $\phi$.
The group $\mathbb{Z}_3$
acts on the set
$\{\, (p_1, p_2, p_3) \,,\, (p_2, p_3, p_1)
\,,\, (p_3, p_1, p_2)\,\}$.}

\end{figure}

Given a compactly supported contactomorphism
$\phi$ of $(\mathbb{R}^{2n} \times S^1, \xi_0)$
contact isotopic to the identity
with generating function quadratic at infinity
$F: \mathbb{R}^{2n+1} \times \mathbb{R}^N \rightarrow \mathbb{R}$,
in \autoref{section: invariant generating functions contact}
we define a function
\[
\mathcal{P}_F^{(k)}: \mathbb{R}^{(2n+2)k} \times \mathbb{R}^{Nk} \rightarrow \mathbb{R}
\]
that is invariant by the action of $\mathbb{Z}_k$
that cyclically permutes the coordinates,
and whose critical points
are in 1--1 correspondence with 
the translated $k$-chains of $\phi$,
with critical values given by the contact action,
in such a way that under this 1--1 correspondence
the $\mathbb{Z}_k$-action
on the set of critical points of $\mathcal{P}_F^{(k)}$
corresponds to the $\mathbb{Z}_k$-action
on the set of translated $k$-chains of $\phi$
generated by the map that sends
a translated $k$-chain $(p_1, \dots, p_k)$
to the translated $k$-chain $(p_2, \dots, p_k, p_1)$.
As we will see,
after a change of coordinates
$\mathcal{P}_F^{(k)}$ extends to a continuous function on
$S^{2n} \times \mathbb{R}^2 \times \mathbb{R}^{(2n+2)(k-1)} \times \mathbb{R}^{Nk}$
that descends to a continuous function
\[
\overline{\mathcal{P}_F^{(k)}}:
S^{2n} \times S^1 \times \mathbb{R}^{(2n+2)(k-1)} \times \mathbb{R}^{Nk}
\rightarrow \mathbb{R} \,,
\]
invariant by the induced $\mathbb{Z}_k$-action.
Moreover, if $F$ is special then $\overline{\mathcal{P}_F^{(k)}}$
is smooth and somehow standard at infinity
(in particular, all its critical points
are contained in a compact set).

For $a \leq b$ in $k\mathbb{Z} \cup \{\pm \infty\}$
that are not equal to the action of any translated $k$-chain of $\phi$,
in \autoref{section: equivariant contact homology}
we define $G_{\mathbb{Z}_k, \ast}^{(a, b]}(F)$
to be the relative $\mathbb{Z}_k$-equivariant homology
(with $\mathbb{Z}_k$-coefficients)
of the sublevel sets of $\overline{\mathcal{P}_F^{(k)}}$ at $a$ and $b$.
Similarly to the symplectic case,
we then define the equivariant homology
$G_{\mathbb{Z}_k, \ast}^{(a, b]} \big(\{\phi_t\}\big)$
of a compactly supported contact isotopy
by a category theoretical construction
that does not require the choice of specific generating functions
and that ensures that $G_{\mathbb{Z}_k, \ast}^{(a, b]} \big(\{\phi_t\}\big)$
is isomorphic to $G_{\mathbb{Z}_k, \ast}^{(a, b]}(F)$
for any generating function quadratic at infinity $F$ of $\phi_1$,
and use these equivariant homology groups
to define the equivariant generating function homology
$G_{\mathbb{Z}_k, \ast}^{(a, b]}(\mathcal{V})$
of domains of $(\mathbb{R}^{2n} \times S^1, \xi_0)$.
As we will see,
contact invariance of these groups
essentially follows from the fact that
translated $k$-chains of contact action in $k\mathbb{Z}$
are invariant by conjugation.

In \autoref{section: relation symplectic contact}
we prove that
\[
G_{\mathbb{Z}_k, \ast}^{(a, b]}(\widehat{\mathcal{U}})
\cong G_{\mathbb{Z}_k, \ast}^{(a, b]}(\mathcal{U})
\otimes H_{\ast}(S^1)
\]
for any domain $\mathcal{U}$ of $(\mathbb{R}^{2n}, \omega_0)$.
\autoref{theorem: main}
then follows from the calculation
of the symplectic equivariant homology
of balls in $(\mathbb{R}^{2n}, \omega_0)$.
We prove in \autoref{section: balls}
that if $k$ is prime and $0 < l < k$
then for any $a > 0$ we have
\[
G_{\mathbb{Z}_k, 2nl}^{(a, \infty]} \big(B^{2n}(R)\big)
\cong \begin{cases}
\mathbb{Z}_k \quad &\text{if } a < l \pi R^2 \\
0 \quad &\text{otherwise,}
\end{cases}
\]
and moreover that for $a < l \pi R_2^2 \leq l \pi R_1^2$
the homomorphism
\[
G_{\mathbb{Z}_k, 2nl}^{(a, \infty]} \big(B^{2n}(R_1)\big)
\rightarrow
G_{\mathbb{Z}_k, 2nl}^{(a, \infty]} \big(B^{2n}(R_2)\big)
\]
induced by the inclusion
of $B^{2n}(R_2)$ into $B^{2n}(R_1)$
is an isomorphism.

These results allow us to obtain a proof
of \autoref{theorem: main}.
Indeed, suppose by contradiction
that for $R_1, R_2$ with $1 \leq \pi R_2^2 \leq \pi R_1^2$
there is a contact squeezing of $\widehat{B^{2n}(R_1)}$
into $\widehat{B^{2n}(R_2)}$.
This then induces a contact squeezing
of a neighborhood of $\widehat{B^{2n}(R_1)}$
into $\widehat{B^{2n}(R_2)}$,
so without loss of generality we may assume 
$1 < \pi R_2^2 < \pi R_1^2$.
By the contact isotopy extension theorem
there is a compactly supported contactomorphism
$\psi$ of $(\mathbb{R}^{2n} \times S^1, \xi_0)$
contact isotopic to the identity
such that $\psi \big(\widehat{B^{2n} (R_1)}\big)
\subset \widehat{B^{2n} (R_2)}$.
Take $R_3$ big enough,
so that $\psi \big( \widehat{B^{2n} (R_3)} \big)
= \widehat{B^{2n} (R_3)}$.
Take $k$ prime and $l < k$
so that $\pi R_2^2 \leq \frac{k}{l} < \pi R_1^2$,
and consider the commutative diagram
\label{page diagram}
\[
\begin{tikzcd}
G_{\mathbb{Z}_k, \ast}^{(k, \infty]} \big(\widehat{B^{2n}(R_3)}\big)
\arrow{d}{\cong} \arrow{rr} &
& G_{\mathbb{Z}_k, \ast}^{(k, \infty]}
\big(\widehat{B^{2n}(R_1)}\big)
\arrow{d}{\cong} \\
G_{\mathbb{Z}_k, \ast}^{(k, \infty]} \big(\widehat{B^{2n}(R_3)}\big)
\arrow{r}
& G_{\mathbb{Z}_k, \ast}^{(k, \infty]} \big(\widehat{B^{2n}(R_2)}\big)
\arrow{r}
& G_{\mathbb{Z}_k, \ast}^{(k, \infty]}
\big(\psi(\widehat{B^{2n}(R_1)}) \big)
\end{tikzcd}
\]
where the horizontal arrows
are the homomorphisms induced by the inclusions
and the vertical arrows the isomorphisms induced by $\psi$.
This gives a contradiction,
because for $\ast = 2nl$
the horizontal arrow on the top
is an isomorphism from $\mathbb{Z}_k$ to $\mathbb{Z}_k$,
while $G_{\mathbb{Z}_k, \ast}^{(k, \infty]} \big(\widehat{B^{2n}(R_2)}\big) \cong 0$.

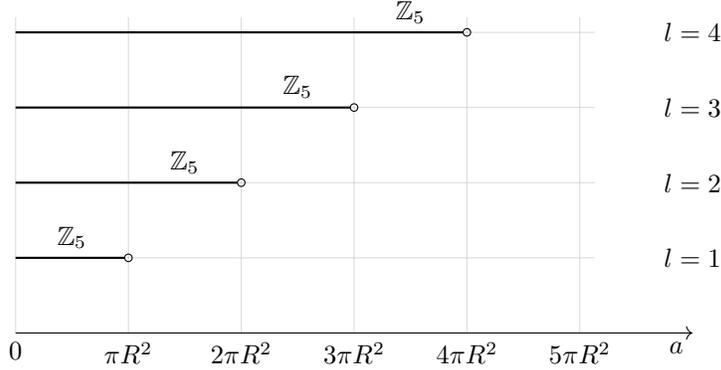
\begin{figure}[htbp]
    \centering
\begin{tikzpicture}

\draw[help lines, gray!30](0,0) grid[step={($(3/2, 1) - (0, 0)$)}] (7.7,4.2);

\draw[->]    (0,0)--(9,0) node[anchor=north east]  {$a$} ;

\node[anchor=north] at (0,0)   {$0$} ;
 \node[anchor=north] at (3/2,0)   {$\pi R^2$} ;
\foreach \i in {2,3,4,5} 
{ \node[anchor=north] at (3*\i/2,0)   {$\i \pi R^2$} ;}

\foreach \i in {1,2,3,4}
{ \node[anchor=east] at (9.5,\i)   {$l={\i} $} ;}

\foreach \i in {1,2,3,4} 
{ \draw[thick]  (0,\i)--({3*\i/2-0.05},\i)     ;}
\foreach \i in {1,2,3,4} 
{ \draw ({3*\i/2},\i) circle (0.05);}

\foreach \i in {1,2,3,4}
{ \node[anchor=south] at ({3*\i/2-3/4},\i)   {$\mathbb{Z}_5$} ;}

\end{tikzpicture}

\caption{The homology groups $G_{\mathbb{Z}_5,2nl}^{(a,\infty]}(B^{2n}(R))$ for $l<5$.}
\label{fig: equivariant barcode}
\end{figure}

Note that this argument
applies only when $1 < \pi R_1^2$
due to the condition $l < k$
in the stated results on equivariant homology of balls.
As we will see,
these results are obtained
by considering a certain sequence of Hamiltonian functions
supported in $B^{2n} (R)$
(the Hamiltonian function of a rotation
composed with suitable cut-off functions,
as in \cite{Traynor, San11}),
and using the fact that,
for the time-1 map of the flow of these Hamiltonian functions,
the spaces of $k$-periodic points
corresponding to the critical values $l \pi R^2$
are diffeomorphic to spheres $S^{2n-1}$
and the $\mathbb{Z}_k$-action on them
is given by $z \mapsto e^{-i \frac{2\pi l}{k}} z$;
in particular the action is free for $l < k$,
but not free for $l = k$.

In \cite{Zhang}
the equivariant symplectic and contact homologies of domains
of $(\mathbb{R}^{2n}, \omega_0)$
and $(\mathbb{R}^{2n} \times S^1, \xi_0)$
defined in \cite{Chiu}
with the microlocal theory of sheaves
are used to define
sequences $(c_j)_{j \in \mathbb{Z}_{\geq 1}}$
of symplectic and contact capacities.
Moreover,
a construction of capacities for lens spaces,
relevant also in the present context,
is sketched
in \cite{Fraser - persistence}.
The possibility of using
the equivariant theory introduced in the present article
to define sequences of capacities
generalizing the symplectic and contact capacities
defined in \cite{Viterbo} and \cite{San11} respectively,
as well as the properties and applications
of such capacities,
will be explored in forthcoming works.

The article is organized as follows.
In \autoref{section: generating functions}
we give some preliminaries on generating functions.
In \autoref{section: invariant generating functions symplectic}
we use a composition formula due to Allais \cite{Allais}
to define $\mathbb{Z}_k$-invariant functions
detecting $k$-periodic points of Hamiltonian symplectomorphisms
of $(\mathbb{R}^{2n}, \omega_0)$,
and in \autoref{section: equivariant symplectic homology}
we use these functions to define
the $\mathbb{Z}_k$-equivariant homology
of domains of $(\mathbb{R}^{2n}, \omega_0)$.
In \autoref{section: invariant generating functions contact}
we introduce a contact version
of the composition formula of
\autoref{section: invariant generating functions symplectic}
to define $\mathbb{Z}_k$-invariant functions
detecting translated $k$-chains
of contactomorphisms of $(\mathbb{R}^{2n} \times S^1, \xi_0)$,
and in \autoref{section: equivariant contact homology}
we use these functions to define
the $\mathbb{Z}_k$-equivariant homology
of domains of $(\mathbb{R}^{2n} \times S^1, \xi_0)$.
In \autoref{section: relation symplectic contact}
we discuss the relation
between the $\mathbb{Z}_k$-equivariant homology
of a domain of $(\mathbb{R}^{2n}, \omega_0)$
and the $\mathbb{Z}_k$-equivariant homology
of its prequantization,
and in \autoref{section: balls}
we calculate the $\mathbb{Z}_k$-equivariant homology of balls,
concluding the proof of \autoref{theorem: main}.

\subsection*{Acknowledgments} 
The second author addresses a special thanks
to Miguel Abreu for his constant support and advice.
She also thanks Simon Allais for helpful discussions,
and more specifically for correcting a mistake in \autoref{proposition: index}
and suggesting an easier proof of \autoref{proposition: calculation phi equivariant}
in the revised version.
The third author warmly thanks Shaoyun Bai, Sylvain Courte, St\'ephane Guillermou, Wenyuan Li and Vivek Shende for helpful discussions.
Some of the work was conducted when the third author visited IRMA (UMR 7501) at the Université de Strasbourg during the winter of 2021.
Additionally, the third author participated in the conference
{\em From smooth to $\mathcal{C}^0$ symplectic geometry: topological aspects and dynamical implications} in the summer of 2023,
which took place at CIRM, where further work was completed.
The third author hereby thanks Université de Strasbourg, CIRM, and the event organizers for their hospitality during the visits.
We thank the referees for useful comments and corrections.

The first author was supported by the Natural Sciences and Engineering Research Council of Canada (NSERC RGPIN-2017-06901).
The second and third authors were supported
by the ANR project COSY (ANR-21-CE40-0002).
The third author was also supported
by the Novo Nordisk Foundation grant NNF20OC0066298
and the VILLUM FONDEN, VILLUM Investigator grant 37814.


\section{Generating functions}\label{section: generating functions}

In this section we gather the results on generating functions
that are needed in the rest of the article.
As explained in the introduction
(\autoref{remark: category theoretic construction}),
in contrast to \cite{Traynor, San11}
where the generating function homology of domains
is constructed by just using generating functions
associated to single Hamiltonian symplectomorphisms and contactomorphisms,
in the present article we work
with Hamiltonian and contact isotopies.
We therefore need to add
Propositions \ref{proposition: Serre fibration}
and \ref{proposition: Serre fibration contact}
and Lemmas \ref{lemma: lemmas gf}, \ref{lemma: lemmas gf contact},
\ref{lemma: special} and \ref{lemma: monotonicity special}
to the material presented in \cite{Traynor, San11}.
For the rest we follow \cite{Traynor, San11},
to which we refer for more details.

A function $F: E \rightarrow \mathbb{R}$
defined on the total space $E = B \times \mathbb{R}^N$
of a trivial vector bundle
$\pi: B \times \mathbb{R}^N \rightarrow B$
is a \emph{generating function}
if the differential $dF: E \rightarrow T^{\ast}E$
is transverse to the fibre conormal bundle
\[
N_E^{\ast} =
\big\{\, \sigma_e \in T^{\ast}E \;\big\lvert\;
\sigma_e = 0 \text{ on } \ker d \pi (e) \,\big\} \,.
\]
The space
\[
\Sigma_F = dF^{-1} \big( N_E^{\ast} \cap \im (dF) \big)
\]
of fibre critical points
is then a submanifold of $E$,
of dimension equal to the dimension of $B$.
Consider the map $i_F: \Sigma_F \rightarrow T^{\ast}B$
that associates to $e \in \Sigma_F$
the covector $i_F(e) \in T_{\pi(e)}^{\ast}B$
defined by
\[
i_F(e) (X) = dF (\widehat{X})
\]
for $X \in T_{\pi(e)}B$,
where $\widehat{X}$ is any vector in $T_eE$
with $d \pi (\widehat{X}) = X$.
Equip $T^{\ast}B$ with its canonical symplectic structure
$\omega_{\can} = d \lambda_{\can}$.
Then $i_F: \Sigma_F \rightarrow T^{\ast}B$
is an exact Lagrangian immersion,
with
\begin{equation}\label{equation: iF exact}
i_F^{\,\ast} \, \lambda_{\can} = d ( \left. F \right\lvert_{\Sigma_F} ) \,.
\end{equation}
If $i_F: \Sigma_F \rightarrow T^{\ast}B$ is an embedding
we say that $F$ is a generating function
of the Lagrangian submanifold $\im (i_F)$
of $(T^{\ast}B, \omega_{\can})$.
The map $i_F$ then induces a bijection
between the critical points of $F$
and the intersections of $\im (i_F)$ with the zero section.

Suppose now that $B$ is compact.
A generating function
$F: E = B \times \mathbb{R}^N \rightarrow \mathbb{R}$
is said to be \emph{quadratic at infinity}
if there exists a non-degenerate quadratic form $F_{\infty}$ on $E$,
i.e.\ a map $F_{\infty}: E \rightarrow \mathbb{R}$
whose restriction to every fibre
is a non-degenerate quadratic form,
such that
$\partial_v (F - F_{\infty}): E \rightarrow E^{\ast}$
is bounded,
where $\partial_v$ denotes the vertical derivative.
Every Lagrangian submanifold of $(T^{\ast}B, \omega_{\can})$
Hamiltonian isotopic to the zero section
has a generating function quadratic at infinity
\cite{Sikorav87},
which is unique up to addition of a constant,
fibre preserving diffeomorphism
and stabilization,
i.e.\ replacing $F: B \times \mathbb{R}^N \rightarrow \mathbb{R}$
by $F \oplus Q: B \times \mathbb{R}^N \times \mathbb{R}^{N'} \rightarrow \mathbb{R}$
for a non-degenerate quadratic form $Q$ on $\mathbb{R}^{N'}$
\cite{Viterbo, Theret}.
The existence theorem is a special case
of the following result \cite{Sikorav87}:
if $L$ is a Lagrangian submanifold
of $(T^{\ast}B, \omega_{\can})$
that has a generating function quadratic at infinity $F$
then for any Hamiltonian isotopy $\{\varphi_t\}_{t \in [0, 1]}$
of $(T^{\ast}B, \omega_{\can})$
there is a 1-parameter family $F_t$
of generating functions quadratic at infinity
for $\{\varphi_t(L)\}$
such that $F_0$ is a stabilization of $F$.
In turn,
this implies the following
more general result \cite{Theret}.

\begin{prop}\label{proposition: Serre fibration}
Let $\mathcal{F}$ be the space
of generating functions quadratic at infinity
over a compact manifold $B$,
and $\Lag$ the space of Lagrangian submanifolds
of $(T^{\ast}B, \omega_{\can})$.
Then the map $\mathcal{F} \rightarrow \Lag$
that sends a generating function
to the generated Lagrangian submanifold
is a Serre fibration up to stabilization:
given a map $f: \Delta_d \rightarrow \Lag$,
where $\Delta_d$ denotes the standard $d$-simplex,
a lift $F: \Delta_d \rightarrow \mathcal{F}$
and a homotopy $f_t: \Delta_d \rightarrow \Lag$,
$t \in [0, 1]$, of $f_0 = f$,
there is a homotopy $F_t: \Delta_d \rightarrow \mathcal{F}$
lifting $f_t$ such that $F_0$ is a stabilization of $F$.
\end{prop}

Fix a point $p$ of $B$,
and denote by $0_p$ the corresponding point of the zero section of $T^{\ast}B$.
We say that a generating function $F$
of a Lagrangian submanifold $L$ of $(T^{\ast}B, \omega_{\can})$
with $0_p \in L$ is normalized (with respect to $p$)
if $F \big( i_F^{\,-1} (0_p) \big) = 0$.

\begin{lemma}\label{lemma: lemmas gf}
Let $B$ be a compact manifold,
and fix a point $p$ of $B$.
Then we have the following results:
\begin{itemize}
\item[(i)]
Let $L$ be a Lagrangian submanifold of $(T^{\ast}B, \omega_{\can})$
with $0_p \in L$.
If $F^{(s)}$, $s \in [0, 1]$, is a 1-parameter family
of normalized generating functions quadratic at infinity of $L$
then there is a 1-parameter family $\Phi_s$
of fibre preserving diffeomorphisms
with $\Phi_0 = \id$
and $F^{(s)} \circ \Phi_s = F^{(0)}$.
\item[(ii)]
Let $\{L_t\}_{t \in [0, 1]}$
be a Lagrangian isotopy 
in $(T^{\ast}B, \omega_{\can})$ with $0_p \in L_t$ for all $t$.
Suppose that $F^{(s)}_t$, $s \in [0, 1]$,
is a 2-parameter family
of normalized generating functions quadratic at infinity
such that every $F^{(s)}_t$
is a generating function of $L_t$.
Then there is a 2-parameter family $\Phi_{s, t}$
of fibre preserving diffeomorphisms
with $\Phi_{0, t} = \id$
and $F^{(s)}_t \circ \Phi_{s, t} = F^{(0)}_t$.

\item[(iii)]
Let $\{L_t\}_{t \in [0, 1]}$ be a Lagrangian isotopy 
in $(T^{\ast}B, \omega_{\can})$ with $0_p \in L_t$ for all $t$.
If $F_t^{(0)}$ and $F_t^{(1)}$ are 1-parameter families
of normalized generating functions quadratic at infinity
for $\{L_t\}$ with $F_0^{(0)} = F_0^{(1)}$
then there are a non-degenerate quadratic form $Q$
and a 2-parameter family $\Phi_{s, t}$
of fibre preserving diffeomorphisms
such that $\Phi_{0, t} = \id$,
$(F_t^{(1)} \oplus Q) \circ \Phi_{1, t} = F_t^{(0)} \oplus \,Q$,
and $(F_0^{(0)} \oplus \,Q) \circ \Phi_{s, 0}^{\,-1}$ is a contractible loop.
Similarly for 2-parameter families
$F_{u, t}^{(0)}$ and $F_{u, t}^{(1)}$
of normalized generating functions quadratic at infinity
for the same 2-parameter family $\{L_{u, t}\}$
of Lagrangian submanifolds with $0_p \in L_{u, t}$
for all $u, t$.
\end{itemize}
\end{lemma}

\begin{proof}
Point (i) is one of the ingredients
of the uniqueness theorem of \cite{Viterbo, Theret}.
Point (ii) is a 1-parameter version of (i),
which can be proved by the same argument
(see also \cite[Lemma 2.16]{lens},
where the proof is given
in the context of conical generating functions).
Point (iii) can be seen as follows.
For every $t$,
consider the path of Lagrangian submanifolds
(which is actually a loop)
obtained by going from $L_t$ to $L_0$
along $\{L_t\}$
and then back to $L_t$.
This 1-parameter family $\gamma$ of paths
has a lift $\Gamma$
to the space of generating functions $\mathcal{F}$:
for every $t$ we consider the path in $\mathcal{F}$
that goes from $F_t^{(0)}$ to $F_0^{(0)}$
and then from $F_0^{(0)} = F_0^{(1)}$ to $F_t^{(1)}$.
By \autoref{proposition: Serre fibration}
the homotopy from the 1-parameter family of loops $\gamma$
to the constant 1-parameter family of loops along $\{L_t\}$
(obtained by contracting for every $t$
the loop based at $L_t$ to the constant one)
can be lifted to a homotopy in $\mathcal{F}$
starting at a stabilization $\Gamma \oplus Q$ of $\Gamma$.
In particular,
we obtain a 2-parameter family
of generating functions quadratic at infinity $G_t^{(s)}$
from $G_t^{(0)} = F_t^{(0)} \oplus Q$ to $G_t^{(1)} = F_t^{(1)} \oplus Q$
such that $G_0^{(s)}$ is a contractible loop
and every $G_t^{(s)}$ is a generating function of $L_t$.
By hypothesis, $F_t^{(0)}$ and $F_t^{(1)}$,
hence $G_t^{(0)}$ and $G_t^{(1)}$,
are normalized for all $t$.
After adding a constant $c = c(s, t)$ (with $c(0, t) = c(1, t) = 0$),
we can thus assume that all the $G_t^{(s)}$ are normalized.
We then apply (ii) to obtain a 2-parameter family
of fibre preserving diffeomorphisms $\Phi_{s, t}$
with $\Phi_{0, t} = \id$
and $G_t^{(s)} \circ \Phi_{s, t} = G_t^{(0)}$,
in particular $(F_t^{(1)} \oplus Q) \circ \Phi_{1, t} = F_t^{(0)} \oplus Q$
and $(F_0^{(0)} \oplus Q) \circ \Phi_{s, 0}^{\,-1} = G_0^{(s)}$ is a contractible loop.
The 2-parameter case can be proved similarly.
\end{proof}

Consider now the standard symplectic Euclidean space
$( \mathbb{R}^{2n} ,
\omega_0 = \sum_{j=1}^n dx_j \wedge dy_j)$,
and the product
\[
\big( \mathbb{R}^{2n} \times \mathbb{R}^{2n} \,,\,
- \omega_0 \oplus \omega_0\big) \,.
\]
The map
\[
\tau: \mathbb{R}^{2n}\times\mathbb{R}^{2n} \rightarrow T^{\ast}\mathbb{R}^{2n} \,,\;
\tau(x, y, X, Y) = \Big(\frac{x+X}{2}, \frac{y+Y}{2}, y - Y, X - x \Big)
\]
is a symplectomorphism,
and sends the diagonal to the zero section.
For a symplectomorphism $\varphi$
of $(\mathbb{R}^{2n} , \omega_0)$
we denote by
\[
\Gamma_{\varphi}:
\mathbb{R}^{2n} \rightarrow T^{\ast}\mathbb{R}^{2n}
\]
the composition of the graph
\[
\gr (\varphi):\mathbb{R}^{2n} \rightarrow
\mathbb{R}^{2n} \times \mathbb{R}^{2n} \,,\;
p \mapsto \big( p, \varphi(p) \big)
\]
with $\tau$.

\begin{rmk}\label{remark: varphi exact}
Let $\lambda_0 = \sum_{j = 1}^n \frac{x_j dy_j - y_j dx_j}{2}$.
If $\varphi^{\ast} \lambda_0 - \lambda_0 = dS$
then
\[
\Gamma_{\varphi}^{\;\ast} \, \lambda_{\can}
= d \Big( S + \frac{y \, \varphi_x - x \,\varphi_y}{2} \Big) \,,
\]
where we denote
$\varphi (x, y) = \big( \varphi_x (x, y), \varphi_y (x, y) \big)$.
\end{rmk}

We say that $F$ is a generating function
of a symplectomorphism $\varphi$
of $(\mathbb{R}^{2n}, \omega_0)$
if it is a generating function
of the Lagrangian submanifold $L_{\varphi} = \im (\Gamma_{\varphi})$
of $(T^{\ast}\mathbb{R}^{2n}, \omega_{\can})$.
Then $i_F: \Sigma_F \rightarrow T^{\ast}\mathbb{R}^{2n}$
gives a diffeomorphism between $\Sigma_F$
and $L_{\varphi}$,
and the composition
\begin{equation}\label{equation: diffeomorphism gf symplectic}
\begin{tikzcd}
{\Phi_F:\,\mathbb{R}^{2n}} \arrow[r, "\Gamma_{\varphi}"] & L_{\varphi} \arrow[r] &[3ex] \Sigma_F 
\end{tikzcd}
\end{equation}
of the inverse of this diffeomorphism with $\Gamma_{\varphi}$
is a diffeomorphism
that induces a bijection
between the fixed points of $\varphi$
and the critical points of $F$.
The next proposition
follows directly from \eqref{equation: iF exact}
and \autoref{remark: varphi exact}.

\begin{prop}\label{proposition: gf symplectic}
If $F$ is a generating function
of a symplectomorphism $\varphi$ of $(\mathbb{R}^{2n}, \omega_0)$
and if $\varphi^{\ast} \lambda_0 - \lambda_0 = dS$
then
\[
d (F \circ \Phi_F)
= d \Big( S + \frac{y \, \varphi_x - x \,\varphi_y}{2} \Big) \,.
\]
\end{prop}

We say that a generating function
$F: \mathbb{R}^{2n} \times \mathbb{R}^N \rightarrow \mathbb{R}$
of a compactly supported symplectomorphism $\varphi$
of $(\mathbb{R}^{2n} , \omega_0)$
is normalized if it is zero at the critical points
corresponding to the fixed points outside the support of $\varphi$.
If $\varphi$ is a compactly supported Hamiltonian symplectomorphism,
the critical values of such generating function
are then equal to the symplectic action
of the fixed points of $\varphi$.
Recall that the symplectic action
of a fixed point $p$
is defined by
\[
\mathcal{A}_{\varphi}(p)
= \int_0^1 \big( \lambda_0 (X_t) + H_t \big)
\big( \varphi_t(p) \big) \, dt \,,
\]
where $\{ \varphi_t \}_{t \in [0,1]}$
is any compactly supported Hamiltonian isotopy with $\varphi_1 = \varphi$,
$X_t$ is the vector field generating this isotopy,
and $H_t$ is the associated compactly supported Hamiltonian function,
with the sign convention
$\iota_{X_t}\omega_0 = dH_t$.
Recall also that $\mathcal{A}_{\varphi}(p) = S (p)$,
where $S: \mathbb{R}^{2n} \rightarrow \mathbb{R}$
is the compactly supported function satisfying
$\varphi^{\ast} \lambda_0 - \lambda_0 = dS$.

\begin{prop}\label{proposition: critical values gf action}
Let $F: \mathbb{R}^{2n} \times \mathbb{R}^N \rightarrow \mathbb{R}$
be a normalized generating function
of a compactly supported Hamiltonian symplectomorphism $\varphi$
of $(\mathbb{R}^{2n}, \omega_0)$,
and let $S: \mathbb{R}^{2n} \rightarrow \mathbb{R}$
be the compactly supported function satisfying
$\varphi^{\ast} \lambda_0 - \lambda_0 = dS$.
Let
\[
\Sigma_F \rightarrow \mathbb{R}^{2n} \,,\;
(x, y, \zeta) \mapsto (\overline{x}, \overline{y})
\]
be the inverse of the diffeomorphism
\eqref{equation: diffeomorphism gf symplectic}.
For every $(x, y, \zeta) \in \Sigma_F$
we then have
\[
F (x, y, \zeta) = S (\overline{x}, \overline{y})
+ \frac{\overline{y} \varphi_x (\overline{x}, \overline{y})
- \overline{x} \varphi_y (\overline{x}, \overline{y})}{2} \,.
\]
If moreover $(x, y, \zeta)$ is a critical point of $F$
then
\[
F (x, y, \zeta) = S (\overline{x}, \overline{y})
= \mathcal{A}_{\varphi} (\overline{x}, \overline{y}) \,.
\]
\end{prop}

\begin{proof}
Since $S$ is compactly supported
and $F$ is normalized,
the first statement follows
from \autoref{proposition: gf symplectic}.
For the second statement,
observe that if $(x, y, \zeta)$ is a critical point of $F$
then $(\overline{x}, \overline{y})$
is a fixed point of $\varphi$,
and so $\overline{y} \varphi_x (\overline{x}, \overline{y})
- \overline{x} \varphi_y (\overline{x}, \overline{y}) = 0$.
\end{proof}

If $\varphi$ is a Hamiltonian symplectomorphism
of $(\mathbb{R}^{2n} , \omega_0)$
then the Lagrangian submanifold $L_{\varphi}$
of $( T^{\ast}\mathbb{R}^{2n}, \omega_{\can})$
is Hamiltonian isotopic to the zero section.
If moreover $\varphi$
is a compactly supported Hamiltonian symplectomorphism then,
by considering the 1-point compactification $S^{2n}$ of $\mathbb{R}^{2n}$,
$L_{\varphi}$ extends to a Lagrangian submanifold $\overline{L_{\varphi}}$
of $( T^{\ast}S^{2n}, \omega_{\can})$,
Hamiltonian isotopic to the zero section.
We identify $\mathbb{R}^{2n}$
with $S^{2n} \smallsetminus \{p_{\infty}\}$
by the stereographic projection,
where $p_{\infty}$ denotes the point at infinity
when seeing $S^{2n}$
as the 1-point compactification of $\mathbb{R}^{2n}$.
If $\overline{F}: S^{2n} \times \mathbb{R}^N \rightarrow \mathbb{R}$
is a generating function of $\overline{L_{\varphi}}$
then the induced function
$F: \mathbb{R}^{2n} \times \mathbb{R}^N \rightarrow \mathbb{R}$
is a generating function of $L_{\varphi}$.
If $\overline{F}$ is quadratic at infinity
then we say that $F$
is a generating function quadratic at infinity of $\varphi$,
and denote by $F_{\infty}$
the quadratic form on $\mathbb{R}^{2n} \times \mathbb{R}^N$
induced by $\overline{F}_{\infty}$.
Notice that $F$ is normalized if and only if
$\overline{F}$ is normalized with respect to $p_{\infty}$.

Let now $J^1B = T^{\ast}B \times \mathbb{R}$
be the 1-jet bundle of a manifold $B$,
endowed with its canonical contact structure
$\xi_{\can} = \ker (d\theta - \lambda_{\can})$,
where $\theta$ denotes the coordinate in $\mathbb{R}$. 
If $F: E \rightarrow \mathbb{R}$,
defined on the total space $E = B \times \mathbb{R}^N$
of a trivial vector bundle $\pi: B \times \mathbb{R}^N \rightarrow B$,
is a generating function
then the map
\[
\widetilde{i_F}: \Sigma_F \rightarrow J^1B \,,\;
\widetilde{i_F} (e) = \big( i_F(e) \,,\, F(e) \big)
\]
is a Legendrian immersion.
If it is an embedding
we say that $F$ is a generating function
of the Legendrian submanifold $\im (\widetilde{i_F})$
of $(J^1B, \xi_{\can})$.
The map $\widetilde{i_F}$ then induces a bijection
between the critical points of $F$
and the Reeb chords
from the zero section to $\im (\widetilde{i_F})$:
a fibre critical point $e$ of $F$
is a critical point
if and only if $\widetilde{i_F} (e) = \big( i_F(e) , F(e) \big)$
belongs to the zero wall of $J^1B$
(the product of the zero section of $T^{\ast}B$ with $\mathbb{R}$),
and so $\widetilde{i_F} (e)$ belongs to the same Reeb orbit
as the point $\big( i_F(e), 0 \big)$
of the zero section;
in this case moreover the length of the Reeb chord
from $\big( i_F(e), 0 \big)$ to $\widetilde{i_F} (e)$
is equal to the critical value $F(e)$.

If $B$ is compact
then every Legendrian submanifold of $(J^1B, \xi_{\can})$
contact isotopic to the zero section
has a generating function quadratic at infinity
\cite{Chaperon, Chekanov},
unique up to fibre preserving diffeomorphism and stabilization \cite{Theret_thesis}.
Similarly to the symplectic case,
the existence theorem is a special case
of the following result
\cite{Chaperon, Chekanov}:
if $\Lambda$ is a Legendrian submanifold
of $(J^1B, \xi_{\can})$
that has a generating function quadratic at infinity $F$
then for any contact isotopy $\{\phi_t\}_{t \in [0, 1]}$
of $(J^1B, \xi_{\can})$
there is a 1-parameter family $F_t$
of generating functions quadratic at infinity
for $\{\phi_t (\Lambda)\}$
such that $F_0$ is a stabilization of $F$.
Moreover,
the analogue of \autoref{proposition: Serre fibration}
also holds,
with a similar proof \cite{Theret_thesis}.

\begin{prop}\label{proposition: Serre fibration contact}
Let $\Leg$ be the space of Legendrian submanifolds
of the 1-jet bundle $(J^1B, \xi_{\can})$
of a compact manifold $B$.
Then the map $\mathcal{F} \rightarrow \Leg$
that sends a generating function
to the generated Legendrian submanifold
is a Serre fibration up to stabilization:
given a map $f: \Delta_d \rightarrow \Leg$,
a lift $F: \Delta_d \rightarrow \mathcal{F}$
and a homotopy $f_t: \Delta_d \rightarrow \Leg$,
$t \in [0, 1]$, of $f_0 = f$,
there is a homotopy $F_t: \Delta_d \rightarrow \mathcal{F}$
lifting $f_t$ such that $F_0$ is a stabilization of $F$.
\end{prop}

Similarly,
the analogue of \autoref{lemma: lemmas gf}
still holds
(point (i) is one of the ingredients
of the uniqueness theorem of \cite{Theret_thesis},
point (ii) is a 1-parameter version of (i),
and point (iii) can be deduced from (ii)
and \autoref{proposition: Serre fibration contact}
by an argument as in the proof
of \autoref{lemma: lemmas gf} (iii)).

\begin{lemma}\label{lemma: lemmas gf contact}
Let $B$ be a compact manifold.
Then we have the following results:
\begin{enumerate}
\item[(i)]
If $F^{(s)}$, $s \in [0, 1]$, is a 1-parameter family
of generating functions quadratic at infinity
of the same Legendrian submanifold
of $(J^1B, \xi_{\can})$
then there is a 1-parameter family $\Phi_s$
of fibre preserving diffeomorphisms
with $\Phi_0 = \id$
and $F^{(s)} \circ \Phi_s = F^{(0)}$.
\item[(ii)] 
Let $\{\Lambda_t\}_{t \in [0, 1]}$
be a Legendrian isotopy
in $(J^1B, \xi_{\can})$.
Suppose that $F^{(s)}_t$, $s \in [0, 1]$,
is a 2-parameter family
of generating functions quadratic at infinity
such that every $F^{(s)}_t$
is a generating function of $\Lambda_t$.
Then there is a 2-parameter family $\Phi_{s, t}$
of fibre preserving diffeomorphisms
with $\Phi_{0, t} = \id$
and $F^{(s)}_t \circ \Phi_{s, t} = F^{(0)}_t$.
\item[(iii)]
If $F_t^{(0)}$ and $F_t^{(1)}$ are 1-parameter families
of generating functions quadratic at infinity
for the same Legendrian isotopy $\{\Lambda_t\}$
in $(J^1B, \xi_{\can})$
with $F_0^{(0)} = F_0^{(1)}$
then there are a non-degenerate quadratic form $Q$
and a 2-parameter family $\Phi_{s,t}$
of fibre preserving diffeomorphisms
such that $\Phi_{0, t} = \id$,
$(F_t^{(1)} \oplus Q) \circ \Phi_{1, t} = F_t^{(0)} \oplus Q$,
and $(F_0^{(0)} \oplus Q) \circ \Phi_{s, 0}^{\,-1}$
is a contractible loop.
Similarly for 2-parameter families
$F_{u, t}^{(0)}$ and $F_{u, t}^{(1)}$
of generating functions quadratic at infinity
for the same 2-parameter family $\{\Lambda_{u, t}\}$
of Legendrian submanifolds.
\end{enumerate}
\end{lemma}

Consider now the standard contact Euclidean space
$\big(\mathbb{R}^{2n+1}, \xi_0 = \ker (\alpha_0)\big)$,
with coordinates
$(x_1, \dots, x_n, y_1, \dots, y_n, \theta)$
and contact form
\[
\alpha_0 = d\theta - \sum_{j=1}^n \frac{x_j dy_j - y_j dx_j}{2} \,,
\]
and the contact product
\[
\big(\mathbb{R}^{2n+1} \times \mathbb{R}^{2n+1} \times \mathbb{R}
\,,\, \ker ( \pi_2^{\ast} \alpha_0
- e^{\rho} \pi_1^{\ast} \alpha_0 ) \big)\,,
\]
where $\rho$ denotes the coordinate in $\mathbb{R}$
and $\pi_1$ and $\pi_2$ denote the projections
on the first and second factors respectively.
The map
\[
\underline{\tau}: \mathbb{R}^{2n+1} \times \mathbb{R}^{2n+1} \times \mathbb{R}
\longrightarrow J^1\mathbb{R}^{2n+1}
\]
defined by 
\begin{gather*}
    \underline{\tau} \, (x, y, \theta, X, Y, \Theta, \rho) =
    \\
\Big(\frac{e^{\frac{\rho}{2}} x + X}{2} \,,\,
\frac{e^{\frac{\rho}{2}} y + Y}{2}
\,,\, \theta \,,
e^{\frac{\rho}{2}} y - Y \,,\,
X - e^{\frac{\rho}{2}} x \,,\,
e^{\rho} - 1 \,,\,
\Theta - \theta
+ \frac{e^{\frac{\rho}{2}} (yX - xY)}{2}\Big)
\end{gather*}
is a contactomorphism,
and sends the Legendrian diagonal
\[
\Delta = \{\, (x, y, \theta, x, y, \theta, 0) \,\} \subset \mathbb{R}^{2n+1} \times \mathbb{R}^{2n+1} \times \mathbb{R}
\]
to the zero section.
This contactomorphism is moreover strict
with respect to the contact forms
that we are considering:
the pullback by $\underline{\tau}$
of the canonical contact form
on $J^1\mathbb{R}^{2n+1}$
is equal to the contact form
$\pi_2^{\ast}\alpha_0 - e^{\rho} \pi_1^{\ast} \alpha_0$
on $\mathbb{R}^{2n+1} \times \mathbb{R}^{2n+1} \times \mathbb{R}$.
For a contactomorphism $\phi$
of $(\mathbb{R}^{2n+1}, \xi_0)$
with $\phi^{\ast} \alpha_0 = e^g \alpha_0$
we denote by
\[
\Gamma_{\phi}: \mathbb{R}^{2n+1} \rightarrow
J^1 \mathbb{R}^{2n+1}
\]
the composition of the Legendrian graph
\[
\gr(\phi): \mathbb{R}^{2n+1} \rightarrow
\mathbb{R}^{2n+1} \times \mathbb{R}^{2n+1} \times \mathbb{R} \,,\;
p \mapsto \big( p, \phi(p), g(p) \big)
\]
with $\underline{\tau}$.
We say that $F$
is a generating function of $\phi$
if it is a generating function
of the Legendrian submanifold $\Lambda_{\phi} = \im (\Gamma_{\phi})$
of $( J^1\mathbb{R}^{2n+1}, \xi_{\can} )$.
Then $\widetilde{i_F}: \Sigma_F \rightarrow J^1\mathbb{R}^{2n+1}$
gives a diffeomorphism between $\Sigma_F$
and $\Lambda_{\phi}$,
and the composition
\begin{equation}\label{equation: diffeomorphism gf contact}
\begin{tikzcd}
{\Psi_F:\,\mathbb{R}^{2n+1}} \arrow[r, "\Gamma_{\phi}"] & \Lambda_{\phi} \arrow[r] &[3ex] \Sigma_F 
\end{tikzcd}
\end{equation}
of the inverse of this diffeomorphism with $\Gamma_{\phi}$
is a diffeomorphism
that induces a bijection
between the translated points of $\phi$
and the critical points of $F$.
Moreover,
the critical value $F \big(\Psi_F (p)\big)$
of the critical point $\Psi_F (p)$
corresponding to a translated point $p$
is equal to the contact action of $p$.

\begin{rmk}\label{remark: gf Reeb}
Let $F: \mathbb{R}^{2n+1} \times \mathbb{R}^N \rightarrow \mathbb{R}$
be a generating function of a contactomorphism $\phi$
of $(\mathbb{R}^{2n+1}, \xi_0)$.
For any $t \in \mathbb{R}$
the function
\[
F_t: \mathbb{R}^{2n+1} \times \mathbb{R}^N \rightarrow \mathbb{R} \,,\;
F_t (x, y, \theta, \zeta)
= F (x, y, \theta, \zeta) + t
\]
is then a generating function
of $\varphi_t^{\alpha_0} \circ \phi$,
where $\{\varphi_t^{\alpha_0}\}$
denotes the Reeb flow
\[
\varphi_t^{\alpha_0} (x, y, \theta)
= (x, y, \theta + t) \,.
\]
\end{rmk}

If $\phi$ is a contactomorphism of $(\mathbb{R}^{2n+1}, \xi_0)$
contact isotopic to the identity
then the Legendrian submanifold $\Lambda_{\phi}$
of $(J^1\mathbb{R}^{2n+1}, \xi_{\can})$
is contact isotopic to the zero section.
If moreover $\phi$ is the lift to $(\mathbb{R}^{2n+1}, \xi_0)$
of a compactly supported contactomorphism
of $(\mathbb{R}^{2n} \times S^1, \xi_0)$
contact isotopic to the identity
then $\Lambda_{\phi}$ extends
to a Legendrian submanifold $\underline{\Lambda}_{\phi}$
of $\big( J^1 (S^{2n} \times \mathbb{R}) , \xi_{\can} \big)$,
which descends to a Legendrian submanifold $\overline{\Lambda_{\phi}}$
of $\big( J^1 (S^{2n} \times S^1) , \xi_{\can} \big)$,
still contact isotopic to the zero section.
If $\overline{F}: (S^{2n} \times S^1) \times \mathbb{R}^N \rightarrow \mathbb{R}$
is a generating function of $\overline{\Lambda_{\phi}}$
then the induced functions $F$ on $\mathbb{R}^{2n + 1} \times \mathbb{R}^N$
and $\underline{F}$ on $(S^{2n} \times \mathbb{R}) \times \mathbb{R}^N$
are generating functions of $\Lambda_{\phi}$
and $\underline{\Lambda}_{\phi}$ respectively,
and are invariant by the action of $\mathbb{Z}$
on $\mathbb{R}^{2n + 1} \times \mathbb{R}^N$
generated by the map
\begin{equation}\label{equation: action Z for F}
(x, y, \theta, \zeta) \mapsto (x, y, \theta + 1, \zeta)
\end{equation}
and the induced action on $S^{2n} \times \mathbb{R} \times \mathbb{R}^N$.
If $\overline{F}$ is quadratic at infinity then
we say that $F$ is a generating function quadratic at infinity
of $\phi$ and of the associated contactomorphism
of $(\mathbb{R}^{2n} \times S^1, \xi_0)$,
and denote by $F_{\infty}$ the quadratic form
on $\mathbb{R}^{2n + 1} \times \mathbb{R}^N$
induced by $\overline{F}_{\infty}$.

Recall that the lift to $(\mathbb{R}^{2n+1}, \xi_0)$
or $(\mathbb{R}^{2n} \times S^1, \xi_0)$
of a compactly supported Hamiltonian symplectomorphism
$\varphi$ of $(\mathbb{R}^{2n}, \omega_0)$
is the contactomorphism $\widetilde{\varphi}$ defined by
\[
\widetilde{\varphi} (x, y, \theta)
= \big( \varphi(x, y), \theta + S (x, y) \big) \,,
\]
where $S$ is the compactly supported function
that satisfies $\varphi^{\ast} \lambda_0 - \lambda_0 = dS$.
The following result can be proved as \cite[Lemma 3.2]{San11}.

\begin{lemma}\label{proposition: gf lift to contact}
If $f: \mathbb{R}^{2n} \times \mathbb{R}^N \rightarrow \mathbb{R}$
is a normalized generating function
of a compactly supported Hamiltonian symplectomorphism $\varphi$
of $(\mathbb{R}^{2n}, \omega_0)$
then
\[
F: \mathbb{R}^{2n+1} \times \mathbb{R}^N
\rightarrow \mathbb{R} \,,\;
F (x, y, \theta, \zeta)
= f (x, y, \zeta)
\]
is a generating function of the lift of $\varphi$
to $(\mathbb{R}^{2n+1}, \xi_0)$ or $(\mathbb{R}^{2n} \times S^1, \xi_0)$.
\end{lemma}

We say that a generating function quadratic at infinity $F$
of a compactly supported Hamiltonian symplectomorphism
of $(\mathbb{R}^{2n}, \omega_0)$
or of a compactly supported contactomorphism
of $(\mathbb{R}^{2n+1}, \xi_0)$ or $(\mathbb{R}^{2n} \times S^1, \xi_0)$
contact isotopic to the identity
is \emph{special}
if $F_{\infty}$ does not depend on the base variable
and $F = F_{\infty}$ outside a compact set.
In particular (in the case of Hamiltonian symplectomorphisms),
such $F$ is normalized.

\begin{lemma}\label{lemma: special}
Every family $\{\phi_t\}$,
for $t$ in a parameter space $[0, 1]^m$,
of compactly supported Hamiltonian symplectomorphisms of $(\mathbb{R}^{2n}, \omega_0)$
or of compactly supported contactomorphisms of $(\mathbb{R}^{2n} \times S^1, \xi_0)$
with $\phi_0 = \id$
has a family $F_t$ of special generating functions quadratic at infinity
such that $F_0$ is a non-degenerate quadratic form
and $(F_t)_{\infty}$ does not depend on $t$.
Moreover, if $\phi_{(t_1, \dots, t_{m-1}, 0)} = \id$
for all $(t_1, \dots, t_{m-1}) \in [0, 1]^{m-1}$
then there exists such family
with $F_{(t_1, \dots, t_{m-1}, 0)}$
a non-degenerate quadratic form
for all $(t_1, \dots, t_{m-1}, 0)$.
\end{lemma}

\begin{proof}
Since the identity of $(\mathbb{R}^{2n}, \omega_0)$
or of $(\mathbb{R}^{2n} \times S^1, \xi_0)$
has a special generating function quadratic at infinity
(any non-degenerate quadratic form
that does not depend on the base variable),
the result follows by induction on $m$
if we prove that if $\{\phi_s\}_{s \in [0, 1]^m}$
is a family of compactly supported
Hamiltonian symplectomorphisms or contactomorphisms
with $\phi_0 = \id$
having a family $F_s$ of special generating functions quadratic at infinity
then for every family $\{\psi_{s, t}\}$
for $s \in [0, 1]^m$ and $t \in [0, 1]$
with $\psi_{s, 0} = \id$,
$\{ \psi_{s, t} \circ \phi_s \}$
has a family $F_{s, t}$ of special generating functions quadratic at infinity
such that $F_{s, 0}$ is a stabilization of $F_s$
and $(F_{s, t})_{\infty}$ does not depend on $t$.
For this, we follow the construction used in \cite[Proposition 3.4]{San11},
which in turn is taken from \cite[Theorem 3]{Chaperon}
and \cite[Section III.2]{Theret_thesis} 
(see also \cite[Proposition 3.4]{Giroux}).
We only present the main lines of the argument,
and invite the reader to consult the above references
for more details.

Recall first that the \textit{transition function}
of a $\mathcal{C}^1$-small contactomorphism $\phi$
of $(J^1 \mathbb{R}^{2n+1}, \xi_{\can})$
is the function $G: J^1 \mathbb{R}^{2n+1} \rightarrow \mathbb{R}$
defined  by 
\[
G (q, p, z) = g_{p, z} (q) - f_{p, z} (q) \,,
\]
where, for every $p$ and $z$,
$f_{p, z} (q) = z + pq$
and $g_{p, z}$ is the function such that
$\im (j^1 g_{p, z}) = \varphi \big( \im (j^1 f_{p, z}) \big)$.
In particular,
$G = 0$ if and only if $\phi$ is the identity.
Moreover,
if $\Lambda$ is a Legendrian submanifold
of $(J^1 \mathbb{R}^{2n+1}, \xi_{\can})$
with generating function
$F: \mathbb{R}^{2n+1} \times \mathbb{R}^N \rightarrow \mathbb{R}$
then the function
\[
G \,\sharp\, F:
\mathbb{R}^{2n+1} \times
\big((\mathbb{R}^{2n+1})^{\ast} \times \mathbb{R}^{2n+1} \times \mathbb{R}^N\big)
\rightarrow \mathbb{R} \,,\;
\]
\[
(\underline{q}; p, q, z) \mapsto
G \big( \underline{q}, p, F (\underline{q} + q, \zeta) - p (\underline{q} + q) \big)
+ F (\underline{q} + q, \zeta) - pq
\]
is a generating function of $\phi (\Lambda)$.
Recall also that every contactomorphism $\phi$
of $(\mathbb{R}^{2n+1}, \xi_0)$
induces a contactomorphism $\Psi_{\phi}$
of $(J^1 \mathbb{R}^{2n+1}, \xi_{\can})$,
which is defined to be the conjugation by $\underline{\tau}$
of the contactomorphism
\[
(p, P, \rho) \mapsto \big( p, \phi (P), \rho + g (P) \big)
\]
of the contact product of $(\mathbb{R}^{2n+1}, \xi_0)$,
where $g$ is the conformal factor of $\phi$.
Then $\Lambda_{\phi} = \Psi_{\phi} (\Lambda_{\id})$
and $\Psi_{\phi \psi} = \Psi_{\phi} \circ \Psi_{\psi}$,
in particular $\Lambda_{\phi \psi} = \Psi_{\phi} (\Lambda_{\psi})$.
Notice moreover that if $\phi$ is the lift
of a compactly supported contactomorphism of $(\mathbb{R}^{2n} \times S^1, \xi_0)$
contact isotopic to the identity
then $\Psi_{\phi}$ is equivariant by translation by $1$ in the $z$-direction,
and if $\phi$ is the lift
of a compactly supported Hamiltonian symplectomorphism
of $(\mathbb{R}^{2n}, \omega_0)$
then $\Psi_{\phi}$ is equivariant by translation by $a$ in the $z$-direction
for every $a \in \mathbb{R}$.

Suppose now that $\{\phi_s\}_{s \in [0, 1]^m}$
is a family of compactly supported contactomorphisms
of $(\mathbb{R}^{2n} \times S^1, \xi_0)$
having a family of special generating functions quadratic at infinity
$F_s: \mathbb{R}^{2n+1} \times \mathbb{R}^N \rightarrow \mathbb{R}$,
and $\{\psi_{s, t}\}$ for $s \in [0, 1]^m$ and $t \in [0, 1]$
is a family of $\mathcal{C}^1$-small compactly supported contactomorphisms
with $\psi_{s, 0} = \id$.
Let $\{\Phi_s\}$ and $\{\Psi_{s, t}\}$
be the lifts to $(\mathbb{R}^{2n+1}, \xi_0)$,
and let $G_{s, t}: J^1 \mathbb{R}^{2n+1} \rightarrow \mathbb{R}$
be the family of transition functions of $\{\Psi_{\Psi_{s, t}}\}$.
Then $G_{s, t} \,\sharp F_s$ is a family of generating functions
for $\{ \psi_{s, t} \circ \phi_s \}$.
Since the $\Psi_{\Psi_{s,t}}$ are equivariant by translation by $1$
in the $z$-direction,
the transition functions satisfy
\[
G_{s, t} (x, y, \theta + 1, p_x, p_y, p_{\theta}, z - p_{\theta})
= G_{s,t} (x, y, \theta, p_x, p_y, p_{\theta}, z) \,.
\]
Since moreover the $F_s$ are invariant
by the $\mathbb{Z}$-action \eqref{equation: action Z for F},
we have
\[
G_{s, t} \,\sharp \, F_s \; \big( \underline{q} + (0, 0, 1) , p , q , \zeta \big)
\]
\[
= G_{s, t} \Big( \underline{q} + (0, 0, 1) \,,\, p \,,\,
F_s \big( \underline{q} + (0, 0, 1) + q \,,\, \zeta \big)
- p (\underline{q} + q) - p_{\theta} \Big)
+ F_s \big( \underline{q} + (0, 0, 1) + q \,,\, \zeta \big) - pq
\]
\[
= G_{s, t} \Big( \underline{q}, p,
F_s \big( \underline{q} + q , \zeta \big)
- p (\underline{q} + q) \Big)
+ F_s \big( \underline{q} + q, \zeta \big) - pq
\]
\[
= G_{s, t} \,\sharp \, F_s ( \underline{q}, p, q, \zeta) \,,
\]
and so the functions $G_{s, t} \,\sharp \, F_s$
descend to functions on
$\mathbb{R}^{2n} \times S^1
\times (\mathbb{R}^{2n+1})^{\ast} \times \mathbb{R}^{2n+1} \times \mathbb{R}^N$.
Similarly,
if $\{\phi_s\}$ and $\{\psi_{s, t}\}$ are lifts
of compactly supported Hamiltonian symplectomorphisms of $(\mathbb{R}^{2n}, \omega_0)$
then the transition functions satisfy
\[
G_{s, t} (x, y, \theta + a, p_x, p_y, p_{\theta}, z - a p_{\theta})
= G_{s,t} (x, y, \theta, p_x, p_y, p_{\theta}, z)
\]
for all $a \in \mathbb{R}$ and
by \autoref{proposition: gf lift to contact}
we can take the functions $F_s$ not depending on $\theta$,
thus the functions $G_{s, t} \,\sharp \, F_s$ descend
to $\mathbb{R}^{2n} \times (\mathbb{R}^{2n+1})^{\ast} \times \mathbb{R}^{2n+1} \times \mathbb{R}^N$.
Moreover,
there is a family of fibre preserving diffeomorphisms $\Phi_{s, t}$
with $\Phi_{s, 0} = \id$
such that the functions $F_{s, t} := (G_{s, t} \,\sharp \, F_s) \circ \Phi_{s, t}$
satisfy
\[
F_{s, t} ( \underline{q}, p, q, \zeta)
= (F_s)_{\infty} (\zeta) - pq
\]
outside a compact set
and $F_{s, 0}$ is a stabilization of $F_s$.
In particular,
$F_{s,t}$ is a family of special generating functions quadratic at infinity
for $\{\psi_{s, t} \circ \phi_s\}$
such that $F_{(s, 0)}$ is a stabilization of $F_s$
and $(F_{s,t})_{\infty}$ does not depend on $t$.
This finishes the proof, in the symplectic and contact case,
under the assumption that all $\{\psi_{s, t}\}$ are $\mathcal{C}^1$-small.
The general case can be obtained by writing $\{\psi_{s, t}\}_{t \in [0, 1]}$
as the concatenation 
$\{\psi_{s, t}\}_{t \in [0, t_1]} \sqcup \dots \sqcup \{\psi_{s, t}\}_{t \in [t_M, 1]}$
for a sufficiently fine subdivision
$0 < t_1 < \dots < t_M < 1$.
\end{proof}

We say that a compactly supported Hamiltonian or contact isotopy is non-negative
if it is generated by a non-negative compactly supported
Hamiltonian function.
Consider the partial order
on the group of compactly supported Hamiltonian symplectomorphisms
of $(\mathbb{R}^{2n}, \omega_0)$
defined by posing $\varphi_0 \leq \varphi_1$
if $\varphi_1 \circ \varphi_0^{-1}$ is the time-1 map
of a non-negative compactly supported Hamiltonian isotopy \cite{Viterbo},
and the partial order
on the group of compactly supported contactomorphisms
of $(\mathbb{R}^{2n} \times S^1, \xi_0)$ contact isotopic to the identity
defined by posing $\phi_0 \leq \phi_1$
if $\phi_1 \circ \phi_0^{-1}$ is the time-1 map
of a non-negative compactly supported contact isotopy \cite{Bhupal, San11}.
Notice that a compactly supported Hamiltonian isotopy $\{\varphi_t\}$
is non-negative if and only if its lift $\{\widetilde{\varphi_t}\}$
is non-negative,
thus for any two compactly supported Hamiltonian symplectomorphisms
$\varphi_0$ and $\varphi_1$ we have
$\varphi_0 \leq \varphi_1$ if and only if
$\widetilde{\varphi_0} \leq \widetilde{\varphi_1}$.
The next result follows from the proof
of \autoref{lemma: special}
and the Hamilton--Jacobi equation in \cite[Lemma 3.6]{San11}.

\begin{lemma}\label{lemma: monotonicity special}
Let $\{\phi_{s, t}^{(0)}\}$ and $\{\phi_{s, t}^{(1)}\}$
be two 2-parameter families of compactly supported
Hamiltonian symplectomorphisms of $(\mathbb{R}^{2n}, \omega_0)$
or compactly supported contactomorphisms
of $(\mathbb{R}^{2n} \times S^1, \xi_0)$
contact isotopic to the identity
with
$\phi_{s, 0}^{(0)} = \phi_{s, 0}^{(1)} = \id$
and $\phi_{s, t}^{(0)} \leq \phi_{s, t}^{(1)}$.
Let $\{\mu_{s, t}^{(u)}\}_{u \in [0, 1]}$
be a family of non-negative compactly supported
Hamiltonian or contact isotopies
such that $\mu_{s, t}^{(0)} = \id$
and $\mu_{s, t}^{(1)} = \phi_{s, t}^{(1)} \circ (\phi_{s, t}^{(0)})^{-1}$.
Denote $\phi_{s, t}^{(u)} = \mu_{s, t}^{(u)} \circ \phi_{s, t}^{(0)}$.
Then there is a family
of special generating functions quadratic at infinity
$F_{s, t}^{(u)}$ for $\{\phi_{s, t}^{(u)}\}$
such that $F_{s, 0}^{(u)}$ is a non-degenerate quadratic form for all $s$ and $u$,
$(F_{s, t}^{(u)})_{\infty}$ does not depend on $s,t$ and $u$,
and $\frac{d}{du} F_{s, t}^{(u)} \geq 0$.
\end{lemma}

\section{\texorpdfstring{$\mathbb{Z}_k$}{}-invariant functions detecting \texorpdfstring{$k$}{}-periodic points of Hamiltonian symplectomorphisms}
\label{section: invariant generating functions symplectic}

An important ingredient
to define the generating function $\mathbb{Z}_k$-equivariant homology
of domains of $(\mathbb{R}^{2n}, \omega_0)$
is the following composition formula from \cite{Allais}
(which is given there
only in the case of generating functions
without fibre variables,
but can easily be adapted to the general case).

\begin{prop}\label{proposition: composition formula}
Let $k$ be an odd natural number.
Suppose that
$F: \mathbb{R}^{2n} \times \mathbb{R}^N \rightarrow \mathbb{R}$
is a generating function
of a symplectomorphism $\varphi$
of $(\mathbb{R}^{2n}, \omega_0)$.
Then the function
\[
F^{\sharp k}:
\mathbb{R}^{2n}
\times (\mathbb{R}^{2n(k-1)} \times \mathbb{R}^{Nk})
\rightarrow \mathbb{R}
\]
defined by
\[
F^{\sharp k}
(x_1, y_1; x_2, y_2, \dots, x_k, y_k, \zeta_1, \dots, \zeta_k)
= \sum_{j=1}^k
F \Big( \frac{x_j + x_{j+1}}{2},
\frac{y_j + y_{j+1}}{2}, \zeta_j \Big)
+ \frac{1}{2} \, (x_j y_{j+1} - x_{j+1}y_j) \,,
\]
with the convention $(x_{k+1}, y_{k+1}) = (x_1, y_1)$,
is a generating function of $\varphi^k$.
Moreover, if $\varphi$ is compactly supported
and $F$ is normalized
then $F^{\sharp k}$ is also normalized.
\end{prop}

\begin{proof}
For every $j$ we have
\[
\frac{\partial F^{\sharp k}}{\partial \zeta_j}
(x_1, y_1, \dots, x_k, y_k, \zeta_1, \dots, \zeta_k)
= \frac{\partial F}{\partial \zeta_j}
\Big( \frac{x_j + x_{j+1}}{2},
\frac{y_j + y_{j+1}}{2}, \zeta_j \Big) \,,
\]
thus
$\frac{\partial F^{\sharp k}}{\partial \zeta_j}
(x_1, y_1, \dots, x_k, y_k, \zeta_1, \dots, \zeta_k) = 0$
if and only if
$\big( \frac{x_j + x_{j+1}}{2},
\frac{y_j + y_{j+1}}{2}, \zeta_j \big)$
is a fibre critical point of $F$.
In this case we define
\begin{equation}\label{equation: change of coordinates 1 sympl}
\begin{cases}
X_j = \frac{x_j + x_{j+1}}{2} \\
Y_j = \frac{y_j + y_{j+1}}{2}
\end{cases}
\end{equation}
and we let $(\overline{X}_j, \overline{Y}_j) \in \mathbb{R}^{2n}$
be the image of $(X_j, Y_j, \zeta_j)$
by the inverse of the diffeomorphism
\eqref{equation: diffeomorphism gf symplectic}.
Then
\begin{equation}\label{equation: change of coordinates 2 sympl}
\begin{cases}
X_j = \frac{\overline{X}_j + \varphi_x (\overline{X}_j, \overline{Y}_j) }{2} \\
Y_j = \frac{\overline{Y}_j + \varphi_y (\overline{X}_j, \overline{Y}_j) }{2}
\end{cases}
\end{equation}
and 
\[
\begin{cases}
d_1F (X_j, Y_j, \zeta_j)
= \overline{Y}_j - \varphi_y (\overline{X}_j, \overline{Y}_j) \\
d_2F (X_j, Y_j, \zeta_j)
= \varphi_x (\overline{X}_j, \overline{Y}_j) - \overline{X}_j \,,
\end{cases}
\]
where we denote
$\varphi (\overline{X}_j, \overline{Y}_j)
= \big( \varphi_x (\overline{X}_j, \overline{Y}_j),
\varphi_y (\overline{X}_j, \overline{Y}_j) \big)$.
Moreover, since $k$ is odd,
\eqref{equation: change of coordinates 1 sympl}
implies
\begin{equation}\label{equation: change of coordinates 3 sympl}
\begin{cases}
x_j = \sum_{l=0}^{k-1} (-1)^l X_{j+l} \\
y_j = \sum_{l=0}^{k-1} (-1)^l Y_{j+l} \,,
\end{cases}
\end{equation}
where we use a cyclic convention on the indices:
$(X_{k+m}, Y_{k+m}) := (X_m, Y_m)$ for every $m$.

For every $j$ we have
\begin{gather*}
\frac{\partial F^{\sharp k}}{\partial x_j}
(x_1, y_1, \dots, x_k, y_k, \zeta_1, \dots, \zeta_k)\\ 
= \frac{1}{2} \, d_1F (X_j, Y_j, \zeta_j)
+ \frac{1}{2} \, d_1F (X_{j-1}, Y_{j-1}, \zeta_{j-1})
+ \frac{1}{2} \, (y_{j+1} - y_{j-1})\\
= \frac{1}{2} \, d_1F (X_j, Y_j, \zeta_j)
+ \frac{1}{2} \, d_1F (X_{j-1}, Y_{j-1}, \zeta_{j-1})
+ Y_j - Y_{j-1} \\
= \overline{Y}_j - \varphi_y (\overline{X}_{j-1}, \overline{Y}_{j-1}) \,,
\end{gather*}
and similarly
\[
\frac{\partial F^{\sharp k}}{\partial y_j}
(x_1, y_1, \dots, x_k, y_k, \zeta_1, \dots, \zeta_k)
= \varphi_x (\overline{X}_{j-1}, \overline{Y}_{j-1}) - \overline{X}_j \,.
\]
Thus $(x_1, y_1, \dots, x_k, y_k, \zeta_1, \dots, \zeta_k)$
is a fibre critical point of $F^{\sharp k}$
if and only if $(X_j, Y_j, \zeta_j)$
is a fibre critical point of $F$ for all $j$
and
\begin{equation}\label{equation: fibre critical points}
(\overline{X}_j, \overline{Y}_j) =
\varphi (\overline{X}_{j-1}, \overline{Y}_{j-1})
\end{equation}
for $j= 2, \dots, k$.

We leave to the reader the verification that
the differential of the vertical derivative of $F^{\sharp k}$
at a fibre critical point has maximal rank,
and so $F^{\sharp k}$ satisfies the transversality condition
in the definition of generating functions.
In order to prove that $F^{\sharp k}$
is a generating function of $\varphi^k$
it then remains to show that the Lagrangian immersion
$i_{F^{\sharp k}}:
\Sigma_{F^{\sharp k}} \rightarrow T^{\ast}\mathbb{R}^{2n}$
induces a diffeomorphism
between $\Sigma_{F^{\sharp k}}$ and $L_{\varphi^k}$.
But for any fibre critical point 
$(x_1, y_1, \dots, x_k, y_k, \zeta_1, \dots, \zeta_k)$
of $F^{\sharp k}$
the relations 
\eqref{equation: change of coordinates 2 sympl},
\eqref{equation: change of coordinates 3 sympl}
and \eqref{equation: fibre critical points}
give
\[
(x_1, y_1)
= \Big(
\frac{\overline{X}_1
+ \varphi_x (\overline{X}_k, \overline{Y}_k)}{2} \,,\,
\frac{\overline{Y}_1
+ \varphi_y (\overline{X}_k, \overline{Y}_k)}{2}
\Big)
\]
and
\[
\varphi (\overline{X}_k, \overline{Y}_k)
= \varphi^k (\overline{X}_1, \overline{Y}_1) \,,
\]
and so we have
\begin{gather*}
i_{F^{\sharp k}} 
(x_1, y_1, \dots, x_k, y_k, \zeta_1, \dots, \zeta_k) \\
= \Big( x_1, y_1,
\frac{\partial F^{\sharp k}}{\partial x_1}
(x_1, y_1, \dots, x_k, y_k, \zeta_1, \dots, \zeta_k) ,
\frac{\partial F^{\sharp k}}{\partial y_1}
(x_1, y_1, \dots, x_k, y_k, \zeta_1, \dots, \zeta_k) \Big)\\
= \Big(
\frac{\overline{X}_1
+ (\varphi^k)_x (\overline{X}_1, \overline{Y}_1)}{2} \,,\,
\frac{\overline{Y}_1
+ (\varphi^k)_y (\overline{X}_1, \overline{Y}_1)}{2} \,,\,
\overline{Y}_1 - (\varphi^k)_y (\overline{X}_1, \overline{Y}_1) \,,\,
(\varphi^k)_x (\overline{X}_1, \overline{Y}_1) - \overline{X}_1
\Big) \\
= \Gamma_{\varphi^k} (\overline{X}_1, \overline{Y}_1) \,.
\end{gather*}
Thus 
$i_{F^{\sharp k}}:
\Sigma_{F^{\sharp k}} \rightarrow T^{\ast}\mathbb{R}^{2n}$
is the composition of the diffeomorphism
\begin{equation}\label{equation: diffeomorphism in the proof}
\Sigma_{F^{\sharp k}} \rightarrow \mathbb{R}^{2n} \,,\;
(x_1, y_1, \dots, x_k, y_k, \zeta_1, \dots, \zeta_k)
\mapsto (\overline{X}_1, \overline{Y}_1)
\end{equation}
with $\Gamma_{\varphi^k}$,
and so it induces a diffeomorphism
between $\Sigma_{F^{\sharp k}}$ and $L_{\varphi^k}$.

Suppose now that $\varphi$ is compactly supported
and $F$ is normalized.
If $(x_1, y_1, \dots, x_k, y_k, \zeta_1, \dots, \zeta_k)$
is a critical point of $F^{\sharp k}$
then \eqref{equation: change of coordinates 2 sympl}
and \eqref{equation: change of coordinates 3 sympl}
give $x_j = \overline{X}_j$ and $y_j = \overline{Y}_j$.
By \autoref{proposition: critical values gf action}
we thus have
\begin{gather*}
F^{\sharp k}
(x_1, y_1, \dots, x_k, y_k, \zeta_1, \dots, \zeta_k)
= \sum_{j=1}^k F (X_j, Y_j, \zeta_j)
+ \frac{1}{2} \, (x_j y_{j+1} - x_{j+1}y_j)
\\
= \sum_{j=1}^k S (\overline{X}_j, \overline{Y}_j)
= \mathcal{A}_{\varphi^k} (\overline{X}_1, \overline{Y}_1) \,,
\end{gather*}
where for the last equality
we have used that
$(\varphi^k)^{\ast} \lambda_0 - \lambda_0 = d \big( \sum_{j=1}^k S \circ \varphi^{j-1} \big)$.
We thus conclude that $F^{\sharp k}$ is normalized.
\end{proof}

The function
\[
F^{\sharp k}: \mathbb{R}^{2nk} \times \mathbb{R}^{Nk}
\rightarrow \mathbb{R}
\]
is invariant by the action of $\mathbb{Z}_k$
generated by the map
\begin{equation}\label{equation: action symplectic}
\sigma_k: (x_1, y_1, \dots, x_k, y_k, \zeta_1, \dots, \zeta_k)
\mapsto (x_2, y_2, \dots, x_k, y_k, x_1, y_1, \zeta_2, \dots, \zeta_k, \zeta_1) \,.
\end{equation}
Using the notations of the proof
of \autoref{proposition: composition formula},
we have seen that
$(x_1, y_1, \dots, x_k, y_k, \zeta_1, \dots, \zeta_k)$
is a critical point of $F^{\sharp k}$
if and only if
$(\overline{X}_1, \overline{Y}_1)$
is a $k$-periodic point of $\varphi$.
It follows from the proof
of \autoref{proposition: composition formula}
that under this bijection
between the critical points of $F^{\sharp k}$
and the $k$-periodic points of $\varphi$
the $\mathbb{Z}_k$-action \eqref{equation: action symplectic}
corresponds to the $\mathbb{Z}_k$-action
on the set of $k$-periodic points of $\varphi$
generated by the map that sends a $k$-periodic point $p$
to the $k$-periodic point $\varphi(p)$.

We now describe the relation
between the index of the Hessian of $F^{\sharp k}$
at a critical point
and the Maslov index of the corresponding fixed point
of $\varphi^k$.
This relation is needed in \autoref{section: balls},
in the calculation of the equivariant homology of balls.
Recall first from \cite[Appendix B]{Theret - camel}
and \cite[7.4.2]{Traynor}
that the Maslov index of a fixed point
of a Hamiltonian symplectomorphism
of $(\mathbb{R}^{2n}, \omega_0)$
can be defined using generating functions as follows.
The Maslov index of a path $\{ L_t \}_{t \in [0, 1]}$
of linear Lagrangian submanifolds
of $(T^{\ast} \mathbb{R}^m, \omega_{\can})$
is defined by
\[
\nu \big( \{ L_t \}_{t \in [0, 1]} \big)
= \ind (Q_1) - \ind(Q_0) \,,
\]
where $Q_t$ is any 1-parameter family
of generating quadratic forms for $\{ L_t \}_{t \in [0, 1]}$
and where we denote by $\ind (Q)$
the index of a quadratic form $Q$.
The Maslov index $\nu (p)$
of a fixed point $p$
of a Hamiltonian symplectomorphism $\varphi$
of $(\mathbb{R}^{2n}, \omega_0)$
is then defined to be the Maslov index of the path
$p \mapsto L_{d\varphi_t(p)}$
of linear Lagrangian submanifolds
of $(T^{\ast}\mathbb{R}^{2n}, \omega_{\can})$,
where $\{\varphi_t\}_{t \in [0, 1]}$
is any Hamiltonian isotopy with $\varphi_1 = \varphi$.
If $F_t: E \rightarrow \mathbb{R}$
is a 1-parameter family of generating functions
for $\{\varphi_t\}$ then $d^2F_t \big(\Phi_{F_t}(p)\big)$,
where $\Phi_{F_t}$ denotes the diffeomorphism
\eqref{equation: diffeomorphism gf symplectic}
from $\mathbb{R}^{2n}$ to $\Sigma_{F_t}$,
is a 1-parameter family
of generating quadratic forms
for $\{ L_{d\varphi_t(p)} \}$,
and so
\[
\nu (p) =
\ind \Big( d^2F_1 \big(\Phi_{F_1}(p)\big) \Big)
- \ind \Big( d^2F_0 \big(\Phi_{F_0}(p)\big) \Big) \,.
\]
We can now prove the following result.

\begin{prop}\label{proposition: index}
Let $\varphi$ be a compactly supported
Hamiltonian symplectomorphism of $(\mathbb{R}^{2n}, \omega_0)$
with generating function quadratic at infinity $F$,
and consider the associated generating function
$F^{\sharp k}$ of $\varphi^k$.
Then for every fixed point $p$ of $\varphi^k$
we have
\[
\nu (p) = \ind
\Big( d^2F^{\sharp k} \big( \Phi_{F^{\sharp k}} (p) \big) \Big)
- k \, \ind (F_{\infty}) - n (k - 1) \,.
\]
\end{prop}

\begin{proof}
Let $\{\varphi_t\}_{t \in [0, 1]}$
be a compactly supported Hamiltonian isotopy
with $\varphi_1 = \varphi$.
By Proposition \ref{proposition: Serre fibration},
there is a 1-parameter family $F_t$
of generating functions quadratic at infinity for $\{\varphi_t\}$
with $F_1 = F \oplus Q$ for a non-degenerate quadratic form $Q$.
By \autoref{proposition: composition formula},
$(F_t)^{\sharp k}$ is a 1-parameter family
of generating functions for $\{\varphi_t^k\}$,
thus by the above discussion we have
\[
\nu (p) = \ind
\Big( d^2 (F_1)^{\sharp k} \big( \Phi_{(F_1)^{\sharp k}} (p) \big) \Big)
- \ind
\Big( d^2 (F_0)^{\sharp k} \big( \Phi_{(F_0)^{\sharp k}} (p) \big) \Big) \,.
\]
But
\[
\ind
\Big( d^2 (F_1)^{\sharp k} \big( \Phi_{(F_1)^{\sharp k}} (p) \big) \Big)
= \ind
\Big( d^2 F^{\sharp k} \big( \Phi_{F^{\sharp k}} (p) \big) \Big)
+ k \, \ind (Q) \,,
\]
since $(F_1)^{\sharp k} = F^{\sharp k} \oplus Q^{\oplus k}$,
and
\[
\ind \Big( d^2 (F_0)^{\sharp k} \big( \Phi_{(F_0)^{\sharp k}} (p) \big) \Big)
= k \, \ind \big((F_0)_{\infty}\big) + n(k - 1) \,,
\]
since
\[
(F_0)^{\sharp k}
(x_1, y_1, \dots, x_k, y_k, \zeta_1, \dots, \zeta_k)
= \sum_{j=1}^k F_0 (x_j, y_j, \zeta_j)
+ \sum_{j=1}^k \frac{1}{2} (x_j y_{j+1} - x_{j+1}y_j) \,,
\]
the index of the quadratic form
$\sum_{j=1}^k \frac{1}{2} (x_j y_{j+1} - x_{j+1}y_j)$
is $n(k - 1)$,
and $F_0 = Q_0 \circ \Phi$ for a fibre preserving diffeomorphism $\Phi$
and a quadratic form $Q_0$ with $\ind (Q_0) = \ind \big( (F_0)_{\infty} \big)$.
Thus
\begin{align*}
    \nu (p) &= \ind
\Big( d^2 F^{\sharp k} \big( \Phi_{F^{\sharp k}} (p) \big) \Big)
- k \Big( \ind \big( (F_0)_{\infty} \big) - \ind (Q) \Big)
- n (k - 1)\\
&= \ind
\Big( d^2F^{\sharp k} \big( \Phi_{F^{\sharp k}} (p) \big) \Big)
- k \, \ind (F_{\infty}) - n (k - 1) \,,
\end{align*}
where for the last equality we have used
that $\ind \big( (F_0)_{\infty} \big) = \ind \big( (F_1)_{\infty} \big)$,
since $(F_t)_{\infty}$ is a 1-parameter family of non-degenerate quadratic forms,
and
\[
\ind \big( (F_1)_{\infty} \big) = \ind (F_{\infty}) + \ind (Q) \,. \qedhere
\]
\end{proof}

Consider the change of variables
$A$ of $\mathbb{R}^{2nk} \times \mathbb{R}^{Nk}$
defined by
\[
A (x_1, y_1, \dots, x_k, y_k, \zeta_1, \dots, \zeta_k)
\]
\[
= \Big(\, \frac{1}{k} \, \sum_{j=1}^k x_j \,,\, \frac{1}{k} \, \sum_{j=1}^k y_j \,,\,
\frac{1}{k} \, (x_2 - x_1) \,,\, \frac{1}{k} \, (y_2 - y_1) \,,\, \dots \,,\,
\frac{1}{k} \, (x_k - x_1) \,,\, \frac{1}{k} \, (y_k - y_1) \, ,
\zeta_1, \dots, \zeta_k \,\Big) \,.
\]
Observe that
\[
A \circ \sigma_k \circ A^{-1}
(x_1, y_1, \dots, x_k, y_k, \zeta_1, \dots, \zeta_k)
\]
\[
= (x_1, y_1, x_3 - x_2, y_3 - y_2, \dots, x_k - x_2, y_k - y_2, -x_2, -y_2,
\zeta_2, \dots, \zeta_k, \zeta_1) \,,
\]
thus $A \circ \sigma_k \circ A^{-1}$
extends to a map $\overline{\sigma_k}$
generating a $\mathbb{Z}_k$-action
on $S^{2n} \times \mathbb{R}^{2n(k-1)} \times \mathbb{R}^{Nk}$.

\begin{prop}\label{proposition: extend to sphere}
Let $F: \mathbb{R}^{2n} \times \mathbb{R}^N \rightarrow \mathbb{R}$
be a generating function quadratic at infinity
of a compactly supported Hamiltonian symplectomorphism
of $(\mathbb{R}^{2n}, \omega_0)$,
and consider the associated function
$F^{\sharp k}: \mathbb{R}^{2nk} \times \mathbb{R}^{Nk} \rightarrow \mathbb{R}$.
Then $F^{\sharp k} \circ A^{-1}$ can be extended
to a function
\[
\overline{F^{\sharp k}}: S^{2n} \times \mathbb{R}^{2n(k-1)} \times \mathbb{R}^{Nk}
\rightarrow \mathbb{R} \,,
\]
invariant by the $\mathbb{Z}_k$-action
generated by $\overline{\sigma_k}$.
Moreover, if $F$ is special
then $\overline{F^{\sharp k}}$ is quadratic at infinity
with respect to the vector bundle
$S^{2n} \times \mathbb{R}^{2n(k-1)} \times \mathbb{R}^{Nk}
\rightarrow S^{2n}$.
\end{prop}

\begin{proof}
We have
\[
F^{\sharp k} \circ A^{-1}
(x_1, y_1, \dots, x_k, y_k, \zeta_1, \dots, \zeta_k)
\]
\[
= F \Big( x_1 + \frac{k-2}{2} \, x_2 - \sum_{l \neq 2} x_l \,,\,
y_1 + \frac{k-2}{2} \, y_2 - \sum_{l \neq 2} y_l \,,\, \zeta_1 \Big)
\]
\[
+ \sum_{j=2}^{k-1}
F \Big( x_1 + \frac{k-2}{2} \, x_j + \frac{k-2}{2} \, x_{j+1} - \sum_{l \neq j, j+1} x_l \,,\,
y_1 + \frac{k-2}{2} \, y_j + \frac{k-2}{2} \, y_{j+1} - \sum_{l \neq j, j+1} y_l \,,\, \zeta_j \Big)
\]
\[
+ F \Big( x_1 + \frac{k-2}{2} \, x_k - \sum_{l \neq k} x_l \,,\,
y_1 + \frac{k-2}{2} \, y_k - \sum_{l \neq k} y_k \,,\, \zeta_k \Big)
+ \sum_{j=2}^{k-1} \frac{k^2}{2} (x_j y_{j+1} - x_{j+1}y_j) \,,
\]
where the index $l$ is always in the set $\{2, \dots, k\}$.
The function $F^{\sharp k} \circ A^{-1}$ extends to a function
$\overline{F^{\sharp k}}$
on $S^{2n} \times \mathbb{R}^{2n(k-1)} \times \mathbb{R}^{Nk}$
by setting
\[
\overline{F^{\sharp k}} \, (p_{\infty}, x_2, y_2, \dots, x_k, y_k, \zeta_1, \dots, \zeta_k)
= \sum_{j=1}^k \overline{F} (p_{\infty}, \zeta_j)
+ \sum_{j=2}^{k-1} \frac{k^2}{2} (x_j y_{j+1} - x_{j+1}y_j) \,.
\]
By definition,
$\overline{F^{\sharp k}}$ is invariant
by the $\mathbb{Z}_k$-action generated by $\overline{\sigma_k}$.
Suppose now that $F$ is special,
and consider the quadratic form $Q$
on $S^{2n} \times \mathbb{R}^{2n(k-1)} \times \mathbb{R}^{Nk}$
defined by
\[
Q (p, x_2, y_2, \dots, x_k, y_k, \zeta_1, \dots, \zeta_k)
= \sum_{j=1}^k F_{\infty} (\zeta_j)
+ \sum_{j=2}^{k-1} \frac{k^2}{2} (x_j y_{j+1} - x_{j+1}y_j)  \,.
\]
Since $\lVert d_1F \rVert$, $\lVert d_2F \rVert$
and $\lVert \frac{\partial}{\partial \zeta} (F - F_{\infty}) \rVert$
are bounded,
$\partial_v ( \overline{F^{\sharp k}} - Q )$ is bounded
and thus $\overline{F^{\sharp k}}$ is quadratic at infinity.
\end{proof}

Finally we state the following proposition,
whose proof is immediate.

\begin{prop}\label{proposition: stabilization etc Fk}
\begin{itemize}
\item[(i)]
Let $F: \mathbb{R}^{2n} \times \mathbb{R}^N \rightarrow \mathbb{R}$
be a generating function quadratic at infinity
of a compactly supported Hamiltonian symplectomorphism
of $(\mathbb{R}^{2n}, \omega_0)$.
Then for every non-degenerate quadratic form
$Q: \mathbb{R}^{N'} \rightarrow \mathbb{R}$
we have
\[
\overline{(F \oplus Q)^{\sharp k}}
= \overline{F^{\sharp k}} \oplus Q^{\oplus k} \,.
\]
\item[(ii)]
Let $F_0: \mathbb{R}^{2n} \times \mathbb{R}^N \rightarrow \mathbb{R}$
and $F_1: \mathbb{R}^{2n} \times \mathbb{R}^N \rightarrow \mathbb{R}$
be generating functions quadratic at infinity
of a compactly supported Hamiltonian symplectomorphism
of $(\mathbb{R}^{2n}, \omega_0)$
such that $\overline{F_0} = \overline{F_1} \circ \overline{\Phi}$
for a fibre preserving diffeomorphism $\overline{\Phi}$
of $S^{2n} \times \mathbb{R}^N$.
Denote by $\Phi$ the induced fibre preserving diffeomorphism
of $\mathbb{R}^{2n} \times \mathbb{R}^N$,
so that $F_0 = F_1 \circ \Phi$.
Then
\[
(F_0)^{\sharp k} = (F_1)^{\sharp k} \circ \Phi_k \,,
\]
where $\Phi_k$ is the diffeomorphism
\[
(x_1, y_1, \dots, x_k, y_k, \zeta_1, \dots, \zeta_k)
\mapsto
(x_1, y_1, \dots, x_k, y_k, \zeta_1', \dots, \zeta_k')
\]
of $\mathbb{R}^{2nk} \times \mathbb{R}^{Nk}$,
equivariant with respect to the $\mathbb{Z}_k$-action
generated by $\sigma_k$,
defined by
\[
\Big( \frac{x_j + x_{j+1}}{2}, \frac{y_j + y_{j+1}}{2}, \zeta_j' \Big)
= \Phi \Big( \frac{x_j + x_{j+1}}{2}, \frac{y_j + y_{j+1}}{2}, \zeta_j \Big) \,.
\]
Moreover, $A \circ \Phi_k \circ A^{-1}$
extends to a fibre preserving diffeomorphism $\overline{\Phi_k}$
of $S^{2n} \times \mathbb{R}^{2n(k-1)} \times \mathbb{R}^{Nk}$,
equivariant with respect to the $\mathbb{Z}_k$-action
generated by $\overline{\sigma_k}$,
and we have
\[
\overline{(F_0)^{\sharp k}} = \overline{(F_1)^{\sharp k}} \circ \overline{\Phi_k} \,.
\]
\end{itemize}
\end{prop}


\section{\texorpdfstring{$\mathbb{Z}_k$}{}-equivariant generating function homology for domains of \texorpdfstring{$(\mathbb{R}^{2n}, \omega_0)$}{}}\label{section: equivariant symplectic homology}

As discussed in the introduction
(\autoref{remark: category theoretic construction}),
in order to define
the $\mathbb{Z}_k$-equivariant generating function homology
of domains of $(\mathbb{R}^{2n}, \omega_0)$
we first define
the $\mathbb{Z}_k$-equivariant homology
of compactly supported Hamiltonian isotopies
and then consider the limit,
in a certain sense,
over all Hamiltonian isotopies supported in a given domain.

Let $\{\varphi_t\}_{t \in [0, 1]}$
be a compactly supported Hamiltonian isotopy
of $(\mathbb{R}^{2n}, \omega_0)$.
Its $\mathbb{Z}_k$-equivariant homology
is defined by the following category theoretical construction.
By \autoref{lemma: special},
there exists a 1-parameter family
of special generating functions quadratic at infinity
$F_t: \mathbb{R}^{2n} \times \mathbb{R}^N \rightarrow \mathbb{R}$
for $\{\varphi_t\}$
such that $F_0$ is a non-degenerate quadratic form.
Moreover, by \autoref{lemma: lemmas gf} (iii),
for any two such 1-parameter families
$F_t^{(0)}$ and $F_t^{(1)}$
there are non-degenerate quadratic forms $Q_0$ and $Q_1$
and a 2-parameter family $\overline{\Phi_{s, t}}$
of fibre preserving diffeomorphisms of $S^{2n} \times \mathbb{R}^N$
such that $\overline{\Phi_{0, t}} = \id$,
$\overline{F_t^{(0)}} \oplus Q_0
= ( \overline{F_t^{(1)}} \oplus Q_1) \circ \overline{\Phi_{1, t}}$,
and $(\overline{F_t^{(0)}} \oplus Q_0) \circ \Phi_{s, 0}^{-1}$
is a contractible loop.
We consider the category
$\mathcal{F} \big( \{\varphi_t\}\big)$
whose objects are the 1-parameter families
$F_t: \mathbb{R}^{2n} \times \mathbb{R}^N \rightarrow \mathbb{R}$
of special generating functions quadratic at infinity for $\{\varphi_t\}$
such that $F_0$ is a non-degenerate quadratic form,
and whose morphisms $F_t^{(0)} \rightarrow F_t^{(1)}$
are the triples $(Q_0, Q_1, \overline{\Phi})$
with $Q_0$ and $Q_1$ non-degenerate quadratic forms
and $\overline{\Phi} = \overline{\Phi_{s, t}}$ a 2-parameter family
of fibre preserving diffeomorphisms of $S^{2n} \times \mathbb{R}^N$
such that $\overline{\Phi_{0, t}} = \id$,
$\overline{F_t^{(0)}} \oplus Q_0
= ( \overline{F_t^{(1)}} \oplus Q_1) \circ \overline{\Phi_{1, t}}$,
and $(\overline{F_t^{(0)}} \oplus Q_0) \circ \Phi_{s, 0}^{-1}$
is a contractible loop.
We define moreover the composition of two morphisms
$(Q_0, Q_1, \overline{\Phi}): F_t^{(0)} \rightarrow F_t^{(1)}$
and $(Q_1', Q_2', \overline{\Psi}): F_t^{(1)} \rightarrow F_t^{(2)}$
to be the morphism
\begin{equation}\label{equation: composition}
(Q_1', Q_2', \overline{\Psi}) \circ (Q_0, Q_1, \overline{\Phi})
= \big( \, Q_0 \oplus Q_1' \,,\, Q_2' \oplus Q_1 \,,\,
\overline{\Psi'} \circ l_{Q_1, Q_1'} \circ \overline{\Phi'} \, \big):
F_t^{(0)} \rightarrow F_t^{(2)} \,,
\end{equation}
where,
denoting the domains of $F_t^{(1)}$, $Q_1$ and $Q_1'$
by $E$, $E_1$ and $E_1'$ respectively,
\[
\overline{\Phi'} = (\overline{\Phi}, \id):
(E \oplus E_1) \oplus E_1' \rightarrow (E \oplus E_1) \oplus E_1' \,,
\]
\[
\overline{\Psi'} = (\overline{\Psi}, \id):
(E \oplus E_1') \oplus E_1 \rightarrow (E \oplus E_1') \oplus E_1 \,,
\]
and $l_{Q_1, Q_1'}: E \oplus E_1 \oplus E_1' \rightarrow E \oplus E_1' \oplus E_1$
is the fibre preserving linear diffeomorphism
that switches the second and third blocks of coordinates.

For $a \leq b$ in $\mathbb{R} \cup \{ \pm \infty \}$
that are not in the action spectrum of $(\varphi_1)^k$
we define a functor $G_{\mathbb{Z}_k, \ast}^{(a,b]}$
from $\mathcal{F} \big(\{\varphi_t\}\big)$
to the category of graded modules over $\mathbb{Z}_k$
by posing
\[
G_{\mathbb{Z}_k, \ast}^{(a,b]} (F_t)
= H_{\mathbb{Z}_k, \ast + k\ind((F_1)_{\infty}) + n(k-1)}
\big( \{ \overline{(F_1)^{\sharp k}} \leq b \} \,,\,
\{ \overline{(F_1)^{\sharp k}} \leq a \}  ; \mathbb{Z}_k \big)
\]
(the equivariant singular relative homology of the pair
$( \{ \overline{(F_1)^{\sharp k}} \leq b \} \,,\, \{ \overline{(F_1)^{\sharp k}} \leq a \} )$
with coefficients in $\mathbb{Z}_k$),
with the convention that
$\overline{( F_1 )^{\sharp k}} \leq \infty$ means
$\overline{( F_1 )^{\sharp k}} \leq c$
for $c$ bigger than all the actions of fixed points of $(\varphi_1)^k$
and $\overline{( F_1 )^{\sharp k}} \leq -\infty$
means $\overline{( F_1 )^{\sharp k}} \leq c$
for $c$ smaller than all the actions of fixed points of $(\varphi_1)^k$,
and by associating to any morphism
$(Q_0, Q_1, \overline{\Phi}): F_t^{(0)} \rightarrow F_t^{(1)}$
a homomorphism
\[
G_{\mathbb{Z}_k, \ast}^{(a,b]} (Q_0, Q_1, \overline{\Phi}):
G_{\mathbb{Z}_k, \ast}^{(a,b]} (F_t^{(1)})
\rightarrow  G_{\mathbb{Z}_k, \ast}^{(a,b]} (F_t^{(0)})
\]
as follows.
Note first that $(Q_0, Q_1, \overline{\Phi})$
can be written as the composition
\[
F_t^{(0)} \xrightarrow{ (Q_0, 0, \id) } F_t^{(0)} \oplus Q_0
\xrightarrow{ (0, 0, \overline{\Phi}) } F_t^{(1)} \oplus Q_1
\xrightarrow{ (0, Q_1, \id) } F_t^{(1)} \,.
\]
By \autoref{proposition: stabilization etc Fk} (i),
$\overline{(F_t^{(1)} \oplus Q_1)^{\sharp k}}
= \overline{(F_t^{(1)})^{\sharp k}} \oplus (Q_1)^{\oplus k}$.
We define
\[
G_{\mathbb{Z}_k, \ast}^{(a,b]} (0, Q_1, \id):
G_{\mathbb{Z}_k, \ast}^{(a,b]} (F_t^{(1)})
\rightarrow  G_{\mathbb{Z}_k, \ast}^{(a,b]} (F_t^{(1)} \oplus Q_1)
\]
to be (similarly to \cite[3.6]{Traynor}
and \cite[Lemma 3.8]{S - equivariant})
the isomorphism
\[
H_{\mathbb{Z}_k, \ast + k\ind ((F^{(1)}_1)_{\infty})+ n(k-1)}
\big( \{ \overline{(F^{(1)}_1)^{\sharp k}} \leq b \} \,,\,
\{ \overline{(F^{(1)}_1)^{\sharp k}} \leq a \}  ; \mathbb{Z}_k \big)
\rightarrow
\]
\[
H_{\mathbb{Z}_k, \ast + k \ind ((F^{(1)}_1)_{\infty}) + n(k-1) + k \ind (Q_1)}
\big( \{ (\overline{F^{(1)}_1)^{\sharp k}} \oplus (Q_1)^{\oplus k}\leq b \} \,,\,
\{ \overline{(F^{(1)}_1)^{\sharp k}} \oplus (Q_1)^{\oplus k} \leq a \}  ; \mathbb{Z}_k \big)
\]
induced by the $\mathbb{Z}_k$-equivariant Thom isomorphism
on the $\mathbb{Z}_k$-invariant vector bundle on which $(Q_1)^{\oplus k}$
is negative definite.
Similarly we define
\[
G_{\mathbb{Z}_k, \ast}^{(a,b]} (Q_0, 0, \id):
G_{\mathbb{Z}_k, \ast}^{(a,b]} (F_t^{(0)} \oplus Q_0)
\rightarrow  G_{\mathbb{Z}_k, \ast}^{(a,b]} (F_t^{(0)})
\]
to be the inverse of the isomorphism
induced by the $\mathbb{Z}_k$-equivariant Thom isomorphism
on the $\mathbb{Z}_k$-invariant vector bundle on which $(Q_0)^{\oplus k}$
is negative definite.
Since $\overline{F_1^{(0)}} \oplus Q_0
= (\overline{F_1^{(1)}} \oplus Q_1) \circ \overline{\Phi_{1, 1}}$,
by \autoref{proposition: stabilization etc Fk} (ii) we have
\[
\overline{ (F_1^{(0)} \oplus Q_0)^{\sharp k} }
= \overline{ (F_1^{(1)} \oplus Q_1)^{\sharp k} } \circ \overline{(\Phi_{1,1})_k} \,.
\]
We define
\[
G_{\mathbb{Z}_k, \ast}^{(a,b]} (0, 0, \overline{\Phi}):
G_{\mathbb{Z}_k, \ast}^{(a,b]} (F_t^{(1)} \oplus Q_1)
\rightarrow  G_{\mathbb{Z}_k, \ast}^{(a,b]} (F_t^{(0)} \oplus Q_0)
\]
to be the inverse of the isomorphism induced by $\overline{(\Phi_{1,1})_k}$.
Finally we define
\[
G_{\mathbb{Z}_k, \ast}^{(a,b]} (Q_0, Q_1, \overline{\Phi})
= G_{\mathbb{Z}_k, \ast}^{(a,b]} (Q_0, 0, \id)
\circ G_{\mathbb{Z}_k, \ast}^{(a,b]} (0, 0, \overline{\Phi})
\circ G_{\mathbb{Z}_k, \ast}^{(a,b]} (0, Q_1, \id) \,,
\]
and remark that this homomorphism
is by construction always an isomorphism.

Given two morphisms
$(Q_0, Q_1, \overline{\Phi}): F_t^{(0)} \rightarrow F_t^{(1)}$
and $(Q_1', Q_2', \overline{\Psi}): F_t^{(1)} \rightarrow F_t^{(2)}$,
naturality of the $\mathbb{Z}_k$-equivariant Thom isomorphism
implies that
\[
G_{\mathbb{Z}_k, \ast}^{(a,b]}
\big( (Q_1', Q_2', \overline{\Psi}) \circ (Q_0, Q_1, \overline{\Phi})\big)
= G_{\mathbb{Z}_k, \ast}^{(a,b]} (Q_0, Q_1, \overline{\Phi})
\circ G_{\mathbb{Z}_k, \ast}^{(a,b]} (Q_1', Q_2', \overline{\Psi}) \,.
\]
The functor $G_{\mathbb{Z}_k,*}^{(a,b]}$
is thus well-defined.

We now define the $\mathbb{Z}_k$-equivariant homology
of the Hamiltonian isotopy $\{ \varphi_t \}_{t \in [0, 1]}$ by
\[
G_{\mathbb{Z}_k,*}^{(a,b]} \big(\{\varphi_t\}\big)
= \underset{}{\varprojlim} \;
\Big\{ G_{\mathbb{Z}_k,*}^{(a,b]} (F_t)
\Big\}_{ F_t \in \mathcal{F} (\{\varphi_t\})} \,.
\]
By definition of the limit,
for every object $F_t$ in $\mathcal{F} \big(\{\varphi_t\}\big)$
there exists a natural homomorphism
$i_{F_t}: G_{\mathbb{Z}_k,*}^{(a,b]} \big(\{\varphi_t\}\big)
\rightarrow G_{\mathbb{Z}_k,*}^{(a,b]} (F_t)$,
which is functorial with respect to $F_t$.
This homomorphism is in fact an isomorphism,
as shown in the next proposition.

\begin{prop}\label{proposition: commuting isomorphism}
For every object $F_t$ in $\mathcal{F} \big(\{\varphi_t\}\big)$ the natural homomorphism
\[
i_{F_t}: G_{\mathbb{Z}_k,*}^{(a,b]} \big(\{\varphi_t\}\big)
\rightarrow G_{\mathbb{Z}_k,*}^{(a,b]} (F_t) \,
\]
is an isomorphism.
\end{prop}

\begin{proof}
The result follows from the fact that
any two objects $F_t^{(0)}$ and $F_t^{(1)}$
in $\mathcal{F} (\{\varphi_t\})$
are related by a morphism,
and all the morphisms between them
induce the same isomorphism
from $G_{\mathbb{Z}_k, \ast}^{(a,b]} (F_t^{(1)})$
to $G_{\mathbb{Z}_k, \ast}^{(a,b]} (F_t^{(0)})$.
The first statement follows
by \autoref{lemma: lemmas gf} (iii).
For the second,
suppose that $(Q_0, Q_1, \overline{\Phi})$ and $(Q_0', Q_1', \overline{\Psi})$
are two morphisms from $F_t^{(0)}$ to $F_t^{(1)}$.
Showing that they induce the same isomorphism
is equivalent to showing that
$G_{\mathbb{Z}_k,*}^{(a,b]} (Q_0, Q_1, \overline{\Phi})
\circ G_{\mathbb{Z}_k,*}^{(a,b]} (Q_0', Q_1', \overline{\Psi})^{-1}$
is the identity.
But $G_{\mathbb{Z}_k,*}^{(a,b]} (Q_0, Q_1, \overline{\Phi})
\circ G_{\mathbb{Z}_k,*}^{(a,b]} (Q_0', Q_1', \overline{\Psi})^{-1}$
is the isomorphism induced by the morphism
$(Q_1', Q_0', \overline{\Psi}^{-1}) \circ (Q_0, Q_1, \overline{\Phi})$
from $F_t^{(0)}$ to itself.
It is thus enough to show that for any morphism
$(Q_0, Q_1, \overline{\Phi})$ from an object $F_t$ to itself
$G_{\mathbb{Z}_k,*}^{(a,b]} (Q_0, Q_1, \overline{\Phi})$
is the identity.
To see this,
observe first that,
since $\overline{F_t} \oplus Q_0
= (\overline{F_t} \oplus Q_1) \circ \overline{\Phi_{1, t}}$,
$Q_0$ and $Q_1$ have the same index.
Thus $Q_1 = Q_0 \circ l$
for an orientation preserving linear diffeomorphism $l$,
and so
\[
\overline{F_t} \oplus Q_0
= (\overline{F_t} \oplus Q_0) \circ (\id, l) \circ \overline{\Phi_{1, t}} \,.
\]
Thus it is enough to show that
for any morphism of the form $(0, 0, \overline{\Phi})$
from an object $F_t$ to itself
$G_{\mathbb{Z}_k,*}^{(a,b]} (0, 0, \overline{\Phi})$
is the identity.
This can be seen as follows.
Let $F_t^{(s)} = F_t \circ \Phi_{s,t}^{-1}$.
Since the loop $\overline{F_1^{(s)}}$ is contractible
through the family $\overline{F_t^{(s)}}$
followed by the contraction of the loop $\overline{F_0^{(s)}}$,
there is a 1-parameter family of loops $\overline{G_u^{(s)}}$
such that $\overline{G_0^{(s)}} = \overline{F_1^{(s)}}$,
$\overline{G_1^{(s)}} = \overline{F_1}$ and, for every $s$,
$\overline{G_u^{(s)}}$ is a lift to the space of generating functions
of the (contractible) loop based at $\varphi_1$
obtained by going from $\varphi_1$
to some $\varphi_{t(s)}$ along $\{\varphi_t\}$
and then back to $\varphi_1$.
By \autoref{proposition: Serre fibration},
the contraction of this 1-parameter family of loops
to the 1-parameter family of constant loops
can be lifted to the space of generating functions.
We thus obtain a 1-parameter family of loops $\overline{I_u^{(s)}}$
of generating functions quadratic at infinity of $\varphi_1$
from $\overline{I_0^{(s)}} = \overline{F_1}$
to $\overline{I_1^{(s)}} = \overline{F_1^{(s)}}$.
By \autoref{lemma: lemmas gf} (ii),
there is a 2-parameter family $\overline{\chi_{s, u}}$
of fibre preserving diffeomorphisms
with $\overline{\chi_{s, 0}} = \id$
and $\overline{I_u^{(s)}} \circ \overline{\chi_{s,u}} = \overline{F_1}$.
By \autoref{proposition: stabilization etc Fk} (ii)
we have $\overline{(I_u^{(s)})^{\sharp k}} \circ \overline{(\chi_{s,u})_{k}}
= \overline{(F_1)^{\sharp k}}$.
The restriction to
$\big( \{ \overline{(F_1)^{\sharp k}} \leq b \} \,,\,
\{ \overline{(F_1)^{\sharp k}} \leq a \} \big)$
of $(\overline{\chi_{s, 1})_k}^{\,-1} \circ \overline{(\Phi_{s, 1})_k}$
is a $\mathbb{Z}_k$-equivariant homotopy
from $(\overline{\chi_{0, 1})_k}^{\,-1}$ to
$(\overline{\chi_{1, 1})_k}^{\,-1} \circ \overline{(\Phi_{1, 1})_k}$.
Since the restriction to
$\big( \{ \overline{(F_1)^{\sharp k}} \leq b \} \,,\,
\{ \overline{(F_1)^{\sharp k}} \leq a \} \big)$
of $(\overline{\chi_{0, 1})_k}^{\,-1}$ and $(\overline{\chi_{1, 1})_k}^{\,-1}$
are homotopic to the identity by a $\mathbb{Z}_k$-equivariant homotopy,
we conclude that the restriction to
$\big( \{ \overline{(F_1)^{\sharp k}} \leq b \} \,,\,
\{ \overline{(F_1)^{\sharp k}} \leq a \} \big)$
of $\overline{(\Phi_{1, 1})_k}$
is homotopic to the identity by a $\mathbb{Z}_k$-equivariant homotopy,
and thus $\overline{(\Phi_{1, 1})_k}$ induces the identity.
\end{proof}

\begin{rmk}
A category theoretical definition similar to the one
for $G_{\mathbb{Z}_k, \ast}^{(a, b]} \big(\{\varphi_t\}\big)$
in principle can be given
also for single Hamiltonian symplectomorphisms,
but in the case of Hamiltonian symplectomorphisms
it is not be clear to us
that the analogue
of \autoref{proposition: commuting isomorphism}
would hold.
\end{rmk}

We now show that the homology groups 
$G_{\mathbb{Z}_k, \ast}^{(a, b]} \big(\{\varphi_t\}\big)$
are invariant by conjugation,
in the sense that every compactly supported Hamiltonian isotopy
$\{\psi_t\}_{t \in [0, 1]}$
induces an isomorphism
from $G_{\mathbb{Z}_k, \ast}^{(a, b]} \big(\{\varphi_t\}\big)$
to $G_{\mathbb{Z}_k, \ast}^{(a, b]}
\big( \{ \psi_1 \circ \varphi_t \circ \psi_1^{-1} \} \big)$
as follow.
By \autoref{lemma: special},
there is a 2-parameter family
$F_t^{(s)}: \mathbb{R}^{2n} \times \mathbb{R}^N \rightarrow \mathbb{R}$
of special generating functions quadratic at infinity
for the 2-parameter family of Hamiltonian symplectomorphisms
$\{\psi_s \circ \varphi_t \circ \psi_s^{-1} \}$
such that $F_0^{(s)}$ is a non-degenerate quadratic form
for every $s$
and $(F_t^{(s)})_{\infty}$ does not depend on $s$ and $t$.
The critical points of $(F_1^{(s)})^{\sharp k}$
are in 1--1 correspondence with the fixed points of
$(\psi_s \circ \varphi_t \circ \psi_s^{-1})^k
= \psi_s \circ (\varphi_1)^k \circ \psi_s^{-1}$,
and the critical values are given by the symplectic actions.
Since the action spectrum is invariant by conjugation,
the set of critical values of $(F_1^{(s)})^{\sharp k}$
is thus independent of $s$.
If we assume that $a$ and $b$ are regular values of $(F_1^{(0)})^{\sharp k}$,
so that $G_{\mathbb{Z}_k, \ast}^{(a, b]} \big(\{\varphi_t\}\big)$
is defined,
we then conclude that $a$ and $b$
are regular values of $(F_1^{(s)})^{\sharp k}$ for all $s$.
In particular,
$G_{\mathbb{Z}_k, \ast}^{(a, b]} \big( \{ \psi_1 \circ \varphi_t \circ \psi_1^{-1} \} \big)$
is also defined.
Moreover,
we can apply to the family of functions $\overline{(F_1^{(s)})^{\sharp k}}$
the following $\mathbb{Z}_k$-equivariant version
of \cite[Lemma 2.17]{San11} and \cite[Lemma 4.14]{lens}
(which is written more generally for time dependent regular values
for later purposes).

\begin{lemma}\label{lemma: continuation}
Let $B$ be a compact manifold
and $f_t: E = B \times \mathbb{R}^N \rightarrow \mathbb{R}$,
for $t \in [0, 1]$,
a 1-parameter family of functions
that are invariant with respect to a $\mathbb{Z}_k$-action on $E$.
Fix a complete $\mathbb{Z}_k$-invariant Riemannian metric on $E$.
Suppose that $\dot{f}_t$ is bounded,
that the norm of the gradient of each $f_t$
is bounded away from zero outside a compact set,
and that $a_t$ and $b_t$ are 1-parameter families of real numbers
such that, for every $t$, $a_t$ and $b_t$ are regular values of $f_t$.
Then there is a $\mathbb{Z}_k$-equivariant isotopy
$\{\theta_t\}$ of $E$
such that $\theta_t \big( \{ f_0  \leq a_0 \} \big) = \{ f_t  \leq a_t \}$
and $\theta_t \big( \{ f_0  \leq b_0 \} \big) = \{ f_t  \leq b_t \}$.
\end{lemma}

\begin{proof}
We follow the proof of \cite[Lemma 4.14]{lens},
adapting it to our situation.
Consider the open subset
$\mathcal{O} = \big\{\, (e,t) \,\big\lvert\, df_t |_e \neq 0 \,\big\}$
of $E \times [0,1]$,
which by assumption contains
$\{ (e,t) \ | \ f_t (e) = a_t \}$ and $\{ (e,t) \ | \ f_t (e) = b_t \}$.
For every $t \in [0,1]$ we consider the vector field
$u_t := \nabla f_t / \| \nabla f_t \|^2$
on $\mathcal{O}_t := \{ e \in E \ | \ df_t|_e \neq 0\}$.
Let $\rho : E \times [0,1] \to \mathbb{R}$
be a smooth $\mathbb{Z}_k$-invariant cut-off function
that is supported in $\mathcal{O}$
and is equal to~$1$ on a neighborhood $\mathcal{O}'$
of $\{ (e,t) \ | \ f_t (e) = a_t \} \cup \{ (e,t) \ | \ f_t (e) = b_t \}$,
and consider the time-dependent vector field $\{X_t\}_{t \in [0,1]}$ on $E$ 
that is given by 
$$
X_t (e) = \rho(e,t) \, \big( \dot{a}_t - \dot{f}_t (e) \big) \, u_t (e)
$$
for $(e,t)$ in $\mathcal{O}$
and that vanishes for $(e,t)$ outside $\mathcal{O}$.
For $(e,t)$ in $\mathcal{O}$
we have
$$
\lVert X_t (e) \rVert
= \frac{\rho(e,t) \, \big| \dot{a}_t - \dot{f}_t (e) \big|}{\lVert \nabla f_t \rVert} \,.
$$
Since by assumption $\dot{f}_t$ and $\frac{1}{\lVert \nabla f_t \rVert}$
are bounded,
we deduce that $\lVert X_t \rVert$ is bounded
and so the flow $\{\theta_t\}$ of $X_t$
is defined for $t \in [0, 1]$.
Moreover we have
$$
\frac{d}{dt} \, f_t \big( \theta_t (e) \big) 
= \dot{f}_t \big( \theta_t (e) \big) + df_t (X_t) \big( \theta_t(e) \big)
= \dot{f}_t \big( \theta_t (e) \big)
+ \rho( \theta_t (e) , t ) \Big( \dot{a}_t - \dot{f}_t \big(\theta_t (e)\big) \Big) \,,
$$
thus in particular
$\frac{d}{dt} f_t \big( \theta_t (e) \big) = \dot{a}_t$
if  $\big( \theta_t (e) , t \big) \in \mathcal{O}'$
and so each $\theta_t$ sends $\{ f_0 = a_0 \}$ to $\{ f_t = a_t \}$
and $\{ f_0 = b_0 \}$ to $\{ f_t = b_t \}$.
Since $\theta_0$ is the identity,
by continuity the isotopy $\theta_t$
sends $\{ f_0 \leq a_0 \}$ to $\{ f_t \leq a_t \}$
and $\{ f_0 \leq b_0 \}$ to $\{ f_t \leq b_t \}$ for all $t$.
Moreover, by construction, it is $\mathbb{Z}_k$-equivariant.
\end{proof}

Since $(F_1^{(s)})_{\infty}$ does not depend on $s$,
$\frac{d}{ds} \, F_1^{(s)}$ is bounded
and so $\frac{d}{ds} \, \overline{(F_1^{(s)})^{\sharp k}}$ is bounded.
Moreover,
the norm of the gradient of each $\overline{(F_1^{(s)})^{\sharp k}}$
is bounded away from zero outside a compact set
since, by \autoref{proposition: extend to sphere},
the functions $\overline{(F_1^{(s)})^{\sharp k}}$
are quadratic at infinity.
We can thus apply \autoref{lemma: continuation}
to the 1-parameter family of functions $\overline{(F_1^{(s)})^{\sharp k}}$
and obtain a $\mathbb{Z}_k$-equivariant isotopy $\{\theta_s\}$
of $S^{2n} \times \mathbb{R}^{2n(k-1)} \times \mathbb{R}^{Nk}$
mapping the sublevel sets of $\overline{(F_1^{(0)})^{\sharp k}}$
at $a$ and $b$
to those of $\overline{(F_1^{(s)})^{\sharp k}}$.
In particular, this gives an isomorphism
\[
(\theta_1)_{\ast}:
G_{\mathbb{Z}_k, \ast}^{(a, b]} (F_t^{(0)})
\rightarrow G_{\mathbb{Z}_k, \ast}^{(a, b]} (F_t^{(1)}) \,,
\]
and so an isomorphism
\[\lambda_{F_t^{(s)}} :=
(i_{F_t^{(1)}})^{-1} \circ (\theta_1)_{\ast} \circ i_{F_t^{(0)}}:
G_{\mathbb{Z}_k, \ast}^{(a, b]} \big(\{\varphi_t\}\big)
\rightarrow
G_{\mathbb{Z}_k, \ast}^{(a, b]}
\big(\{\psi_1 \circ \varphi_t \circ \psi_1^{-1} \} \big) \,.
\]
This isomorphism does not depend on the choice of the cut-off function $\rho$
in the proof of \autoref{lemma: continuation},
and so it is uniquely determined
by the 1-parameter family of functions $\overline{(F_1^{(s)})^{\sharp k}}$,
hence (as suggested by the notation)
by the 2-parameter family of functions $F_t^{(s)}$.
By a slight abuse of language
we also say that $\{\theta_s\}$ is \emph{the} $\mathbb{Z}_k$-equivariant isotopy
on $S^{2n} \times \mathbb{R}^{2n(k-1)} \times \mathbb{R}^{Nk}$
defined by the 1-parameter family of functions $\overline{(F_1^{(s)})^{\sharp k}}$.

\begin{prop}
\label{proposition: invariance by conjugation symplectic bis}
Let $\{\varphi_t\}_{t \in [0, 1]}$ and $\{\psi_t\}_{t \in [0, 1]}$
be compactly supported Hamiltonian isotopies
of $(\mathbb{R}^{2n}, \omega_0)$,
and suppose that $a \leq b$ in $\mathbb{R} \cup \{\pm \infty\}$
are not in the action spectrum of $(\varphi_1)^k$.
Then the isomorphism
\[
\lambda_{\{\psi_t\}}:
G_{\mathbb{Z}_k, \ast}^{(a, b]} \big(\{\varphi_t\}\big)
\rightarrow 
G_{\mathbb{Z}_k, \ast}^{(a, b]}
\big( \{ \psi_1 \circ \varphi_t \circ \psi_1^{-1} \} \big)
\]
defined by $\lambda_{\{\psi_t\}} = \lambda_{F_t^{(s)}}$
for any 2-parameter family $F_t^{(s)}$
of special generating functions quadratic at infinity
for the 2-parameter family of Hamiltonian symplectomorphims
$\{\psi_s \circ \varphi_t \circ \psi_s^{-1} \}$
is well-defined,
i.e.\ it does not depend on the choice of $F_t^{(s)}$.
\end{prop}

\begin{proof}
For any non-degenerate quadratic form
$Q: \mathbb{R}^{N'} \rightarrow \mathbb{R}$
we have $\lambda_{F_t^{(s)} \oplus Q} = \lambda_{F_t^{(s)}}$.
Indeed,
if $\{\theta_s\}$ is the $\mathbb{Z}_k$-equivariant isotopy
on $S^{2n} \times \mathbb{R}^{2n(k-1)} \times \mathbb{R}^{Nk}$
defined by the 1-parameter family of functions $\overline{(F_1^{(s)})^{\sharp k}}$
then $\{ (\theta_s, \id) \}$ is the $\mathbb{Z}_k$-equivariant isotopy
on $S^{2n} \times \mathbb{R}^{2n(k-1)} \times \mathbb{R}^{Nk} \times \mathbb{R}^{N'k}$
defined by the 1-parameter family of functions 
$\overline{(F_1^{(s)} \oplus Q)^{\sharp k}} = \overline{(F_1^{(s)})^{\sharp k}} \oplus Q^{\oplus k}$,
and by naturality of the $\mathbb{Z}_k$-equivariant Thom isomorphism
the diagram
\[
\begin{tikzcd}
{G_{\mathbb{Z}_k, \ast}^{(a, b]} \big(\{\varphi_t\}\big)}
\arrow[r, "i_{F_t^{(0)}}"] \arrow[rd, "i_{F_t^{(0)} \oplus Q}"' near end] 
& {G_{\mathbb{Z}_k, \ast}^{(a, b]} (F_t^{(0)})}
\arrow[r, "(\theta_1)_{\ast}"]
& {G_{\mathbb{Z}_k, \ast}^{(a, b]} (F_t^{(1)})}
& {G_{\mathbb{Z}_k, \ast}^{(a, b]} \big( \{\psi_1 \circ \varphi_t \circ \psi_1^{-1}\} \big)}
\arrow[l, "i_{F_t^{(1)}}"']
\arrow[ld, "i_{F_t^{(1)} \oplus Q}" near end] \\
& {G_{\mathbb{Z}_k, \ast}^{(a, b]} (F_t^{(0)} \oplus Q)}
\arrow[u] 
\arrow[r, "(\theta_1 \times \id)_{\ast}"] 
& {G_{\mathbb{Z}_k, \ast}^{(a, b]} (F_t^{(1)} \oplus Q)}
\arrow[u]
&  
\end{tikzcd}
\]
commutes.
By \autoref{lemma: lemmas gf} (iii)
it thus remains to show that
if $F_t^{(s)}$ and $\underline{F}_t^{(s)}$
are 2-parameter families
of special generating functions quadratic at infinity
for $\{\psi_s \circ \varphi_t \circ \psi_s^{-1} \}$
with $\overline{\underline{F}_t^{(s)}} = \overline{F_t^{(s)}} \circ \overline{\Phi_{s, t}}$
for a 2-parameter family of fibre preserving diffeomorphisms $\overline{\Phi_{s, t}}$
then $\lambda_{F_t^{(s)}} = \lambda_{\underline{F}_t^{(s)}}$.
But this follows from the fact that
if $\{\theta_s\}$ and $\{\underline{\theta}_s\}$
are the $\mathbb{Z}_k$-equivariant isotopies
on $S^{2n} \times \mathbb{R}^{2n(k-1)} \times \mathbb{R}^{Nk}$
defined by the families $\overline{(F_1^{(s)})^{\sharp k}}$
and $\overline{(\underline{F}_1^{(s)})^{\sharp k}}$ respectively
then the maps $\underline{\theta}_1 \circ (\overline{\Phi_{0, 1})_k}^{\,-1}$
and $(\overline{\Phi_{1, 1})_k}^{\,-1} \circ \theta_1$
are homotopic through the $\mathbb{Z}_k$-equivariant homotopy
$\underline{\theta}_1 \circ (\underline{\theta}_s)^{-1}
\circ (\overline{\Phi_{s, 1})_k}^{\,-1} \circ \theta_s$,
and so the diagram
\[
\begin{tikzcd}
{G_{\mathbb{Z}_k, \ast}^{(a, b]} \big(\{\varphi_t\}\big)}
\arrow[r, "i_{F_t^{(0)}}"]
\arrow[rd, "i_{\underline{F}_t^{(0)}}"' near end] 
& {G_{\mathbb{Z}_k, \ast}^{(a, b]} (F_t^{(0)})}
\arrow[r, "(\theta_1)_{\ast}"]                     
& {G_{\mathbb{Z}_k, \ast}^{(a, b]} (F_t^{(1)})}      
& {G_{\mathbb{Z}_k, \ast}^{(a, b]} \big( \{\psi_1 \circ \varphi_t \circ \psi_1^{-1}\} \big)}
\arrow[l, "i_{F_t^{(1)}}"'] \arrow[ld, "i_{\underline{F}_t^{(1)}}" near end] \\
& {G_{\mathbb{Z}_k, \ast}^{(a, b]} (\underline{F}_t^{(0)})} \arrow[u] 
\arrow[r, "(\underline{\theta}_1)_{\ast}"] 
& {G_{\mathbb{Z}_k, \ast}^{(a, b]} (\underline{F}_t^{(1)})} \arrow[u] &         
\end{tikzcd}\] 
commutes.
\end{proof}

We now define the $\mathbb{Z}_k$-equivariant
generating function homology
$G_{\mathbb{Z}_k, \ast}^{(a, b]} (\mathcal{U})$
of a domain $\mathcal{U}$
of $(\mathbb{R}^{2n}, \omega_0)$.
Let $\mathcal{H}^k_{a, b} (\mathcal{U})$
be the set of compactly supported
Hamiltonian isotopies $\{\varphi_t\}_{t \in [0, 1]}$
supported in $\mathcal{U}$
such that $a$ and $b$ do not belong
to the action spectrum of $(\varphi_1)^k$.
Define a partial order $\leq$
on $\mathcal{H}^k_{a, b} (\mathcal{U})$
by posing $\{\varphi_t^{(1)}\} \leq \{\varphi_t^{(2)}\}$
if $\{\varphi_t^{(2)} \circ (\varphi_t^{(1)})^{-1}\}$
is a non-negative compactly supported Hamiltonian isotopy.
For $\{\varphi_t^{(i)}\}$ and $\{\varphi_t^{(j)}\}$
in $\mathcal{H}^k_{a, b} (\mathcal{U})$
with $\{\varphi_t^{(i)}\} \leq \{\varphi_t^{(j)}\}$
we define an induced homomorphism
\[
\mu_i^j: 
G_{\mathbb{Z}_k, \ast}^{(a, b]} \big(\{\varphi_t^{(j)}\}\big)
\rightarrow 
G_{\mathbb{Z}_k, \ast}^{(a, b]} \big(\{\varphi_t^{(i)}\}\big)
\]
as follows.
Consider the non-negative Hamiltonian isotopy
$\{\varphi_t\}$ defined by
$\varphi_t = \varphi_t^{(j)} \circ (\varphi_t^{(i)})^{-1}$.
By \autoref{lemma: monotonicity special},
there is a 2-parameter family $F_t^{(s)}$
of special generating functions quadratic at infinity
for the 2-parameter family
$\{\varphi_{st} \circ \varphi_t^{(i)}\}_{s \in [0,1]}$
of Hamiltonian symplectomorphisms
such that $F_0^{(s)}$ is a non-degenerate quadratic form for every $s$
and $\frac{d}{ds} F_t^{(s)} \geq 0$;
in particular, $F_t^{(0)}$ and $F_t^{(1)}$
are 1-parameter families
of special generating functions quadratic at infinity
for $\{\varphi_t^{(i)}\}$ and $\{\varphi_t^{(j)}\}$ respectively
with $F_t^{(0)} \leq F_t^{(1)}$,
and so $\overline{(F_t^{(0)})^{\sharp k}} \leq \overline{(F_t^{(1)})^{\sharp k}}$.
The inclusion of the sublevel sets at $a$ and $b$
of $\overline{(F_1^{(1)})^{\sharp k}}$ into those of $\overline{(F_1^{(0)})^{\sharp k}}$
induces a homomorphism
\[
(i_{F_t^{(1)}, F_t^{(0)}})_{\ast}:
G_{\mathbb{Z}_k, \ast}^{(a, b]} (F_t^{(1)})
\rightarrow 
G_{\mathbb{Z}_k, \ast}^{(a, b]} (F_t^{(0)}) \,,
\]
and so a homomorphism
\begin{equation}\label{equation: monotonicity homomorphism symplectic}
\mu_i^j :=
(i_{F_t^{(0)}})^{-1} \circ (i_{F_t^{(1)}, F_t^{(0)}})_{\ast}
\circ i_{F_t^{(1)}}:
G_{\mathbb{Z}_k, \ast}^{(a, b]} \big( \{\varphi_t^{(j)} \}\big)
\rightarrow 
G_{\mathbb{Z}_k, \ast}^{(a, b]} \big(\{\varphi_t^{(i)} \}\big) \,,
\end{equation}
which can be seen to be well-defined
(i.e.\ independent of the choice of $F_t^{(s)}$)
by an argument similar to the one
in the proof of
\autoref{proposition: invariance by conjugation symplectic bis}.
The monotonicity homomorphisms $\mu_i^j$ satisfy the cocycle conditions
\[
\begin{cases}
\mu_i^{i} = \id \\
\mu_i^{j} \circ \mu_j^l = \mu_i^l
\quad \text{ for }
\{\varphi_t^{(i)}\} \leq \{ \varphi_t^{(j)} \}
\leq \{ \varphi_t^{(l)} \} \,.
\end{cases}
\]
Thus $\big\{ G_{\mathbb{Z}_k,*}^{(a,b]} \big( \{\varphi_t\} \big)
\big\}_{ \{\varphi_t\} \in \mathcal{H}^k_{a, b} (\mathcal{U}) }$
is an inversely directed family
of graded modules over $\mathbb{Z}_k$,
and we define
\[
G_{\mathbb{Z}_k,*}^{(a,b]} (\mathcal{U}) =
\underset{}{\varprojlim} \;
\Big\{ G_{\mathbb{Z}_k,*}^{(a,b]} \big( \{\varphi_t\} \big)
\Big\}_{ \{\varphi_t\} \in \mathcal{H}^k_{a, b} (\mathcal{U}) } \,.
\]

We now show that these groups are symplectic invariants.
Notice first that for any domain $\mathcal{U}$
and any compactly supported Hamiltonian isotopy $\{\psi_t\}_{t \in [0, 1]}$,
if $\{\varphi_t\} \in \mathcal{H}_{a, b}^k (\mathcal{U})$
then $\{\psi_1 \circ \varphi_t \circ \psi_1^{-1}\}
\in \mathcal{H}_{a, b}^k \big(\psi_1(\mathcal{U})\big)$,
and recall from \autoref{proposition: invariance by conjugation symplectic bis}
that we have an isomorphism
\[
\lambda_{\{\psi_t\}}:
G_{\mathbb{Z}_k, \ast}^{(a, b]} \big(\{\varphi_t\}\big)
\rightarrow 
G_{\mathbb{Z}_k, \ast}^{(a, b]}
\big( \{ \psi_1 \circ \varphi_t \circ \psi_1^{-1} \} \big) \,.
\]

\begin{prop}
For every domain $\mathcal{U}$ of $(\mathbb{R}^{2n}, \omega_0)$
and every compactly supported Hamiltonian isotopy
$\{\psi_t\}_{t \in [0, 1]}$,
the isomorphisms
$\lambda_{\{\psi_t\}}:
G_{\mathbb{Z}_k, \ast}^{(a, b]} \big(\{\varphi_t\}\big)
\rightarrow 
G_{\mathbb{Z}_k, \ast}^{(a, b]}
\big( \{ \psi_1 \circ \varphi_t \circ \psi_1^{-1} \} \big)$
for $\{\varphi_t\} \in \mathcal{H}_{a, b}^k (\mathcal{U})$
induce a well-defined isomorphism
\[
\lambda_{\{\psi_t\}}:
G_{\mathbb{Z}_k,*}^{(a,b]} (\mathcal{U}) \rightarrow
G_{\mathbb{Z}_k,*}^{(a,b]} \big(\psi_1(\mathcal{U})\big) \,.
\]
\end{prop}

\begin{proof}
If $\{\varphi_t^{(0)}\}$ and $\{\varphi_t^{(1)}\}$
are elements of $\mathcal{H}_{a, b}^k (\mathcal{U})$
with $\{\varphi_t^{(0)}\} \leq \{ \varphi_t^{(1)} \}$
then $\{\psi_1 \circ \varphi_t^{(0)} \circ \psi_1^{-1}\}$
and $\{\psi_1 \circ \varphi_t^{(1)} \circ \psi_1^{-1}\}$
are elements of $\mathcal{H}_{a, b}^k \big(\psi_1(\mathcal{U})\big)$
with$\{ \psi_1 \circ \varphi_t^{(0)} \circ \psi_1^{-1} \}
\leq \{ \psi_1 \circ \varphi_t^{(1)} \circ \psi_1^{-1} \}$.
We show that the isomorphisms $\lambda_{\{\psi_t\}}$
commute with the monotonicity homomorphisms,
and so induce an isomorphism between the limits.
Consider the 2-parameter family
$\{ \psi_s \circ \varphi_t^{(0)} \circ \psi_s^{-1} \}$,
and the non-negative Hamiltonian isotopy
$\varphi_t = \varphi_t^{(1)} \circ (\varphi_t^{(0)})^{-1}$.
By \autoref{lemma: monotonicity special},
there is a 3-parameter family $F_{s, t}^{(u)}$
of special generating functions quadratic at infinity
for the 3-parameter family
$\{\psi_s \circ \varphi_{ut} \circ \varphi_t^{(0)} \circ \psi_s^{-1}\}$
such that $F_{s, 0}^{(u)}$ is a non-degenerate quadratic form for all $s$ and $u$,
$(F_{s, 0}^{(u)})_{\infty}$ does not depend on $s, t$ and $u$,
and $\frac{d}{du} \, F_{s, t}^{(u)} \geq 0$.
In particular,
$F_{0, t}^{(0)}$, $F_{1, t}^{(0)}$, $F_{0, t}^{(1)}$
and $F_{1, t}^{(1)}$
are 1-parameter families of special generating functions quadratic at infinity
respectively for $\{\varphi_t^{(0)}\}$,
$\{ \psi_1 \circ \varphi_t^{(0)} \circ \psi_1^{-1} \}$,
$\{\varphi_t^{(1)}\}$
and $\{ \psi_1 \circ \varphi_t^{(1)} \circ \psi_1^{-1} \}$
satisfying $F_{0, t}^{(0)} \leq F_{0, t}^{(1)}$
and $F_{1, t}^{(0)} \leq F_{1, t}^{(1)}$.
We then have a commutative diagram
\[
\begin{tikzcd}
{G_{\mathbb{Z}_k, \ast}^{(a, b]} (F_{0, t}^{(1)})}
\arrow[d, "(\theta_1^{(1)})_{\ast}"]
\arrow[r, "{(i_{F_{0, t}^{(1)}, F_{0, t}^{(0)}})_{\ast}}"] 
&[4em] {G_{\mathbb{Z}_k, \ast}^{(a, b]} (F_{0, t}^{(0)})}
\arrow[d, "(\theta_1^{(0)})_{\ast}"] \\
[1em] {G_{\mathbb{Z}_k, \ast}^{(a, b]} (F_{1, t}^{(1)})}
\arrow[r, "{(i_{F_{1, t}^{(1)}, F_{1, t}^{(0)}})_{\ast}}"]                             
&[4em] {G_{\mathbb{Z}_k, \ast}^{(a, b]} (F_{0, t}^{(0)}) \,,}                                 
\end{tikzcd}
\]
where the vertical arrows
are the isomorphisms,
defined as in the discussion before
\autoref{proposition: invariance by conjugation symplectic bis},
that are used to define the isomorphisms
$\lambda_{\{\psi_t\}}:
G_{\mathbb{Z}_k, \ast}^{(a, b]} \big(\{\varphi_t^{(1)}\}\big)
\rightarrow 
G_{\mathbb{Z}_k, \ast}^{(a, b]}
\big( \{ \psi_1 \circ \varphi_t^{(1)} \circ \psi_1^{-1} \} \big)$
and
$\lambda_{\{\psi_t\}}:
G_{\mathbb{Z}_k, \ast}^{(a, b]} \big(\{\varphi_t^{(0)}\}\big)
\rightarrow 
G_{\mathbb{Z}_k, \ast}^{(a, b]}
\big( \{ \psi_1 \circ \varphi_t^{(0)} \circ \psi_1^{-1} \} \big)$
respectively,
and the horizontal arrows are the homomorphisms
that are used to define
the monotonicity homomorphisms
\eqref{equation: monotonicity homomorphism symplectic}.
This shows that the isomorphisms $\lambda_{\{\psi_t\}}$
commute with the monotonicity homomorphisms,
as we wanted.
\end{proof}

If $\mathcal{U}_1 \subset \mathcal{U}_2$,
the inclusion of posets
$\mathcal{H}^k_{a, b} (\mathcal{U}_1)
\subset \mathcal{H}^k_{a, b} (\mathcal{U}_2)$
induces a homomorphism
\[
G_{\mathbb{Z}_k,*}^{(a,b]} (\mathcal{U}_2) \rightarrow G_{\mathbb{Z}_k,*}^{(a,b]} (\mathcal{U}_1) \,.
\]
Moreover, for
$\mathcal{U}_1 \subset \mathcal{U}_2 \subset \mathcal{U}_3$
and for any compactly supported Hamiltonian isotopy
$\{\psi_t\}_{t \in [0,1]}$
these homomorphisms fit into commutative diagrams
\[
\begin{tikzcd}
G_{\mathbb{Z}_k, \ast}^{(a, b]} (\mathcal{U}_3) \arrow{r} \arrow {rd} & G_{\mathbb{Z}_k, \ast}^{(a, b]} (\mathcal{U}_2) \arrow{d} \\
& G_{\mathbb{Z}_k, \ast}^{(a, b]} (\mathcal{U}_1)
\end{tikzcd}
\]
and
\[
\begin{tikzcd}
G_{\mathbb{Z}_k, \ast}^{(a, b]} (\mathcal{U}_2)
\arrow{r} \arrow{d}{\lambda_{\{\psi_t\}}}
& G_{\mathbb{Z}_k, \ast}^{(a, b]} (\mathcal{U}_1)
\arrow{d}{\lambda_{\{\psi_t\}}} \\
G_{\mathbb{Z}_k, \ast}^{(a, b]} \big(\psi_1(\mathcal{U}_2)\big)
\arrow{r}
& G_{\mathbb{Z}_k, \ast}^{(a, b]} \big(\psi_1(\mathcal{U}_1)\big) \,. 
\end{tikzcd}
\]


\section{\texorpdfstring{$\mathbb{Z}_k$}{}-invariant functions detecting translated \texorpdfstring{$k$}{}-chains of contactomorphisms}
\label{section: invariant generating functions contact}

We have seen that in the symplectic case
there is a composition formula that,
given a generating function
of a Hamiltonian symplectomorphism $\varphi$,
allows to obtain a generating function of $\varphi^k$
that is invariant by a natural $\mathbb{Z}_k$-action.
Such a composition formula does not seem to exist in the contact case.
Indeed,
if we had a generating function
of the $k$-th iteration $\phi^k$
of a contactomorphism $\phi$ of $(\mathbb{R}^{2n+1}, \xi_0)$
that is invariant by a $\mathbb{Z}_k$-action
then its set of critical points
would also be invariant by the $\mathbb{Z}_k$-action,
and would be in 1--1 correspondence
with the set of translated points of $\phi^k$.
But, there is no natural $\mathbb{Z}_k$-action
on the set of translated points of $\phi^k$.
Indeed, for consistency with the symplectic case
(for any compactly supported Hamiltonian symplectomorphism
$\varphi$ of $(\mathbb{R}^{2n}, \omega_0)$
we would like the $\mathbb{Z}_k$-action
on the set of translated points of $\widetilde{\varphi}$
to project to the $\mathbb{Z}_k$-action
on the set of fixed points of $\varphi$)
such an action would send a translated point $p$ of $\phi^k$
to a point in the same Reeb orbit as $\phi(p)$;
the natural choice would be to send $p$ to $\phi(p)$,
but in general $\phi(p)$ is not a translated point of $\phi^k$
(since $\phi$ does not necessarily preserve the contact form,
the fact that $p$ and $\phi^k(p)$ are in the same Reeb orbit
does not imply that their images by $\phi$
should also be in the same Reeb orbit).
On the other hand,
we have seen in the introduction
that there is a natural $\mathbb{Z}_k$-action
on the set of translated $k$-chains of $\phi$.
We thus construct in this section
$\mathbb{Z}_k$-invariant functions
that detect translated $k$-chains
of contactomorphisms of $(\mathbb{R}^{2n+1}, \xi_0)$.
These functions will be used
in \autoref{section: equivariant contact homology}
to define the $\mathbb{Z}_k$-equivariant homology
of domains of $(\mathbb{R}^{2n} \times S^1, \xi_0)$.

Recall that a translated $k$-chain of action $tk$
of a contactomorphism $\phi$ of $(\mathbb{R}^{2n + 1}, \xi_0)$
with respect to the contact form $\alpha_0$
is a $k$-tuple of points $(p_1, \dots, p_k)$
such that
\[
g(p_1) + \dots + g (p_k) = 0 \,,
\]
where $g$ denotes the conformal factor of $\phi$,
and
\[
p_{j+1} = \varphi_{-t}^{\alpha_0} \circ \phi \, (p_j)
\]
for all $j$,
with the convention $p_{k+1} = p_1$,
where $\{\varphi_t^{\alpha_0}\}$
denotes the Reeb flow.
Recall also that a discriminant point is a translated point
of action zero.

\begin{lemma}\label{lemma: translated chains vs discriminant points}
The translated $k$-chains of action $tk$
of a contactomorphism $\phi$ of $(\mathbb{R}^{2n + 1}, \xi_0)$
are in 1--1 correspondence with the discriminant points
of $(\varphi_{-t}^{\alpha_0} \circ \phi)^k$.
\end{lemma}

\begin{proof}
If $(p_1, \dots, p_k)$ is a translated $k$-chain
of action $tk$ of $\phi$
then $p_{j} = (\varphi_{-t}^{\alpha_0} \circ \phi)^{j-1} (p_1)$
for $j = 2, \dots, k$
and $p_1 = (\varphi_{-t}^{\alpha_0} \circ \phi)^k (p_1)$.
Conversely,
if $p_1 = (\varphi_{-t}^{\alpha_0} \circ \phi)^k (p_1)$
and we pose
$p_{j} = (\varphi_{-t}^{\alpha_0} \circ \phi)^{j-1} (p_1)$
for $j = 2, \dots, k$
then $(p_1, \dots, p_k)$ satisfies 
$p_{j+1} = \varphi_{-t}^{\alpha_0} \circ \phi (p_j)$
for all $j$.
Moreover,
the conformal factor
of $(\varphi_{-t}^{\alpha_0} \circ \phi)^k$
is $\sum_{j=0}^{k-1} g \circ (\varphi_{-t}^{\alpha_0} \circ \phi)^j$,
where $g$ denotes the conformal factor of $\phi$,
and thus it vanishes at $p_1$ if and only if $g(p_1) + \dots + g (p_k) = 0$.
We conclude that the map $(p_1, \dots, p_k) \mapsto p_1$
is a 1--1 correspondence between the translated $k$-chains of action $tk$ of $\phi$
and the discriminant points of $(\varphi_{-t}^{\alpha_0} \circ \phi)^k$.
\end{proof}

In order to construct functions that detect translated $k$-chains
our first step is thus to construct functions
that detect discriminant points
of $k$-th iterations of contactomorphisms.

Let $\phi$ be a contactomorphism of $(\mathbb{R}^{2n+1}, \xi_0)$
with generating function
$F: \mathbb{R}^{2n+1} \times \mathbb{R}^N
\rightarrow \mathbb{R}$.
Consider the function
\[
F^{\sharp k}: \mathbb{R}^{(2n+2)k} \times \mathbb{R}^{Nk}
\rightarrow \mathbb{R}
\]
defined by
\begin{gather*}
  F^{\sharp k}
(x_1, y_1, \theta_1, r_1, \dots, x_k, y_k, \theta_k, r_k,
\zeta_1, \dots, \zeta_k)\\
= \sum_{j = 1}^k e^{r_j}
F \Big( e^{- \frac{r_j}{2}} \, \frac{x_j + x_{j+1}}{2} \,,\,
e^{- \frac{r_j}{2}} \, \frac{y_j + y_{j+1}}{2} \,,\,
\theta_{j+1} \,,\, \zeta_j \Big)\\
+ \frac{1}{2} \, (x_j y_{j+1} - x_{j+1} y_j)
+ e^{r_{j-1}} (\theta_{j} - \theta_{j+1}) \,,
\end{gather*}
with the usual cyclic convention on the indices.
For any $a \in \mathbb{R}$
we have
\begin{gather*}
F^{\sharp k}
(e^{\frac{a}{2}} x_1, e^{\frac{a}{2}} y_1, \theta_1, r_1 + a,
\dots,
e^{\frac{a}{2}} x_k, e^{\frac{a}{2}} y_k, \theta_k, r_k + a,
\zeta_1, \dots, \zeta_k)\\
= e^{a} \, F^{\sharp k}
(x_1, y_1, \theta_1, r_1, \dots, x_k, y_k, \theta_k, r_k, \zeta_1, \dots, \zeta_k) \,.  
\end{gather*}
The set of critical points of $F^{\sharp k}$
is thus invariant by the $\mathbb{R}$-action
\begin{gather}\label{equation: action R}
\begin{gathered}
(x_1, y_1, \theta_1, r_1, \dots, x_k, y_k, \theta_k, r_k, \zeta_1, \dots, \zeta_k) \\
\mapsto
(e^{\frac{a}{2}} x_1, e^{\frac{a}{2}} y_1, \theta_1, r_1 + a,
\dots,
e^{\frac{a}{2}} x_k, e^{\frac{a}{2}} y_k, \theta_k, r_k + a,
\zeta_1, \dots, \zeta_k) \,,
\end{gathered}
\end{gather}
and all the critical points of $F^{\sharp k}$
have critical value equal to zero.

\begin{prop}\label{proposition: composition formula contact}
The $\mathbb{R}$-orbits of critical points of $F^{\sharp k}$
are in 1--1 correspondence
with the discriminant points of $\phi^k$.
\end{prop}

\begin{proof}
Let $(x_1, y_1, \theta_1, r_1, \dots, x_k, y_k, \theta_k, r_k,
\zeta_1, \dots, \zeta_k)$
be a critical point of $F^{\sharp k}$.
For every $j$ we have
\begin{align*}
0 &= \frac{\partial F^{\sharp k}}{\partial \zeta_j}
(x_1, y_1, \theta_1, r_1, \dots, x_k, y_k, \theta_k, r_k,
\zeta_1, \dots, \zeta_k) \\
&= e^{r_j} \frac{\partial F}{\partial \zeta_j}
\Big( e^{- \frac{r_j}{2}} \, \frac{x_j + x_{j+1}}{2} \,\,
e^{- \frac{r_j}{2}} \, \frac{y_j + y_{j+1}}{2} \,,\,
\theta_{j+1} \,,\, \zeta_j \Big) \,,
\end{align*}
thus $\big( e^{- \frac{r_j}{2}} \, \frac{x_j + x_{j+1}}{2} \,\,
e^{- \frac{r_j}{2}} \, \frac{y_j + y_{j+1}}{2} \,,\,
\theta_{j+1} \,,\, \zeta_j \big)$
is a fibre critical point of $F$.
Define
\begin{equation}\label{equation: change of coordinates contact}
\begin{cases}
X_j = e^{- \frac{r_j}{2}} \, \frac{x_j + x_{j+1}}{2} \\
Y_j = e^{- \frac{r_j}{2}} \, \frac{y_j + y_{j+1}}{2} \\
\Theta_j = \theta_{j+1} \,,
\end{cases}
\end{equation}
and let $(\overline{X_j}, \overline{Y}_j, \overline{\Theta}_j)
\in \mathbb{R}^{2n+1}$
be the image of $(X_j, Y_j, \Theta_j, \zeta_j)$
by the inverse of the diffeomorphism
\eqref{equation: diffeomorphism gf contact}.
Then
\begin{gather*}
\begin{cases}
X_j =  \frac{e^{\frac{1}{2} g (\overline{X}_j, \overline{Y}_j, \overline{\Theta}_j)} \overline{X}_j
+ \phi_x (\overline{X}_j, \overline{Y}_j, \overline{\Theta}_j)}{2} \\
Y_j = \frac{e^{\frac{1}{2} g (\overline{X}_j, \overline{Y}_j, \overline{\Theta}_j)} \overline{Y}_j
+ \phi_y (\overline{X}_j, \overline{Y}_j, \overline{\Theta}_j)}{2} \\
\Theta_j = \overline{\Theta}_j \,,
\end{cases}\\
\begin{cases}
d_1F (X_j, Y_j, \Theta_j, \zeta_j)
= e^{\frac{1}{2} g (\overline{X}_j, \overline{Y}_j, \overline{\Theta}_j)} \overline{Y}_j
- \phi_y (\overline{X}_j, \overline{Y}_j, \overline{\Theta}_j)\\
d_2F (X_j, Y_j, \Theta_j, \zeta_j)
= \phi_x (\overline{X}_j, \overline{Y}_j, \overline{\Theta}_j)
- e^{\frac{1}{2} g (\overline{X}_j, \overline{Y}_j, \overline{\Theta}_j)} \overline{X}_j \\
d_3F (X_j, Y_j, \Theta_j, \zeta_j)
= e^{g (\overline{X}_j, \overline{Y}_j, \overline{\Theta}_j)} - 1 
\end{cases}
\end{gather*}
and
\[
F (X_j, Y_j, \Theta_j, \zeta_j)
\]
\[
= \phi_{\theta} (\overline{X}_j, \overline{Y}_j, \overline{\Theta}_j)
- \overline{\Theta}_j
+ \frac{ e^{\frac{1}{2} g (\overline{X}_j, \overline{Y}_j, \overline{\Theta}_j)}
\big( \overline{Y}_j \phi_x (\overline{X_j}, \overline{Y}_j, \overline{\Theta}_j) 
- \overline{X}_j \phi_y (\overline{X}_j, \overline{Y}_j, \overline{\Theta}_j) \big) }{2} \,,
\]
where we denote $\phi (\overline{X}_j, \overline{Y}_j, \overline{\Theta}_j)
= \big( \phi_x (\overline{X}_j, \overline{Y}_j, \overline{\Theta}_j),
\phi_y (\overline{X}_j, \overline{Y}_j, \overline{\Theta}_j),
\phi_{\theta} (\overline{X}_j, \overline{Y}_j, \overline{\Theta}_j)\big)$.
For every $j$ we have
\begin{align*}
0 &= \frac{\partial F^{\sharp k}}{\partial \theta_{j+1}}
(x_1, y_1, \theta_1, r_1, \dots, x_k, y_k, \theta_k, r_k,
\zeta_1, \dots, \zeta_k) \\
&= e^{r_j} d_3F (X_j, Y_j, \Theta_j, \zeta_j)
+ e^{r_j} - e^{r_{j-1}} \\
&= e^{r_j + g (\overline{X}_j, \overline{Y}_j, \overline{\Theta}_j)} - e^{r_{j-1}} \,,
\end{align*}
thus
\begin{equation}\label{equation: condition derivatives Theta}
r_j + g (\overline{X}_j, \overline{Y}_j, \overline{\Theta}_j)
= r_{j-1} \,.
\end{equation}
In particular,
\begin{equation}\label{equation: g}
\sum_{j=1}^k g (\overline{X}_j, \overline{Y}_j, \overline{\Theta}_j) = 0 \,.
\end{equation}
Moreover,
\begin{align*}
0 = & \; \frac{\partial F^{\sharp k}}{\partial x_j}
(x_1, y_1, \theta_1, r_1, \dots, x_k, y_k, \theta_k, r_k,
\zeta_1, \dots, \zeta_k) \\
= & \; \frac{1}{2} \, e^{\frac{r_j}{2}} \,
d_1F (X_j, Y_j, \Theta_j, \zeta_j) + \frac{1}{2} \, e^{\frac{r_{j-1}}{2}} \, 
d_1F (X_{j-1}, Y_{j-1}, \Theta_{j-1}, \zeta_{j-1})
+ \frac{1}{2} (y_{j+1} - y_{j-1}) \\
= & \; \frac{1}{2} \, e^{\frac{r_j}{2}} \,
d_1F (X_j, Y_j, \Theta_j, \zeta_j)
+ \frac{1}{2} \, e^{\frac{r_{j-1}}{2}} \,
d_1F (X_{j-1}, Y_{j-1}, \Theta_{j-1}, \zeta_{j-1})
+ e^{\frac{r_j}{2}} Y_j - e^{\frac{r_{j-1}}{2}} Y_{j-1} \\
= & \; e^{\frac{r_j}{2} + \frac{1}{2} g (\overline{X}_j, \overline{Y}_j, \overline{\Theta}_j)}
\, \overline{Y}_j - e^{\frac{r_{j-1}}{2}} \, \phi_y (\overline{X}_{j-1}, \overline{Y}_{j-1}, \overline{\Theta}_{j-1})\,,
\end{align*}
thus, using \eqref{equation: condition derivatives Theta},
\[
 \overline{Y}_j = \phi_y (\overline{X}_{j-1}, \overline{Y}_{j-1}, \overline{\Theta}_{j-1}) \,.
\]
Similarly,
the vanishing of $\frac{\partial F^{\sharp k}}{\partial y_j}$
at $(x_1, y_1, \theta_1, r_1, \dots, x_k, y_k, \theta_k, r_k,
\zeta_1, \dots, \zeta_k)$
gives
\[
\overline{X}_j = \phi_x (\overline{X}_{j-1}, \overline{Y}_{j-1}, \overline{\Theta}_{j-1}) \,.
\]
Finally we have
\begin{gather*}
    0 = \frac{\partial F^{\sharp k}}{\partial r_{j-1}}
(x_1, y_1, \theta_1, r_1, \dots, x_k, y_k, \theta_k, r_k,
\zeta_1, \dots, \zeta_k)\\
= e^{r_{j-1}} F (X_{j-1}, Y_{j-1}, \Theta_{j-1}, \zeta_{j-1})
- \frac{1}{2} e^{r_{j-1}} X_{j-1}
\, d_1F (X_{j-1}, Y_{j-1}, \Theta_{j-1}, \zeta_{j-1})\\
- \frac{1}{2} \, e^{r_{j-1}} \, Y_{j-1} \,
d_2F (X_{j-1}, Y_{j-1}, \Theta_{j-1}, \zeta_{j-1})
+ e^{r_{j-1}} (\Theta_{j-1} - \Theta_j)\\
= e^{r_{j-1}}
\big( \phi_{\theta} (\overline{X}_{j-1}, \overline{Y}_{j-1}, \overline{\Theta}_{j-1}) - \overline{\Theta}_j \big) \,,
\end{gather*}
thus
\[
\overline{\Theta}_j
= \phi_{\theta} (\overline{X}_{j-1}, \overline{Y}_{j-1}, \overline{\Theta}_{j-1}) \,.
\]
We conclude that
\[
(\overline{X}_j, \overline{Y}_j, \overline{\Theta}_j)
= \phi (\overline{X}_{j-1}, \overline{Y}_{j-1}, \overline{\Theta}_{j-1})
\]
for all $j$,
and so, since \eqref{equation: g} also holds,
$(\overline{X}_1, \overline{Y}_1, \overline{\Theta}_1)$
is a discriminant point of $\phi^k$.
The map
\[
(x_1, y_1, \theta_1, r_1, \dots, x_k, y_k, \theta_k, r_k,
\zeta_1, \dots, \zeta_k)
\mapsto (\overline{X}_1, \overline{Y}_1, \overline{\Theta}_1)
\]
gives thus a 1--1 correspondence
between the $\mathbb{R}$-orbits
of critical points of $F^{\sharp k}$
and the discriminant points of $\phi^k$.
\end{proof}

\begin{rmk}
The discriminant points of $\phi^k$
are in 1--1 correspondence with the $\mathbb{R}$-orbits
of fixed points of the lift of $\phi^k$ to the symplectization.
The function $F^{\sharp k}$ can be obtained
by applying a version of the composition formula
of \autoref{proposition: composition formula}
to a generating function of the lift of $\phi$
to the symplectization,
after appropriate identifications
of the symplectization of $(\mathbb{R}^{2n+1}, \xi_0)$
with $(\mathbb{R}^{2n+2}, \omega_0)$
and of the symplectic product of $(\mathbb{R}^{2n+2}, \omega_0)$
with $(T^{\ast}\mathbb{R}^{2n+2}, \omega_{\can})$.
\end{rmk}

Let $\phi$ be a contactomorphism of $(\mathbb{R}^{2n+1}, \xi_0)$,
and let
$F: \mathbb{R}^{2n+1} \times \mathbb{R}^N \rightarrow \mathbb{R}$
be a generating function of $\phi$.
By \autoref{remark: gf Reeb},
for any $t \in \mathbb{R}$
the function
\[
F_{t}: \mathbb{R}^{2n+1} \times \mathbb{R}^N \rightarrow \mathbb{R} \,,\;
F_{t} (x, y, \theta, \zeta)
= F (x, y, \theta, \zeta) - t
\]
is a generating function
of $\varphi_{-t}^{\alpha_0} \circ \phi$,
where $\{\varphi_t^{\alpha_0}\}$
denotes the Reeb flow.
By \autoref{lemma: translated chains vs discriminant points}
and \autoref{proposition: composition formula contact},
the $\mathbb{R}$-orbits of critical points
of the function
\[
(F_{t})^{\sharp k}:
\mathbb{R}^{(2n+2)k} \times \mathbb{R}^{Nk} \rightarrow \mathbb{R}
\]
are in 1--1 correspondence
with the translated $k$-chains of $\phi$ of action $tk$.
Consider now the function
\[
G_F^{(k)}: (\mathbb{R}^{(2n+2)k} \times \mathbb{R}^{Nk})
\times \mathbb{R} \rightarrow \mathbb{R} \,,\;
G_F^{(k)} (p, t) = (F_{t/k})^{\sharp k} (p) \,,
\]
and the function $P_F^{(k)}$ defined by the diagram
\[
\begin{tikzcd}
\big(G_F^{(k)}\big)^{-1}(0) \arrow{r}
\arrow[rd, "P_F^{(k)}" ']
& (\mathbb{R}^{(2n+2)k} \times \mathbb{R}^{Nk})
\times \mathbb{R}
\arrow{r}{G_F^{(k)}} \arrow{d} & \mathbb{R} \\
& \mathbb{R}  &
\end{tikzcd}
\]
where the first horizontal arrow is the inclusion
and the vertical arrow is the projection on the last factor.
The function $G_F^{(k)}$ has no critical points,
in particular $0$ is a regular value.
This implies that,
for any $(p, t) \in \big(G_F^{(k)}\big)^{-1}(0)$,
$(p, t)$ is a critical point of $P_F^{(k)}$
if and only if $p$
is a critical point of $(F_{t/k})^{\sharp k}$.
The set of critical points of $P_F^{(k)}$
is thus also invariant by the $\mathbb{R}$-action \eqref{equation: action R},
and the $\mathbb{R}$-orbits of critical points of $P_F^{(k)}$
are in 1--1 correspondence with the translated $k$-chains of $\phi$.

For $(p, t) \in \big(G_F^{(k)}\big)^{-1}(0)$ we have
\[
t = \frac{k}{\sum_{j=1}^k e^{r_j}} \, F^{\sharp k} (p) \,.
\]
We define
\[
\mathcal{P}_F^{(k)}: 
\mathbb{R}^{(2n+2)k} \times \mathbb{R}^{Nk}
\rightarrow \mathbb{R}
\]
to be the composition of $P_F^{(k)}$
with the diffeomorphism
\[
\mathbb{R}^{(2n+2)k} \times \mathbb{R}^{Nk}
\rightarrow \big(G_F^{(k)}\big)^{-1}(0) \;,\quad
p \mapsto \Big( p \,,\, \frac{k}{\sum_{j=1}^k e^{r_j}} \, F^{\sharp k} (p) \Big) \,,
\]
thus
\begin{gather*}
\mathcal{P}_F^{(k)}
(x_1, y_1, \theta_1, r_1, \dots, x_k, y_k, \theta_k, r_k, \zeta_1, \dots, \zeta_k) \\   
= \frac{k}{\sum_{j=1}^k e^{r_j}} \;
\sum_{j = 1}^k e^{r_j}
F \Big( e^{- \frac{r_j}{2}} \, \frac{x_j + x_{j+1}}{2} \,,\,
e^{- \frac{r_j}{2}} \, \frac{y_j + y_{j+1}}{2} \,,\,
\theta_{j+1} \,,\, \zeta_j \Big)\\
+ \frac{1}{2} \, (x_j y_{j+1} - x_{j+1} y_j)
+ e^{r_{j-1}} (\theta_{j} - \theta_{j+1}) \,.
\end{gather*}
The function $\mathcal{P}_F^{(k)}$ is invariant
by the $\mathbb{R}$-action \eqref{equation: action R},
and by the above discussion we have
the following result.

\begin{prop}\label{proposition: critical points contact}
The $\mathbb{R}$-orbits
of critical points of $\mathcal{P}_F^{(k)}$
are in 1--1 correspondence
with the translated $k$-chains of $\phi$,
with critical values given by the action.
\end{prop}

The function $\mathcal{P}_F^{(k)}$
is invariant by the action of $\mathbb{Z}_k$
generated by the map
\begin{gather}\label{equation: action contact}
\begin{gathered}
\underline{\sigma}_k:
(x_1, y_1, \theta_1, r_1, \dots, x_k, y_k, \theta_k, r_k, \zeta_1, \dots, \zeta_k) \\
\mapsto (x_2, y_2, \theta_2, r_2, \dots, x_k, y_k, \theta_k, r_k, x_1, y_1, \theta_1, r_1,
\zeta_2, \dots, \zeta_k, \zeta_1) \,. 
\end{gathered}
\end{gather}

Using the notations of the proof
of \autoref{proposition: composition formula contact},
$(x_1, y_1, \theta_1, r_1, \dots, x_k, y_k, \theta_k, r_k, \zeta_1, \dots, \zeta_k)$
is a critical point of $\mathcal{P}_F^{(k)}$
if and only if
$\big( (\overline{X}_1, \overline{Y}_1, \overline{\Theta}_1),
\dots, (\overline{X}_k, \overline{Y}_k, \overline{\Theta}_k) \big)$
is a translated $k$-chain of $\phi$.
It follows from the proof
of \autoref{proposition: composition formula contact}
that, under this bijection between
the $\mathbb{R}$-orbits of critical points
of $\mathcal{P}_F^{(k)}$
and the translated $k$-chains of $\phi$,
the $\mathbb{Z}_k$-action \eqref{equation: action contact}
corresponds to the $\mathbb{Z}_k$-action
on the set of translated $k$-chains of $\phi$
generated by the map that sends
a translated $k$-chain $(p_1, \dots, p_k)$
to the translated $k$-chain $(p_2, \dots, p_k, p_1)$.

Consider now the change of variables $B$
of $\mathbb{R}^{(2n+2)k} \times \mathbb{R}^{Nk}$
defined by
\[
B (x_1, y_1, \theta_1, r_1, \dots, x_k, y_k, \theta_k, r_k, \zeta_1, \dots, \zeta_k)
\]
\[
= \Big(\, \Big( \frac{k}{\sum_{j=1}^k e^{r_j}} \Big)^{\frac{1}{2}} \,\cdot\, \frac{1}{k} \, \sum_{j=1}^k x_j \,,\,
\Big( \frac{k}{\sum_{j=1}^k e^{r_j}} \Big)^{\frac{1}{2}} \,\cdot\, \frac{1}{k} \, \sum_{j=1}^k y_j \,,\,
\frac{1}{k} \, \sum_{j=1}^k \theta_j \,,\, \frac{1}{k} \, \sum_{j=1}^k r_j \,,\,
\]
\[
\Big( \frac{k}{\sum_{j=1}^k e^{r_j}} \Big)^{\frac{1}{2}} \,\cdot\, \frac{1}{k} \, (x_2 - x_1)
\,,\,
\Big( \frac{k}{\sum_{j=1}^k e^{r_j}} \Big)^{\frac{1}{2}} \,\cdot\, \frac{1}{k} \, (y_2 - y_1)
\,,\, \frac{1}{k} \, (\theta_2 - \theta_1) \,,\, \frac{1}{k} \, (r_2 - r_1) \,,\, \dots \,,
\]
\[
\Big( \frac{k}{\sum_{j=1}^k e^{r_j}} \Big)^{\frac{1}{2}} \,\cdot\, \frac{1}{k} \, (x_k - x_1)
\,,\,
\Big( \frac{k}{\sum_{j=1}^k e^{r_j}} \Big)^{\frac{1}{2}} \,\cdot\, \frac{1}{k} \, (y_k - y_1)
\,,\, \frac{1}{k} \, (\theta_k - \theta_1) \,,\, \frac{1}{k} \, (r_k - r_1) \,,
\]
\[
\Big( \frac{k \, e^{r_1}}{\sum_{j=1}^k e^{r_j}} \Big)^{\frac{1}{2}} \,\cdot\, \zeta_1, \dots,
\Big( \frac{k \, e^{r_k}}{\sum_{j=1}^k e^{r_j}} \Big)^{\frac{1}{2}} \,\cdot\, \zeta_k \,\Big) \,.
\]
Observe that
\[
B \circ \underline{\sigma}_k \circ B^{-1}
(x_1, y_1, \theta_1, r_1, \dots, x_k, y_k, \theta_k, r_k, \zeta_1, \dots, \zeta_k)
\]
\[
= (x_1, y_1, \theta_1, r_1,
x_3 - x_2, y_3 - y_2, \theta_3 - \theta_2, r_3 - r_2,
\dots, x_k - x_2, y_k - y_2, \theta_k - \theta_2, r_k - r_2,
\]
\[
-x_2, -y_2, - \theta_2, - r_2, \zeta_2, \dots, \zeta_k, \zeta_1) \,,
\]
thus $B \circ \underline{\sigma}_k \circ B^{-1}$
extends to a map on
$S^{2n} \times \mathbb{R} \times \mathbb{R} \times \mathbb{R}^{(2n+2) (k-1)} \times \mathbb{R}^{Nk}$
that descends to a map $\overline{\underline{\sigma}_k}$
generating a $\mathbb{Z}_k$-action on 
$S^{2n} \times S^1 \times \mathbb{R}^{(2n+2) (k-1)} \times \mathbb{R}^{Nk}$.

\begin{prop}\label{proposition: extend to sphere contact}
Let $F: \mathbb{R}^{2n+1} \times \mathbb{R}^N \rightarrow \mathbb{R}$
be a generating function quadratic at infinity
of a compactly supported contactomorphism $\phi$
of $(\mathbb{R}^{2n} \times S^1, \xi_0)$
contact isotopic to the identity,
and consider the associated function
$\mathcal{P}_F^{(k)}: \mathbb{R}^{(2n+2)k} \times \mathbb{R}^{Nk} \rightarrow \mathbb{R}$.
Then $\mathcal{P}_F^{(k)} \circ B^{-1}$
extends to a continuous function
on $S^{2n} \times \mathbb{R}^2 \times \mathbb{R}^{(2n+2)(k-1)} \times \mathbb{R}^{Nk}$
that descends to a continuous function
\[
\overline{\mathcal{P}_F^{(k)}}:
S^{2n} \times S^1 \times \mathbb{R}^{(2n+2)(k-1)} \times \mathbb{R}^{Nk}
\rightarrow \mathbb{R} \,,
\]
invariant by the $\mathbb{Z}_k$-action
generated by $\overline{\underline{\sigma}_k}$.
Moreover,
if $F$ is special then $\overline{\mathcal{P}_F^{(k)}}$ is smooth
and the norm of its gradient
is bounded away from zero outside a compact set.
\end{prop}

\begin{proof}
For $e = (x_1, y_1, \theta_1, r_1, \dots, x_k, y_k, \theta_k, r_k, \zeta_1, \dots, \zeta_k)$
we have
\[
\mathcal{P}_F^{(k)} \circ B^{-1}(e)
= \sum_{j=1}^k R \cdot R_j \cdot F \big(P_j(e)\big)
+ \sum_{j = 2}^{k-1} \frac{k^2}{2} \, \left(x_j y_{j+1} - x_{j+1}y_j\right)
+ \sum_{j = 2}^k k \, (R \cdot R_{j-1} - R \cdot R_{j-2}) \, \theta_j \,,
\]
with the usual cyclic convention on the indices,
where $R$, $R_j$ and $P_j$ are given by the following expressions,
in which the index $l$ is always in the set $\{2, \dots, k\}$:
\begin{align*}
&R = R (r_2, \dots, r_k)
= \frac{k}{ e^{- \sum_l r_l} + \sum_{j=2}^k e^{(k-1)r_j - \sum_{l \neq j} r_l } } \,, \\
&R_1 = R_1 (r_2, \dots, r_k) = e^{- \sum_l r_l} \,, \\
&R_j = R_j (r_2, \dots, r_k) =  e^{(k-1)r_j - \sum_{l \neq j} r_l }
\quad \text{ for  } j=2, \dots, k \,,
\end{align*}
\[
P_1 (e) = \Big( (R \cdot R_1)^{-\frac{1}{2}} \big(x_1 + \frac{k-2}{2} \, x_2
- \sum_{l \neq 2} x_l \big) \,,\,
(R \cdot R_1)^{-\frac{1}{2}} \big(y_1 + \frac{k-2}{2} \, y_2 - \sum_{l \neq 2} y_l) \,,
\]
\[
\theta_1 + (k-1)\theta_2 - \sum_{l \neq 2} \theta_l \,,\,
(R \cdot R_1)^{-\frac{1}{2}} \, \zeta_1 \Big) \,,
\]
\[
P_j (e) = \Big( (R \cdot R_j)^{-\frac{1}{2}} \big(x_1 + \frac{k-2}{2} \, x_j + \frac{k-2}{2} \, x_{j+1}
- \sum_{l \neq j, j+1} x_l) \,,\,
\]
\[
(R \cdot R_j)^{-\frac{1}{2}} \big(y_1 + \frac{k-2}{2} \, y_j + \frac{k-2}{2} \, y_{j+1}
- \sum_{l \neq j, j+1} y_l) \,,\,
\theta_1 + (k-1) \theta_{j+1} - \sum_{l \neq j+1} \theta_l \,,\,
(R \cdot R_j)^{-\frac{1}{2}} \, \zeta_j \Big)
\]
for $j = 2, \dots, k-1$,
and
\[
P_k (e) = \Big( (R \cdot R_k)^{-\frac{1}{2}} \big(x_1 + \frac{k-2}{2} \, x_k - \sum_{l \neq k} x_l) \,,\,
(R \cdot R_k)^{-\frac{1}{2}} \big(y_1 + \frac{k-2}{2} \, y_k - \sum_{l \neq k} y_k) \,,\,
\]
\[
\theta_1 - \sum_{l} \theta_l \,,\,
(R \cdot R_k)^{-\frac{1}{2}} \, \zeta_k \Big) \,.
\]

The function $\mathcal{P}_F^{(k)} \circ B^{-1}$
extends to a continuous function $\underline{\mathcal{P}}_F^{(k)}$ on
$S^{2n} \times \mathbb{R}^2 \times \mathbb{R}^{(2n+2)(k-1)} \times \mathbb{R}^{Nk}$
by setting
\[
\underline{\mathcal{P}}_F^{(k)}
(p_{\infty}, \theta_1, r_1, x_2, y_2, \theta_2, r_2, \dots, x_k, y_k, \theta_k, r_k,
\zeta_1, \dots, \zeta_k)
\]
\[
= \sum_{j=1}^{k-1} R \cdot R_j \cdot \underline{F}
\Big(p_{\infty} \,,\,
\theta_1 + (k-1) \theta_{j+1} - \sum_{l \neq j+1} \theta_l \,,\,
(R \cdot R_j)^{-\frac{1}{2}} \, \zeta_j \Big)
+ R \cdot R_k \cdot \underline{F}
\Big( p_{\infty} \,,\,
\theta_1 - \sum_{l} \theta_l \,,\,
(R \cdot R_k)^{-\frac{1}{2}} \, \zeta_k \Big)
\]
\[
+ \sum_{j = 2}^{k-1} \frac{k^2}{2} \, \left(x_j y_{j+1} - x_{j+1}y_j\right)
+ \sum_{j = 2}^{k} k \, (R \cdot R_{j-1} - R \cdot R_{j-2}) \, \theta_j
\,.
\]
The function $\underline{\mathcal{P}}_F^{(k)}$
does not depend on the coordinate $r_1$.
Moreover,
since $F$ is invariant by the $\mathbb{Z}$-action \eqref{equation: action Z for F},
$\underline{\mathcal{P}}_F^{(k)}$ is invariant
by translation by $1$ in the coordinate $\theta_1$,
and thus descends to a continuous function
\[
\overline{\mathcal{P}_F^{(k)}}:
S^{2n} \times S^1 \times \mathbb{R}^{(2n+2)(k-1)} \times \mathbb{R}^{Nk}
\rightarrow \mathbb{R} \,,
\]
which by definition is invariant by the $\mathbb{Z}_k$-action
generated by $\overline{\underline{\sigma}_k}$,
and smooth if $F$ is special.

Suppose now that $F$ is special,
hence $\overline{\mathcal{P}_F^{(k)}}$ smooth.
In order to prove that the norm of the gradient of $\overline{\mathcal{P}_F^{(k)}}$
is bounded away from zero outside a compact set,
we show that any sequence of points $\{e_l\}$
with $\big\lVert \, \nabla \overline{\mathcal{P}_F^{(k)}} (e_l) \, \big\rVert \rightarrow 0$
is contained in a compact set.
For this,
it is equivalent to prove that for any sequence
\[
e_l = (x_1^{(l)}, y_1^{(l)}, \theta_1^{(l)}, r_1^{(l)}, \dots,
x_k^{(l)}, y_k^{(l)}, \theta_k^{(l)}, r_k^{(l)}, \zeta_1^{(l)}, \dots, \zeta_k^{(l)})
\]
with $\big\lVert \, \nabla (\mathcal{P}_F^{(k)} \circ B^{-1}) (e_l) \, \big\rVert \rightarrow 0$
the norms of $x_j^{(l)}, y_j^{(l)}, \theta_j^{(l)}, r_j^{(l)}$
for $j = 2, \dots, k$
and of $\zeta_j^{(l)}$ for $j = 1, \dots, k$
are bounded.
This is clear for $x_j^{(l)}$, $y_j^{(l)}$ and $\zeta_j^{(l)}$,
since $\mathcal{P}_F^{(k)} \circ B^{-1}$
is quadratic at infinity
in these coordinates.
Indeed,
$R \cdot R_j \cdot F \big(P_j(e)\big) = F_{\infty} (\zeta_j)$
for $e$ outside a compact set,
and the norm of the derivatives of $R \cdot R_j \cdot F \circ P_j$
with respect to $x_j$ and $y_j$ is bounded
since $\lVert d_1F \rVert$, $\lVert d_2F \rVert$, $\lvert d_3F \rvert$
and the functions $R \cdot R_j$ are bounded.
For the other coordinates we can argue as follows.
We have
\begin{equation}\label{equation: bounded 1}
0 = \lim_{l \to \infty}
\Big( \frac{\partial}{\partial \theta_1} (\mathcal{P}_F^{(k)} \circ B^{-1}) (e_l)
+ \frac{\partial}{\partial \theta_j} (\mathcal{P}_F^{(k)} \circ B^{-1}) (e_l) \Big)
\end{equation}
\[
= \lim_{l \to \infty} \Big( k \, (R \cdot R_{j-1}) (r_2^{(l)}, \dots, r_k^{(l)})
\, \big(d_3F \big(P_{j-1} (e_l)\big) + 1 \big) \Big)
- \lim_{l \to \infty} \Big( k \, (R \cdot R_{j-2}) (r_2^{(l)}, \dots, r_k^{(l)})\Big)
\]
for $j = 2, \dots, k$,
and
\begin{equation}\label{equation: bounded 2}
0 = \lim_{l \to \infty}
\Big( \frac{\partial}{\partial \theta_1} (\mathcal{P}_F^{(k)} \circ B^{-1}) (e_l)
- \sum_{j=2}^k \frac{\partial}{\partial \theta_j} (\mathcal{P}_F^{(k)} \circ B^{-1}) (e_l) \Big)
\end{equation}
\[
= \lim_{l \to \infty} \Big( k \, (R \cdot R_k) (r_2^{(l)}, \dots, r_k^{(l)})
\, \big(d_3F \big(P_k (e_l)\big) + 1 \big) \Big)
- \lim_{l \to \infty} \Big( k \, (R \cdot R_{k-1}) (r_2^{(l)}, \dots, r_k^{(l)})\Big) \,.
\]
Since $\sum_{j=1}^k R \cdot R_j = k$,
we cannot have $\lim_{l \to \infty} k \, (R \cdot R_j) (r_2^{(l)}, \dots, r_k^{(l)}) = 0$
for all $j = 1, \dots, k$.
On the other hand,
since $\lvert d_3F \rvert$ is bounded,
the equations \eqref{equation: bounded 1} and \eqref{equation: bounded 2}
imply that if this limit is zero
for some $j$ then it is zero for all $j = 1, \dots, k$.
We thus conclude that $\lim_{l \to \infty} k \, (R \cdot R_j) (r_2^{(l)}, \dots, r_k^{(l)}) \neq 0$
and $\lim_{l \to \infty} \big(d_3F \big(P_j (e_l)\big) + 1 \big) \neq 0$
for all $j = 1, \dots, k$.
Equations \eqref{equation: bounded 1} and \eqref{equation: bounded 2}
then give
\[
\lim_{l \to \infty} \Big( d_3F \big(P_{j} (e_l)\big) + 1 
- \frac{R \cdot R_{j-1}}{R \cdot R_{j}} \, (r_2^{(l)}, \dots, r_k^{(l)})\Big) = 0 \,,
\]
thus
\begin{equation}\label{equation: bounded 3}
\lim_{l \to \infty} \frac{R_{j-1}}{R_{j}} \, (r_2^{(l)}, \dots, r_k^{(l)})
= \lim_{l \to \infty} d_3F \big(P_{j} (e_l)\big) + 1 \,.
\end{equation}
Since
\[
\frac{R_{j-1}}{R_{j}}
= \begin{cases}
e^{k r_k} \text{ for } j = 1 \\
e^{-k r_2} \text{ for } j = 2 \\
e^{k r_{j-1} - k r_j} \text{ for } j = 3, \dots, k \,,
\end{cases}
\]
and since the right hand side of \eqref{equation: bounded 3}
is bounded and different than zero,
we conclude that the sequences
$\lvert r_2^{(l)} \rvert, \dots, \lvert r_k^{(l)} \rvert$
are bounded.
Pose now
\begin{equation}\label{equation: theta bounded}
\mathcal{R}_j = R \cdot R_{j-1} - R \cdot R_{j-2}
\end{equation}
for $j = 2, \dots, k$.
The map
\[
(r_2, \dots, r_k) \mapsto (R \cdot R_2 \,, \dots ,\, R \cdot R_k)
\mapsto (\mathcal{R}_2, \dots, \mathcal{R}_k)
\]
is a diffeomorphism into its image.
Indeed,
recall that $R \cdot R_1 = k - \sum_{j=2}^k R \cdot R_j$.
The second map is thus a linear diffeomorphism
and the first map is a diffeomorphism into its image,
since for every $j = 2, \dots, k$ we have
\[
e^{-kr_j} = \frac{R \cdot R_1}{R \cdot R_j}
= \frac{k - R \cdot R_2 - \dots - R \cdot R_k}{R \cdot R_j} \,.
\]
We have seen that
$\lvert r_2^{(l)} \rvert$, $\dots$, $\lvert r_k^{(l)} \rvert$
are bounded.
We thus have
\[
\frac{\partial}{\partial \mathcal{R}_j} (e_l)
= a_{1,j}^{(l)} \, \frac{\partial}{\partial r_1} (e_l)
+ \dots + a_{k,j}^{(l)} \, \frac{\partial}{\partial r_k} (e_l)
\]
for bounded sequences $\{a_{i, j}^{(l)}\}$
of real numbers,
and so
\[
0 = \lim_{l \to \infty}
\Big( \frac{\partial}{\partial \mathcal{R}_j} (\mathcal{P}_F^{(k)} \circ B^{-1}) (e_l) \Big)
= \lim_{l \to \infty}
\Big( \frac{\partial}{\partial \mathcal{R}_j}
\Big(\sum_{h=1}^k R \cdot R_h \cdot F \circ P_h \Big) (e_l) + \theta_j^{(l)} \Big)
\]
\[
= \lim_{l \to \infty}
\Big( \sum_{i=1}^k a_{i, j}^{(l)} \, \frac{\partial}{\partial r_i}
\Big(\sum_{h=1}^k R \cdot R_h \cdot F \circ P_h \Big) (e_l) \Big)
+ \lim_{l \to \infty} \theta_j^{(l)} \,.
\]
Since the first limit is finite,
we conclude that $\lvert \theta_j^{(l)} \rvert$ is bounded
for $j = 2, \dots, k$.
\end{proof}

The proof of the following proposition is immediate,
and left to the reader.

\begin{prop}\label{proposition: stabilization etc contact}
\begin{itemize}
\item[(i)]
Let $F: \mathbb{R}^{2n+1} \times \mathbb{R}^N \rightarrow \mathbb{R}$
be a generating function quadratic at infinity
of a compactly supported contactomorphism
of $(\mathbb{R}^{2n} \times S^1, \xi_0)$
contact isotopic to the identity.
Then for every non-degenerate quadratic form
$Q: \mathbb{R}^{N'} \rightarrow \mathbb{R}$
we have
\[
\overline{\mathcal{P}_{F \oplus Q}^{(k)}}
= \overline{\mathcal{P}_F^{(k)}} \oplus Q^{\oplus k} \,.
\]
\item[(ii)]
Let $F_0: \mathbb{R}^{2n+1} \times \mathbb{R}^N \rightarrow \mathbb{R}$
and $F_1: \mathbb{R}^{2n+1} \times \mathbb{R}^N \rightarrow \mathbb{R}$
be generating functions quadratic at infinity
of a compactly supported contactomorphism
of $(\mathbb{R}^{2n} \times S^1, \xi_0)$
contact isotopic to the identity
such that $\overline{F_0} = \overline{F_1} \circ \overline{\Phi}$
for a fibre preserving diffeomorphism $\overline{\Phi}$
of $S^{2n} \times S^1 \times \mathbb{R}^N$.
Denote by $\Phi$ the induced fibre preserving diffeomorphism
of $\mathbb{R}^{2n+1} \times \mathbb{R}^N$,
so that $F_0 = F_1 \circ \Phi$.
Then
\[
\mathcal{P}_{F_0}^{(k)}
= \mathcal{P}_{F_1}^{(k)} \circ \Phi_k \,,
\]
where $\Phi_k$ is the diffeomorphism
\[
(x_1, y_1, \theta_1, r_1, \dots , x_k, y_k, \theta_k, r_k,
\zeta_1, \dots , \zeta_k)
\]
\[
\mapsto (x_1, y_1, \theta_1, r_1, \dots , x_k, y_k, \theta_k, r_k,
\zeta_1', \dots , \zeta_k')
\]
of $\mathbb{R}^{(2n+2)k} \times \mathbb{R}^{Nk}$,
equivariant with respect to the $\mathbb{Z}_k$-action
generated by $\underline{\sigma}_k$,
defined by
\[
\Big( e^{-\frac{r_j}{2}} \frac{x_j + x_{j+1}}{2},
e^{-\frac{r_j}{2}} \frac{y_j + y_{j+1}}{2}, \theta_{j+1},
\zeta_j' \Big)
\]
\[
= \Phi \Big( e^{-\frac{r_j}{2}} \frac{x_j + x_{j+1}}{2},
e^{-\frac{r_j}{2}} \frac{y_j + y_{j+1}}{2}, \theta_{j+1},
\zeta_j \Big) \,.
\]
Moreover, $B \circ \Phi_k \circ B^{-1}$
extends to a fibre preserving homeomorphism 
of $S^{2n} \times \mathbb{R} \times \mathbb{R} \times \mathbb{R}^{(2n+2)(k-1)} \times \mathbb{R}^{Nk}$
that descends to a fibre preserving homeomorphism $\overline{\Phi_k}$
of $S^{2n} \times S^1 \times \mathbb{R}^{(2n+2)(k-1)} \times \mathbb{R}^{Nk}$,
equivariant with respect to the $\mathbb{Z}_k$-action
generated by $\overline{\underline{\sigma}_k}$,
and we have
\[
\overline{\mathcal{P}_{F_0}^{(k)}}
=\overline{\mathcal{P}_{F_1}^{(k)}} \circ \overline{\Phi_k} \,.
\]
\end{itemize}
\end{prop}


\section{\texorpdfstring{$\mathbb{Z}_k$}{}-equivariant generating function homology for domains of \texorpdfstring{$(\mathbb{R}^{2n} \times S^1, \xi_0)$}{}}\label{section: equivariant contact homology}

As in the symplectic case,
in order to define the $\mathbb{Z}_k$-equivariant
generating function homology
of domains of $(\mathbb{R}^{2n} \times S^1, \xi_0)$
the first step is to define
the $\mathbb{Z}_k$-equivariant homology
of a compactly supported contact isotopy $\{\phi_t\}_{t \in [0, 1]}$
of $(\mathbb{R}^{2n} \times S^1, \xi_0)$,
by the following category theoretical construction.

Consider the category $\mathcal{F} \big( \{\phi_t\}\big)$
whose objects are the 1-parameter families
$F_t: \mathbb{R}^{2n+1} \times \mathbb{R}^N \rightarrow \mathbb{R}$
of special generating functions quadratic at infinity for $\{\phi_t\}$
such that $F_0$ is a non-degenerate quadratic form,
and whose morphisms $F_t^{(0)} \rightarrow F_t^{(1)}$
are the triples $(Q_0, Q_1, \overline{\Phi})$
with $Q_0$ and $Q_1$ non-degenerate quadratic forms
and $\overline{\Phi} = \overline{\Phi_{s, t}}$ a 2-parameter family
of fibre preserving diffeomorphisms
such that $\overline{\Phi_{0, t}} = \id$,
$\overline{F_t^{(0)}} \oplus Q_0
= ( \overline{F_t^{(1)}} \oplus Q_1) \circ \overline{\Phi_{1, t}}$,
and $(\overline{F_t^{(0)}} \oplus Q_0) \circ \Phi_{s, 0}^{-1}$
is a contractible loop,
with the same composition law as in the symplectic case
\eqref{equation: composition}.

For $a \leq b$ in $\mathbb{R} \cup \{ \pm \infty \}$
that are not equal to the action
of a translated $k$-chain of $\phi_1$
we define a functor $G_{\mathbb{Z}_k, \ast}^{(a,b]}$
from $\mathcal{F} \big(\{\phi_t\}\big)$
to the category of graded modules over $\mathbb{Z}_k$
by posing
\[
G_{\mathbb{Z}_k, \ast}^{(a,b]} (F_t)
= H_{\mathbb{Z}_k, \ast + k\ind ((F_1)_{\infty}) + n(k-1)}
\big( \{ \overline{\mathcal{P}_{F_1}^{(k)}} \leq b \} \,,\,
\{ \overline{\mathcal{P}_{F_1}^{(k)}} \leq a \} ;
\mathbb{Z}_k \big) \,,
\]
with the convention that
$\overline{\mathcal{P}_{F_1}^{(k)}} \leq \infty$ means
$\overline{\mathcal{P}_{F_1}^{(k)}} \leq c$
for $c$ bigger than all the actions of translated $k$-chains of $\phi_1$
and $\overline{\mathcal{P}_{F_1}^{(k)}} \leq - \infty$
means $\overline{\mathcal{P}_{F_1}^{(k)}} \leq c$
for $c$ smaller than all the actions of translated $k$-chains of $\phi_1$,
and by associating to any morphism
$(Q_0, Q_1, \overline{\Phi}): F_t^{(0)} \rightarrow F_t^{(1)}$
a homomorphism
\[
G_{\mathbb{Z}_k, \ast}^{(a,b]} (Q_0, Q_1, \overline{\Phi}):
G_{\mathbb{Z}_k, \ast}^{(a,b]} (F_t^{(1)})
\rightarrow  G_{\mathbb{Z}_k, \ast}^{(a,b]} (F_t^{(0)})
\]
as follows.
Similarly to the symplectic case,
$(Q_0, Q_1, \overline{\Phi})$ can be written as the composition
\[
F_t^{(0)} \xrightarrow{ (Q_0, 0, \id) } F_t^{(0)} \oplus Q_0
\xrightarrow{ (0, 0, \overline{\Phi}) } F_t^{(1)} \oplus Q_1
\xrightarrow{ (0, Q_1, \id) } F_t^{(1)} \,.
\]
By \autoref{proposition: stabilization etc contact} (i),
\[
\overline{ \mathcal{P}_{F_1^{(1)} \oplus Q_1}^{(k)} }
= \overline{ \mathcal{P}_{F_1^{(1)}}^{(k)} } \oplus (Q_1)^{\oplus k} \,.
\]
We define $G_{\mathbb{Z}_k, \ast}^{(a,b]} (0, Q_1, \id)$
to be the isomorphism induced by
the $\mathbb{Z}_k$-equivariant Thom isomorphism
on the $\mathbb{Z}_k$-invariant vector bundle
on which $(Q_1)^{\oplus k}$ is negative definite.
Similarly, we define $G_{\mathbb{Z}_k, \ast}^{(a,b]} (Q_0, 0, \id)$
to be the inverse of  the isomorphism induced by
the $\mathbb{Z}_k$-equivariant Thom isomorphism
on the $\mathbb{Z}_k$-invariant vector bundle
on which $(Q_0)^{\oplus k}$ is negative definite.
By \autoref{proposition: stabilization etc contact} (ii),
\[
\overline{ \mathcal{P}_{F_1^{(0)} \oplus Q_0}^{(k)} } =
\overline{ \mathcal{P}_{F_1^{(1)} \oplus Q_1}^{(k)} }
\circ \overline{(\Phi_{1,1})_k} \,.
\]
We define
$G_{\mathbb{Z}_k, \ast}^{(a,b]}(0, 0, \overline{\Phi})$
to be the inverse of the isomorphism induced by $\overline{(\Phi_{1,1})_k}$.
Finally we define
\[
G_{\mathbb{Z}_k, \ast}^{(a,b]} (Q_0, Q_1, \overline{\Phi})
= G_{\mathbb{Z}_k, \ast}^{(a,b]} (Q_0, 0, \id)
\circ G_{\mathbb{Z}_k, \ast}^{(a,b]} (0, 0, \overline{\Phi})
\circ G_{\mathbb{Z}_k, \ast}^{(a,b]} (0, Q_1, \id) \,,
\]
and notice that this homomorphism
is by construction always an isomorphism.

As in the symplectic case,
given two morphisms
$(Q_0, Q_1, \overline{\Phi}): F_t^{(0)} \rightarrow F_t^{(1)}$
and $(Q_1', Q_2', \overline{\Psi}): F_t^{(1)} \rightarrow F_t^{(2)}$,
naturality of the $\mathbb{Z}_k$-equivariant Thom isomorphism
implies that
\[
G_{\mathbb{Z}_k, \ast}^{(a,b]} \big( (Q_1', Q_2', \overline{\Psi})
\circ (Q_0, Q_1, \overline{\Phi})\big)
= G_{\mathbb{Z}_k, \ast}^{(a,b]} (Q_0, Q_1, \overline{\Phi})
\circ G_{\mathbb{Z}_k, \ast}^{(a,b]} (Q_1', Q_2', \overline{\Psi}) \,,
\]
thus the functor $G_{\mathbb{Z}_k,*}^{(a,b]}$
is well-defined.

We define the $\mathbb{Z}_k$-equivariant homology
of the contact isotopy $\{\phi_t\}_{t \in [0, 1]}$
by
\[
G_{\mathbb{Z}_k,*}^{(a,b]} \big(\{\phi_t\}\big)
= \underset{}{\varprojlim} \;
\Big\{ G_{\mathbb{Z}_k,*}^{(a,b]} (F_t)
\Big\}_{ F_t \in \mathcal{F} (\{\phi_t\})} \,.
\]
As in the symplectic case
(using now \autoref{proposition: Serre fibration contact}
and \autoref{lemma: lemmas gf contact} (iii)),
any two objects $F_t^{(0)}$ and $F_t^{(1)}$
in $\mathcal{F} (\{\phi_t\})$
are related by a morphism,
and all morphisms between them
induce the same isomorphism
$G_{\mathbb{Z}_k, \ast}^{(a,b]} (F_t^{(1)})
\rightarrow  G_{\mathbb{Z}_k, \ast}^{(a,b]} (F_t^{(0)})$.
Thus,
as in \autoref{proposition: commuting isomorphism},
for every object $F_t$ in $\mathcal{F} \big(\{\phi_t\}\big)$
the natural homomorphism
$i_{F_t}: G_{\mathbb{Z}_k,*}^{(a,b]} \big(\{\phi_t\}\big)
\rightarrow G_{\mathbb{Z}_k,*}^{(a,b]} (F_t)$
is an isomorphism.

We now show that if $a, b \in k \mathbb{Z} \cup \{\pm \infty\}$
then the $\mathbb{Z}_k$-equivariant homology groups
$G_{\mathbb{Z}_k, \ast}^{(a, b]} \big(\{\phi_t\}\big)$
are invariant by conjugation.
The key ingredient is the following basic fact.

\begin{lemma}\label{lemma: invariance translated chains}
Let $\phi$ and $\psi$ be two compactly supported contactomorphisms
of $(\mathbb{R}^{2n} \times S^1, \xi_0)$ contact isotopic to the identity,
and let $t \in \mathbb{Z}$.
Then $(p_1, \dots , p_k)$ is a translated $k$-chain of $\phi$
of action $tk$
if and only if $\big( \psi(p_1), \dots , \psi(p_k) \big)$
is a translated $k$-chain of $\psi \circ \phi \circ \psi^{-1}$
of action $tk$.
\end{lemma}

\begin{proof}
Suppose that $(p_1, \dots , p_k)$ is a translated $k$-chain of $\phi$ of action $tk$.
By definition, there is then a $k$-tuple of points $(P_1, \dots , P_k)$
of $\mathbb{R}^{2n+1}$ that projects to $(p_1, \dots , p_k)$
and is a translated $k$-chain of the lift $\Phi$ of $\phi$
to $(\mathbb{R}^{2n+1}, \xi_0)$.
Thus $g (P_1) + \cdots + g(P_k) = 0$,
where $g$ denotes the conformal factor of $\Phi$,
and $P_{j+1} = \varphi_{-t}^{\alpha_0} \circ \Phi \, (P_j)$
for all $j$.
But then,
denoting by $\Psi$ the lift of $\psi$
to $(\mathbb{R}^{2n+1}, \xi_0)$,
\[
\varphi_{-t}^{\alpha_0} \circ (\Psi \circ \Phi \circ \Psi^{-1}) \, \big(\Psi(P_j)\big)
= \Psi \circ \varphi_{-t}^{\alpha_0} \circ \Phi \, (P_j)
= \Psi (P_{j+1}) \,,
\]
where we have used that $t \in \mathbb{Z}$
and that $\Psi$ is equivariant by translation by $1$
in the Reeb direction.
Moreover,
the conformal factor of $\Psi \circ \Phi \circ \Psi^{-1}$
is $g \circ \Psi^{-1}$
and we have
\[
g \circ \Psi^{-1} \big(\Psi(P_1)\big) + \cdots + g \circ \Psi^{-1} \big(\Psi(P_k)\big)
= g (P_1) + \cdots + g(P_k) = 0 \,.
\]
This shows that $\big(\Psi(P_1), \cdots, \Psi(P_k)\big)$
is a translated $k$-chain of $\Psi \circ \Phi \circ \Psi^{-1}$ of action $tk$,
hence $\big( \psi(p_1), \dots , \psi(p_k) \big)$
is a translated $k$-chain of $\psi \circ \phi \circ \psi^{-1}$ of action $tk$.
The converse implication is similar.
\end{proof}

Let $\{\phi_t\}_{t \in [0, 1]}$ be a compactly supported contact isotopy
of $(\mathbb{R}^{2n} \times S^1, \xi_0)$,
and suppose that $a \leq b$ are in $k\mathbb{Z} \cup \{\pm \infty\}$
and are not equal to the action
of any translated $k$-chain of $\phi_1$.
By \autoref{lemma: special},
there is a 2-parameter family
$F_t^{(s)}: \mathbb{R}^{2n+1} \times \mathbb{R}^N \rightarrow \mathbb{R}$
of special generating functions quadratic at infinity
for the 2-parameter family $\{\psi_s \circ \phi_t \circ \psi_s^{-1} \}$
such that $F_0^{(s)}$ is a non-degenerate quadratic form for every $s$
and $(F_t^{(s)})_{\infty}$ does not depend on $s$ and $t$.
By \autoref{proposition: critical points contact},
the critical points of $\mathcal{P}_{F_1^{(s)}}^{(k)}$
are in 1--1 correspondence
with the translated $k$-chains of
$\psi_s \circ \phi_1 \circ \psi_s^{-1}$,
and the critical values are given
by the actions.
Since $a$ and $b$ are in $k\mathbb{Z} \cup \{\pm \infty\}$
and are regular values of $\mathcal{P}_{F_1^{(0)}}^{(k)}$,
it thus follows from \autoref{lemma: invariance translated chains}
that they are regular values
of $\mathcal{P}_{F_1^{(s)}}^{(k)}$,
hence of $\overline{\mathcal{P}_{F_1^{(s)}}^{(k)}}$,
for all $s$.
Since $(F_1^{(s)})_{\infty}$ does not depend on $s$,
$\frac{d}{ds} F_1^{(s)}$ is bounded
and so $\frac{d}{ds} \, \overline{ \mathcal{P}_{F_1^{(s)}}^{(k)}}$ is bounded.
Moreover,
by \autoref{proposition: extend to sphere contact}
the norm of the gradient of each $\overline{ \mathcal{P}_{F_1^{(s)}}^{(k)}}$
is bounded away from zero outside a compact set.
We can thus apply \autoref{lemma: continuation}
to obtain a $\mathbb{Z}_k$-equivariant isotopy
$\{\theta_s\}$ of $S^{2n} \times S^1 \times \mathbb{R}^{(2n+2)(k-1)} \times \mathbb{R}^{Nk}$
mapping the sublevel sets of $\overline{\mathcal{P}_{F_1^{(0)}}^{(k)}}$
at $a$ and $b$ to those of $\overline{\mathcal{P}_{F_1^{(s)}}^{(k)}}$.
In particular, this gives an isomorphism
\[
(\theta_1)_{\ast}:
G_{\mathbb{Z}_k, \ast}^{(a, b]} (F_t^{(0)})
\rightarrow G_{\mathbb{Z}_k, \ast}^{(a, b]} (F_t^{(1)}) \,,
\]
and so an isomorphism
\[
\lambda_{F_t^{(s)}} :=
(i_{F_t^{(1)}})^{-1} \circ (\theta_1)_{\ast} \circ i_{F_t^{(0)}}:
G_{\mathbb{Z}_k, \ast}^{(a, b]} \big(\{\phi_t\}\big)
\rightarrow
G_{\mathbb{Z}_k, \ast}^{(a, b]}
\big(\{\psi_1 \circ \phi_t \circ \psi_1^{-1} \} \big) \,.
\]
The next proposition can be proved
as the corresponding result in the symplectic case
(Proposition \ref{proposition: invariance by conjugation symplectic bis}).

\begin{prop}\label{proposition: invariance by conjugation contact bis}
Let $\{\phi_t\}$ and $\{\psi_t\}_{t \in [0, 1]}$
be compactly supported contact isotopies
of $(\mathbb{R}^{2n} \times S^1, \xi_0)$,
and suppose that $a \leq b$ in $k\mathbb{Z} \cup \{\pm \infty\}$
are not equal to the action
of any translated $k$-chain of $\phi_1$.
Then the isomorphism
\[
\lambda_{\{\psi_t\}}:
G_{\mathbb{Z}_k, \ast}^{(a, b]} \big(\{\phi_t\}\big)
\rightarrow 
G_{\mathbb{Z}_k, \ast}^{(a, b]}
\big( \{ \psi_1 \circ \phi_t \circ \psi_1^{-1} \} \big)
\]
defined by $\lambda_{\{\psi_t\}} = \lambda_{F_t^{(s)}}$
for any 2-parameter family
of special generating functions quadratic at infinity $F_t^{(s)}$
for the 2-parameter family of contactomorphisms
$\{\psi_s \circ \phi_t \circ \psi_s^{-1} \}$
is well-defined,
i.e.\ it does not depend on the choice of $F_t^{(s)}$.
\end{prop}

Consider now a domain $\mathcal{V}$
of $(\mathbb{R}^{2n} \times S^1, \xi_0)$.
Let $\mathcal{H}^k_{a, b} (\mathcal{V})$
be the set of compactly supported
contact isotopies $\{\phi_t\}_{t \in [0, 1]}$
supported in $\mathcal{V}$
such that $a$ and $b$ are not equal to the action
of any translated $k$-chain of $\phi_1$.
Define a partial order $\leq$
on $\mathcal{H}^k_{a, b} (\mathcal{V})$
by posing $\{\phi_t^{(1)}\} \leq \{\phi_t^{(2)}\}$
if $\{\phi_t^{(2)} \circ (\phi_t^{(1)})^{-1}\}$
is a non-negative contact isotopy.
Similarly to the symplectic case,
for $\{\phi_t^{(i)}\}$ and $\{\phi_t^{(j)}\}$
in $\mathcal{H}^k_{a, b} (\mathcal{V})$
with $\{\phi_t^{(i)}\} \leq \{\phi_t^{(j)}\}$
there is an induced homomorphism
\begin{equation}\label{equation: monotonicity homomorphism contact}
   \mu_i^j: 
G_{\mathbb{Z}_k, \ast}^{(a, b]} \big(\{\phi_t^{(j)}\}\big)
\rightarrow 
G_{\mathbb{Z}_k, \ast}^{(a, b]} \big(\{\phi_t^{(i)}\}\big) \,. 
\end{equation}
Moreover,
the homomorphisms $\mu_i^j$ satisfy the cocycle conditions
\[
\begin{cases}
\mu_i^{i} = \id \\
\mu_i^{j} \circ \mu_j^l = \mu_i^l
\quad \text{ for }
\{\phi_t^{(i)}\} \leq \{ \phi_t^{(j)} \}
\leq \{ \phi_t^{(l)} \} \,,
\end{cases}
\]
making
$\big\{ G_{\mathbb{Z}_k,*}^{(a,b]} \big( \{\phi_t\} \big)
\big\}_{ \{\phi_t\} \in \mathcal{H}^k_{a, b} (\mathcal{V}) }$
an inversely directed family
of graded modules over $\mathbb{Z}_k$.
We define
\[
G_{\mathbb{Z}_k,*}^{(a,b]} (\mathcal{V}) =
\underset{}{\varprojlim} \;
\Big\{ G_{\mathbb{Z}_k,*}^{(a,b]} \big( \{\phi_t\} \big)
\Big\}_{ \{\phi_t\} \in \mathcal{H}^k_{a, b} (\mathcal{V}) } \,.
\]

As in the symplectic case,
using now
\autoref{proposition: invariance by conjugation contact bis},
we obtain the following invariance result.

\begin{prop}
Let $a \leq b$ in $k\mathbb{Z} \cup \{\pm \infty\}$.
For every domain $\mathcal{V}$ of $(\mathbb{R}^{2n} \times S^1, \xi_0)$
and every compactly supported contact isotopy
$\{\psi_t\}_{t \in [0, 1]}$
the isomorphisms
$\lambda_{\{\psi_t\}}:
G_{\mathbb{Z}_k, \ast}^{(a, b]} \big(\{\phi_t\}\big)
\rightarrow 
G_{\mathbb{Z}_k, \ast}^{(a, b]}
\big( \{ \psi_1 \circ \phi_t \circ \psi_1^{-1} \} \big)$
for $\{\phi_t\} \in \mathcal{H}_{a, b}^k (\mathcal{V})$
induce a well-defined isomorphism
\[
\lambda_{\{\psi_t\}}:
G_{\mathbb{Z}_k,*}^{(a,b]} (\mathcal{V})
\rightarrow G_{\mathbb{Z}_k,*}^{(a,b]} \big(\psi_1(\mathcal{V})\big) \,.
\]
\end{prop}

The groups $G_{\mathbb{Z}_k,*}^{(a,b]} (\mathcal{V})$
satisfy the same functorial properties
as in the symplectic case.
If $\mathcal{V}_1 \subset \mathcal{V}_2$,
the inclusion of posets $\mathcal{H}^k_{a, b} (\mathcal{V}_1)\subset \mathcal{H}^k_{a, b} (\mathcal{V}_2)$
induces a homomorphism
\[
G_{\mathbb{Z}_k,*}^{(a,b]} (\mathcal{V}_2) \rightarrow G_{\mathbb{Z}_k,*}^{(a,b]} (\mathcal{V}_1) \,,
\]
and for
$\mathcal{V}_1 \subset \mathcal{V}_2 \subset \mathcal{V}_3$
and for any compactly supported contact isotopy
$\{\psi_t\}_{t \in [0, 1]}$
these homomorphisms fit into commutative diagrams
\[
\begin{tikzcd}
G_{\mathbb{Z}_k, \ast}^{(a, b]} (\mathcal{V}_3) \arrow{r} \arrow {rd} & G_{\mathbb{Z}_k, \ast}^{(a, b]} (\mathcal{V}_2) \arrow{d} \\
& G_{\mathbb{Z}_k, \ast}^{(a, b]} (\mathcal{V}_1)
\end{tikzcd}
\]
and, for $a \leq b$ in $k\mathbb{Z} \cup \{\pm \infty\}$, 
\[
\begin{tikzcd}
G_{\mathbb{Z}_k, \ast}^{(a, b]} (\mathcal{V}_2)
\arrow{r} \arrow{d}{\lambda_{\{\psi_t\}}}
& G_{\mathbb{Z}_k, \ast}^{(a, b]} (\mathcal{V}_1)
\arrow{d}{\lambda_{\{\psi_t\}}} \\
G_{\mathbb{Z}_k, \ast}^{(a, b]} \big(\psi_1(\mathcal{V}_2)\big)
\arrow{r}
& G_{\mathbb{Z}_k, \ast}^{(a, b]} \big(\psi_1(\mathcal{V}_1)\big) \,. 
\end{tikzcd}
\]


\section{Relation between the \texorpdfstring{$\mathbb{Z}_k$}{}-equivariant homology
of a domain of \texorpdfstring{$(\mathbb{R}^{2n}, \omega_0)$}{}
and the \texorpdfstring{$\mathbb{Z}_k$}{}-equivariant homology of its prequantization}
\label{section: relation symplectic contact}

We start by describing the relation
between the $\mathbb{Z}_k$-equivariant homology
of a compactly supported Hamiltonian isotopy
$\{\varphi_t\}_{t \in [0, 1]}$ of $(\mathbb{R}^{2n}, \omega_0)$
and that of its lift $\{\widetilde{\varphi_t}\}_{t \in [0, 1]}$
to $(\mathbb{R}^{2n} \times S^1, \xi_0)$.
Let  $f_t: \mathbb{R}^{2n} \times \mathbb{R}^N
\rightarrow \mathbb{R}$
be a 1-parameter family
of special generating function quadratic at infinity
for $\{\varphi_t\}$.
Then, by \autoref{proposition: gf lift to contact},
\begin{equation}\label{equation: f to F}
F_t: \mathbb{R}^{2n+1} \times \mathbb{R}^N
\rightarrow \mathbb{R} \,,\;
F_t (x, y, \theta, \zeta)
= f _t (x, y, \zeta)
\end{equation}
is a 1-parameter family
of special generating function quadratic at infinity
for $\{\widetilde{\varphi_t}\}$.
By Proposition \ref{proposition: composition formula}
and \autoref{proposition: critical points contact},
the critical points of $(f_1)^{\sharp k}$ and $\mathcal{P}_{F_1}^{(k)}$
are in 1--1 correspondence
with the $k$-periodic points of $\varphi_1$
and the translated $k$-chains of $\widetilde{\varphi_1}$ respectively,
with critical values given by the symplectic and contact actions.
Denoting by $S$ the compactly supported function
that satisfies $(\varphi_1)^{\ast} \lambda_0 - \lambda_0 = dS$,
$\big( (x_1, y_1, \theta_1), \dots, (x_k, y_k, \theta_k) \big)$
is a translated $k$-chain of $\widetilde{\varphi_1}$ of action $tk$
if and only if $(x_{j+1}, y_{j+1}) = \varphi_1 (x_j, y_j)$
and $\theta_{j+1} = \theta_j + S (x_j, y_j) - t$ for all $j$,
with the usual cyclic convention on the indices.
In this case,
$(x_1, y_1)$ is a fixed point of $(\varphi_1)^k$
and its action $\sum_{j=1}^k S (x_j, y_j)$
coincides with the action of the translated $k$-chain
$\big( (x_1, y_1, \theta_1), \dots, (x_k, y_k, \theta_k) \big)$
of $\widetilde{\varphi_1}$.
We conclude that  $(f_1)^{\sharp k}$ and $\mathcal{P}_{F_1}^{(k)}$,
hence $\overline{(f_1)^{\sharp k}}$ and $\overline{\mathcal{P}_{F_1}^{(k)}}$,
have the same critical values.
Moreover,
we have the following result.

\begin{prop}\label{proposition: equivariant reduction 1}
For every $a, b \in \mathbb{R}$
that are regular values of $\overline{(f_1)^{\sharp k}}$
and $\overline{\mathcal{P}_{F_1}^{(k)}}$
we have
\[
H_{\mathbb{Z}_k, \ast} \big( \big\{ \, \overline{\mathcal{P}_{F_1}^{(k)}} \leq b \,\big\} \,,\,
\big\{\, \overline{\mathcal{P}_{F_1}^{(k)}} \leq a \,\big\} \big)
\cong H_{\mathbb{Z}_k, \ast} \big(\, \{ \, \overline{(f_1)^{\sharp k}} \leq b \,\} \times S^1 \,,\,
\{\, \overline{(f_1)^{\sharp k}} \leq a \,\} \times S^1 \,\big) \,,
\]
where the $\mathbb{Z}_k$-action on the first pair
is generated by $\overline{\underline{\sigma}_k}$
and the $\mathbb{Z}_k$-action on the second pair
is generated by $(\overline{\sigma_k}, \id_{S^1})$.
\end{prop}

\begin{proof}
Consider the 1-parameter family of functions $G_s$,
for $s \in [0, 1]$,
defined by
\begin{gather*}
G_s (p, \theta, x_2, y_2, \theta_2, r_2, \dots ,
x_k, y_k, \theta_k, r_k, \zeta_1, \dots , \zeta_k) \\
= \overline{\mathcal{P}_{F_1}^{(k)}}
(p, \theta, x_2, y_2, \theta_2, s r_2, \dots ,
x_k, y_k, \theta_k, s r_k, \zeta_1, \dots , \zeta_k) \,.
\end{gather*}
Then $G_1 = \overline{\mathcal{P}_{F_1}^{(k)}}$ and
\[
G_0 (p, \theta, x_2, y_2, \theta_2, r_2, \dots , x_k, y_k, \theta_k, r_k, \zeta_1, \dots , \zeta_k)
= \overline{f_1^{\sharp k}} \, (p, x_2, y_2, \dots , x_k, y_k, \zeta_1, \dots, \zeta_k) \,.
\]
All the functions $G_s$ have the same critical values.
Since, by hypothesis, $a$ and $b$ are regular values of $G_1$,
they are thus regular values of $G_s$ for all $s$.
By \autoref{proposition: extend to sphere contact},
the critical points of $G_1 = \overline{\mathcal{P}_{F_1}^{(k)}}$
are contained in a compact set.
Moreover,
since $F_1$ does not depend on $\theta$,
all the critical points
of $G_1$ satisfy $r_2 = \dots = r_k = 0$.
Since $(p, \theta, x_2, y_2, \theta_2, r_2, \dots , x_k, y_k, \theta_k, r_k, \zeta_1, \dots , \zeta_k)$
is a critical point of $G_s$
if and only if $(p, \theta, x_2, y_2, \theta_2, s r_2, \dots ,
x_k, y_k, \theta_k, s r_k, \zeta_1, \dots , \zeta_k)$
is a critical point of $G_1$,
we deduce that all the critical points of $G_s$ for $s \in (0, 1]$
satisfy $r_2 = \dots = r_k = 0$
and are thus contained in a compact set $K$ independent of $s$.
On the other hand,
since $G_0$ does not depend on $r_2, \dots, r_k$ and $\theta_2, \dots, \theta_k$,
its critical points are not contained in a compact set.
We therefore deform the family $G_s$ as follows.
Let $\underline{K}$ be a compact set that strictly contains $K$,
and let $f_s$ be a 1-parameter family of functions such that, for every $s$,
$f_s \equiv 0$ on $\underline{K}$
and $f_s \equiv \epsilon_s$ outside a neighborhood $\mathcal{K}$ of $\underline{K}$,
where $\epsilon_s$ is a 1-parameter family
of non-negative small enough real numbers
with $\epsilon_0 > 0$ and $\epsilon_s = 0$ for $s$ in a neighborhood of $1$.
Consider the 1-parameter family of functions $\underline{G}_s$
defined by
\[
\underline{G}_s (e) = G_s (e) + f_s (e) (r_2 + \dots + r_k + \theta_2 + \dots + \theta_k)
\]
for $e = (p, \theta, x_2, y_2, \theta_2, r_2, \dots , x_k, y_k, \theta_k, r_k, \zeta_1, \dots , \zeta_k)$.
For every $s \in (0, 1]$,
$\underline{G}_s$ has the same critical points and critical values as $G_s$.
Indeed, the two functions coincide on $\underline{K}$
and an argument similar to the one that shows
that all the critical points of $G_s$ are contained in $K$
also shows that all the critical points of $\underline{G}_s$
are contained in $\underline{K}$
if the numbers $\epsilon_s$ have been chosen small enough.
On the other hand,
now $\underline{G}_0$ as well
does not have any critical point outside $\underline{K}$.
Since $G_0$ does not depend on $r_2, \dots, r_k$ and $\theta_2, \dots, \theta_k$,
$( \{ \underline{G}_0 \leq b \} , \{ \underline{G}_0 \leq a \} )$
is homotopy equivalent to $( \{ G_0 \leq b \} ,\{ G_0 \leq a \} )$,
hence to $\big(\, \{ \, \overline{(f_1)^{\sharp k}} \leq b \,\} \times S^1 \,,\,
\{\, \overline{(f_1)^{\sharp k}} \leq a \,\} \times S^1 \,\big)$,
by a $\mathbb{Z}_k$-equivariant homotopy equivalence.
Since $\underline{G}_1 = G_1 = \overline{\mathcal{P}_{F_1}^{(k)}}$,
$( \{ \underline{G}_1 \leq b \} ,\{ \underline{G}_1 \leq a \} )
= ( \{ \overline{\mathcal{P}_{F_1}^{(k)}} \leq b \} ,\{ \overline{\mathcal{P}_{F_1}^{(k)}} \leq a \} )$.
It thus remains to show that
\begin{equation}\label{equation: in proof relation symplectic contact}
H_{\mathbb{Z}_k, \ast} ( \{ \underline{G}_0 \leq b \} , \{ \underline{G}_0 \leq a \} )
\cong H_{\mathbb{Z}_k, \ast} ( \{ \underline{G}_1 \leq b \} , \{ \underline{G}_1 \leq a \} ) \,.
\end{equation}
This can be seen as follows.
For every $s$,
we have seen that there are no critical points
of $\underline{G}_s$ outside $\underline{K}$.
Moreover,
by the same arguments as in the proof of \autoref{proposition: extend to sphere contact}
we see that the norm of the gradient of $\underline{G}_s$
is bounded away from zero outside $\underline{K}$.
Thus
\begin{equation}\label{equation: in proof relation symplectic contact 2}
H_{\mathbb{Z}_k, \ast} ( \{ \underline{G}_s \leq b \} ,\{ \underline{G}_s \leq a \} )
\cong H_{\mathbb{Z}_k, \ast} ( \{ \underline{G}_s \leq a + \rho (b - a) \},
\{ \underline{G}_s \leq a \} ) \,,
\end{equation}
where $\rho$ is a cut-off function
supported in the neighborhood $\mathcal{K}$ of $\underline{K}$
and equal to $1$ on $\underline{K}$.
By an argument similar to the one in the proof of \autoref{lemma: continuation},
there is a $\mathbb{Z}_k$-equivariant isotopy $\{\theta_s\}$
supported in $\mathcal{K}' \supset \mathcal{K}$
such that $\theta_1$ sends
$( \{ \underline{G}_0 \leq a + \rho (b - a) \} \cap \mathcal{K}
\,,\,\{ \underline{G}_0 \leq a \} \cap \mathcal{K} )$
to $( \{ \underline{G}_1 \leq a + \rho (b - a) \} \cap \mathcal{K}
\,,\,\{ \underline{G}_1 \leq a \} \cap \mathcal{K} )$.
We thus have
\begin{align*}
H_{\mathbb{Z}_k, \ast} ( \{ \underline{G}_0 \leq a + \rho (b - a) \} \,,\,\{ \underline{G}_0 \leq a \} )
&\cong H_{\mathbb{Z}_k, \ast} \big( \theta_1 (\{ \underline{G}_0 \leq a + \rho (b - a) \})
\,,\, \theta_1 (\{ \underline{G}_0 \leq a \} ) \big)  \\
&\cong H_{\mathbb{Z}_k, \ast} ( \{ \underline{G}_1 \leq a + \rho (b - a) \} \cap \mathcal{K}
\,,\,\{ \underline{G}_1 \leq a \} \cap \mathcal{K} ) \\
&\cong H_{\mathbb{Z}_k, \ast} ( \{ \underline{G}_1 \leq a + \rho (b - a) \}
\,,\,\{ \underline{G}_1 \leq a \} )  \,,
\end{align*}
where the last two isomorphisms follow from excision.
Together with \eqref{equation: in proof relation symplectic contact 2},
this gives \eqref{equation: in proof relation symplectic contact}.
\end{proof}

By \autoref{proposition: equivariant reduction 1}
and the  K\"unneth formula
we have
\begin{align*}
G_{\mathbb{Z}_k, \ast}^{(a, b]} (F_t)
&\cong H_{\mathbb{Z}_k, \ast + k \ind ((f_1)_{\infty}) + n(k-1)}
\big(\{ \overline{(f_1)^{\sharp k}} \leq b\} \times S^1 \,,\,
\{  \overline{(f_1)^{\sharp k}} \leq a\} \times S^1 \big) \\
&= H_{\ast + k \ind ((f_1)_{\infty}) + n(k-1)}
\big( (\{  \overline{(f_1)^{\sharp k}} \leq b \} \times S^1 \times E\mathbb{Z}_k) / \mathbb{Z}_k \,,
(\{  \overline{(f_1)^{\sharp k}} \leq a \} \times S^1 \times E\mathbb{Z}_k) / \mathbb{Z}_k \big) \\
&\cong H_{\ast + k \ind ((f_1)_{\infty}) + n(k-1)}
\big( (\{  \overline{(f_1)^{\sharp k}} \leq b \} \times E\mathbb{Z}_k) / \mathbb{Z}_k \times S^1 \,,\,
(\{  \overline{(f_1)^{\sharp k}} \leq a \} \times E\mathbb{Z}_k) / \mathbb{Z}_k \times S^1 \big) \\
&\cong H_{\mathbb{Z}_k, \ast + k \ind ((f_1)_{\infty}) + n(k-1)}
(\{  \overline{(f_1)^{\sharp k}} \leq b\} ,\{  \overline{(f_1)^{\sharp k}} \leq a\}) \otimes H_*(S^1) \\
&= G_{\mathbb{Z}_k, \ast}^{(a, b]} (f_t) \otimes H_{\ast} (S^1) \,.
\end{align*}

We define a functor
\[
\mathcal{F} \big( \{\varphi_t\}\big) \rightarrow
\mathcal{F} \big(\{\widetilde{\varphi_t}\}\big)
\]
by mapping an object $f_t$ of $\mathcal{F} \big( \{\varphi_t\}\big)$
to the object $F_t$ of $\mathcal{F} \big(\{\widetilde{\varphi_t}\}\big)$
defined by \eqref{equation: f to F}
and a morphism $(Q_0, Q_1, \overline{\Phi})$
to the morphism $\big(Q_0, Q_1, (\overline{\Phi}, \id_{S^1}) \big)$.
The diagram
\[\begin{tikzcd}
\mathcal{F} \big( \{\varphi_t\}\big) \arrow[r] \arrow[d, "G_{\mathbb{Z}_k, \ast}^{(a, b]}"]
& \mathcal{F} \big(\{\widetilde{\varphi_t}\}\big) \arrow[d, "G_{\mathbb{Z}_k, \ast}^{(a, b]}"] \\
\operatorname{GrMod}_{\mathbb{Z}_k} \; \arrow[r, "\otimes H_*(S^1)"]
& \operatorname{GrMod}_{\mathbb{Z}_k}
\end{tikzcd}\]
is commutative,
and so we have an isomorphism
\begin{equation}\label{equation: symplectic/contact isomorphism.}
    G_{\mathbb{Z}_k, \ast}^{(a, b]} (\{\widetilde{\varphi_t}\})
\cong G_{\mathbb{Z}_k, \ast}^{(a, b]} (\{\varphi_t\})
\otimes H_{\ast} (S^1).
\end{equation}

Consider now a domain $\mathcal{U}$ of $(\mathbb{R}^{2n}, \omega_0)$,
and its prequantization $\widehat{\mathcal{U}} = \mathcal{U} \times S^1$
in $(\mathbb{R}^{2n} \times S^1, \xi_0)$.
If $\{\varphi_t\} \in \mathcal{H}^k_{a, b} (\mathcal{U})$
then $\{\widetilde{\varphi_t}\} \in \mathcal{H}^k_{a, b} (\widehat{\mathcal{U}})$,
and
\[
\mathcal{H}^k_{a, b} (\mathcal{U}) \rightarrow \mathcal{H}^k_{a, b} ( \widehat{\mathcal{U}})
\,,\; \{\varphi_t\} \mapsto \{\widetilde{\varphi_t}\}
\]
is a monotone map between posets.
The monotonicity maps
for $\{{\varphi_t^{(0)}}\}\leq \{{\varphi_t^{(1)}}\}$
defined in \eqref{equation: monotonicity homomorphism symplectic}
and those for $\{\widetilde{\varphi_t^{(0)}}\}\leq \{\widetilde{\varphi_t^{(1)}}\}$
defined in \eqref{equation: monotonicity homomorphism contact}
commute with the isomorphism \eqref{equation: symplectic/contact isomorphism.},
and so we have an isomorphism
\[
G_{\mathbb{Z}_k, \ast}^{(a, b]} (\widehat{\mathcal{U}}) \cong 
G_{\mathbb{Z}_k, \ast}^{(a, b]} (\mathcal{U})
\otimes H_{\ast} (S^1) \,.
\]
Moreover,
for every inclusion $\mathcal{U}_1 \hookrightarrow \mathcal{U}_2$
we have a commutative diagram
\[
\begin{tikzcd}
{G_{\mathbb{Z}_k, \ast}^{(a, b]} (\widehat{\mathcal{U}_2})} \arrow[r, "\cong"] \arrow[d]
&
{G_{\mathbb{Z}_k, \ast}^{(a, b]}(\mathcal{U}_2)
\otimes H_{\ast} (S^1)}
\arrow[d] \\
{G_{\mathbb{Z}_k, \ast}^{(a, b]} (\widehat{\mathcal{U}_1})}
\arrow[r, "\cong"]
&
{G_{\mathbb{Z}_k, \ast}^{(a, b]} (\mathcal{U}_1)
\otimes H_{\ast} (S^1)} \,.       
\end{tikzcd}
\]


\section{Calculations for balls}\label{section: balls}

In this section we calculate 
the equivariant homology
$G_{\mathbb{Z}_k, 2nl}^{(a, \infty]} \big(B^{2n}(R)\big)$
for $a > 0$ and $0< l < k$ with $k$ prime.
Our arguments are based
on the calculations in the non-equivariant case
done in \cite{Traynor},
and are similar to those
in \cite{S - equivariant}, \cite{Milin}, \cite{Fraser},
\cite{Chiu}, \cite{Zhang}.

Consider the Hamiltonian function
\[
H (z_1, \dots , z_n)
= \sum_{j=1}^n \frac{\lvert z_j \rvert^2}{R^2} \,,
\]
and its flow
\[
\varphi_t (z_1, \dots , z_n)
= ( e^{2 i t /R^2} z_1, \dots , e^{2 i t /R^2} z_n ) \,,
\]
which is periodic of period $\pi R^2$.
Let $\rho: \mathbb{R}_{\geq 0} \rightarrow \mathbb{R}$
be a function supported in $[0, 1]$
such that $\rho'' \geq 0$,
$\rho''(m) > 0$ for $m$ with
$\rho'(m) \in - \mathbb{N} \cdot \pi R^2$,
and $\left. \rho' \right\lvert_{[0, \delta]} \equiv c$
for some $c < 0$ and $\delta > 0$.
The flow $\{\varphi_t^{\rho}\}$ of $H_{\rho} := \rho \circ H$
is supported in $B^{2n}(R)$,
and is given by
\[
\varphi_t^{\rho} (z) =
\varphi_{ t \rho' (H(z))} (z) \,.
\]
Consider now a sequence
$\rho_j: \mathbb{R}_{\geq 0} \rightarrow \mathbb{R}$
of such functions
with $\lim_{j \to \infty} \rho_j(0) = \infty$,
$\lim_{j \to \infty} \rho_j'(0) = - \infty$
and so that,
denoting by $\{\varphi_t^{\rho_j}\}$
the flows of the Hamiltonian functions $H_{\rho_j}$,
the sequence
\[
\{\varphi_t^{\rho_1}\} \leq \{\varphi_t^{\rho_2}\}
\leq \{\varphi_t^{\rho_3}\} \leq \cdots
\]
is cofinal in $\mathcal{H}_{a, \infty}^k \big(B^{2n} (R)\big)$.
For each $j$,
let $F_t^{(j)}: \mathbb{R}^{2n} \times \mathbb{R}^{N_j}
\rightarrow \mathbb{R}$
be a 1-parameter family
of special generating functions quadratic at infinity
for $\{\varphi_t^{\rho_j}\}$,
and consider the associated 1-parameter family
of generating functions $(F_t^{(j)})^{\sharp k}$
for $\{(\varphi_t^{\rho_j})^k\}$.

The critical points of $(F_1^{(j)})^{\sharp k}$
are in 1--1 correspondence
with the fixed points of $(\varphi_1^{\rho_j})^{k}$.
A point $z$ is a fixed point of $(\varphi_1^{\rho_j})^{k}$
if and only if
$\rho_j' \big(H(z)\big) = - \frac{l}{k} \pi R^2$
for some integer $l \geq 0$.
The space of fixed points of $(\varphi_1^{\rho_j})^{k}$
is thus the union of the spaces
$Z_{j, \infty} = \{0\}$,
\[
Z_{j, 0} = \{\, z \in \mathbb{R}^{2n} \;\lvert\;
\rho_j' \big(H(z)\big) = 0 \,\} \,,
\]
and
\[
Z_{j, l} = \big\{\, z \in \mathbb{R}^{2n} \;\big\lvert\;
\rho_j' \big(H(z)\big) = - \frac{l}{k} \pi R^2 \,\big\}
= \{\, z \in \mathbb{R}^{2n} \;\lvert\;
H(z) = m_{j, l} \,\}
\]
for $l \in \mathbb{Z}_{>0}$
with $\frac{l}{k} < - \rho_j'(0)$,
where $m_{j, l}$ is the point of $[0, 1]$
such that $\rho_j' (m_{j, l}) = - \frac{l}{k} \pi R^2$.
Let $X_{j, \infty}$, $X_{j, 0}$ and $X_{j, l}$
be the spaces of critical points of $\overline{(F_1^{(j)})^{\sharp k}}$
corresponding to $Z_{j, \infty}$,
$Z_{j, 0}$ and $Z_{j, l}$.
Since $\{ ( \varphi_t^{\rho_j} )^k \} = \{ \varphi_t^{k\rho_j} \}$,
we know from \cite[7.4]{Traynor}
that the fixed point $Z_{j, \infty}$ has Maslov index $2n (l_j + 1)$,
where $l_j$ denotes the maximal $l \in \mathbb{Z}_{>0}$
with $\frac{l}{k} < - \rho_j'(0)$,
the fixed points in $Z_{j, l}$ have Maslov index $2nl$,
and moreover that $X_{j, \infty}$
is a non-degenerate critical point of critical value
$k \rho_j(0) \xrightarrow{j \to \infty} \infty$
and each $X_{j, l}$ is a non-degenerate critical submanifold,
diffeomorphic to $S^{2n-1}$,
with critical value
\[
c_{j, l}
= l \, m_{j, l} \, \pi R^2  + k \, \rho_j (m_{j, l})
\xrightarrow{j \to \infty} l \pi R^2 \,.
\]
By \autoref{proposition: index},
$X_{j, \infty}$ has thus Morse index
$2n (l_j + 1) + k\iota_j + n(k - 1)$,
where $\iota_j$ denotes the index of $(F_1^{(j)})_{\infty}$,
and $X_{j, l}$ has Morse--Bott index $2nl + k\iota_j + n(k - 1)$.
Finally, $X_{j, 0}$ is a space of critical points,
diffeomorphic to a disk,
of critical value $c_{j, 0} = 0$,
corresponding to fixed points of Maslov index $0$.

Recall also from \autoref{section: invariant generating functions symplectic}
that the $\mathbb{Z}_k$-action
on the critical submanifolds of $(F_1^{(j)})^{\sharp k}$
corresponds to the $\mathbb{Z}_k$-action
on the space of fixed points of $(\varphi_1^{\rho_j})^k$
generated by the map that sends a fixed point $p$
to $\varphi_1^{\rho_j}(p)$.
If $k$ is prime and $l < k$
the $\mathbb{Z}_k$-action
on $X_{j, l} \cong S^{2n-1}$ is thus free,
and so the quotient is diffeomorphic
to the lens space $L_k^{2n-1}$.

To calculate
\[
G_{\mathbb{Z}_k, 2nl}^{(a, \infty]} (F_t^{(j)})
= H_{\mathbb{Z}_k, 2nl + k\iota_j + n(k-1)}
\big( \{ \overline{(F_1^{(j)})^{\sharp k}} \leq \infty \} \,,\,
\{ \overline{(F_1^{(j)})^{\sharp k}} \leq a \} \,;\, \mathbb{Z}_k \big)
\]
for $0 < l < k$
we then use the following facts:
\begin{itemize}
\item[(i)]
For $b_1 < b_2$ we have the long exact sequence
\[
\dots \longrightarrow G_{\mathbb{Z}_k, \ast}^{(b_1, b_2]}
(F_t^{(j)})
\longrightarrow G_{\mathbb{Z}_k, \ast}^{(b_1, \infty]}
(F_t^{(j)})
\longrightarrow G_{\mathbb{Z}_k, \ast}^{(b_2, \infty]}
(F_t^{(j)})
\longrightarrow G_{\mathbb{Z}_k, \ast - 1}^{(b_1, b_2]}
(F_t^{(j)})
\longrightarrow \cdots
\]
of the triple
$\{ \overline{(F_1^{(j)})^{\sharp k}} \leq b_1 \}
\subset \{ \overline{(F_1^{(j)})^{\sharp k}} \leq b_2 \}
\subset \{ \overline{(F_1^{(j)})^{\sharp k}} \leq \infty \}$.

\item[(ii)]
If $(b_1, b_2]$ does not contain
critical values of $\overline{(F_1^{(j)})^{\sharp k}}$
then
$G_{\mathbb{Z}_k, \ast}^{(b_1, b_2]} (F_t^{(j)}) \cong 0$
for all $\ast$.
If $(b_1, b_2]$ contains only one critical value $c$,
corresponding to a critical submanifold $X_c$ of index $i(c)$,
then $\big( \{ \overline{(F_1^{(j)})^{\sharp k}} \leq b_2 \} \,,\,
\{ \overline{(F_1^{(j)})^{\sharp k}} \leq b_1 \}\big)$
is homotopy equivalent to $\big( D(E), S(E) \big)$,
where $E$ is a $\mathbb{Z}_k$-invariant vector bundle over $X_c$
of rank $i(c)$,
and so by the equivariant Thom isomorphism
we have
\[
G_{\mathbb{Z}_k, \ast}^{(b_1, b_2]} (F_t^{(j)})
\cong
H_{\mathbb{Z}_k, \ast + k\iota_j + n(k-1) - i(c)} (X_c; \mathbb{Z}_k) \,.
\]
In particular,
\[
G_{\mathbb{Z}_k, \ast}^{(b_1, b_2]} (F_t^{(j)})\cong 0
\quad \text{ for all } \ast
< i(c) - k\iota_j - n(k - 1)\,,
\]
and if $c = c_{j, l}$ with $0 < l < k$
then
\[
G_{\mathbb{Z}_k, \ast}^{(b_1, b_2]} (F_t^{(j)})
\cong H_{\mathbb{Z}_k, \ast - 2nl} (S^{2n-1}; \mathbb{Z}_k)
\cong H_{\ast - 2nl} (L_k^{2n-1}; \mathbb{Z}_k)
\]
\[
\cong \begin{cases}
\mathbb{Z}_k &\text{if } \ast = 2nl, \dots, 2n(l+1) - 1 \\
0 &\text{otherwise.}
\end{cases}
\]
\item[(iii)] For every $\epsilon > 0$
such that $0$ is the only critical value
of $\overline{(F_1^{(j)})^{\sharp k}}$ in $[- \epsilon, \epsilon]$
we have
\[
G_{\mathbb{Z}_k, \ast}^{(- \epsilon, \epsilon]} (F_t^{(j)})
\cong H_{\mathbb{Z}_k, \ast} (\pt; \mathbb{Z}_k) \,.
\]
Indeed, since the space $X_{j, 0}$ of critical points of critical value $0$
is diffeomorphic to a disk $D$
and the corresponding fixed points have Maslov index $0$
(hence, by \autoref{proposition: index},
the index of the Hessian of $\overline{(F_1^{(j)})^{\sharp k}}$
at every point of $X_{j, 0}$ is $k \iota_j + n (k-1)$),
$\big( \{ \overline{(F_1^{(j)})^{\sharp k}} \leq \epsilon \} \,,\,
\{ \overline{(F_1^{(j)})^{\sharp k}} \leq - \epsilon \}\big)$
is homotopy equivalent to $\big( D(E), S(E) \big)$,
where $E$ is a $\mathbb{Z}_k$-invariant vector bundle over $D$ 
of rank $k \iota_j + n (k-1)$.
The claim then follows from the equivariant Thom isomorphism.
\item[(iv)] By the equivariant Thom isomorphism,
and since the $\mathbb{Z}_k$-action
on the base $S^{2n}$ of the vector bundle
$S^{2n} \times \mathbb{R}^{2n (k-1)} \times \mathbb{R}^{N_jk} \rightarrow S^{2n}$
(on the total space of which the function
$\overline{(F_1^{(j)})^{\sharp k}}$ is defined) is trivial,
we have
\[
G_{\mathbb{Z}_k, \ast}^{(- \infty, \infty]} (F_t^{(j)})
\cong H_{\mathbb{Z}_k, \ast} (S^{2n}; \mathbb{Z}_k)
\cong H_{\ast} (S^{2n}; \mathbb{Z}_k) \otimes H_{\mathbb{Z}_k, \ast} (\pt; \mathbb{Z}_k) \,.
\]
\end{itemize}

We can now prove the following result.

\begin{lemma}\label{proposition: calculation phi equivariant}
Let $a > 0$ and $0 < l < k$ with $k$ prime.
Then
\[
G_{\mathbb{Z}_k, 2nl}^{(a, \infty]} (F_t^{(j)})
\cong
\begin{cases}
\mathbb{Z}_k \quad &\text{if }
a < c_{j, l} \\
0 &\text{otherwise.}
\end{cases}
\]
Moreover,
for every $a, a'$ with $a < a' < c_{j, l}$
the map $G_{\mathbb{Z}_k, 2nl}^{(a, \infty]} (F_t^{(j)})
\rightarrow G_{\mathbb{Z}_k, 2nl}^{(a', \infty]} (F_t^{(j)})$
induced by the inclusion
is an isomorphism.
\end{lemma}

\begin{proof}
Using (ii) we see that for $a > c_{j, l}$
and $\ast < 2n (l+1)$
the long exact sequence in (i)
for $b_1 = a$ and $b_2 = b$ such that
$(a, b]$ contains only one critical value is
\[
0 \longrightarrow
G_{\mathbb{Z}_k, \ast}^{(a, \infty]}(F_t^{(j)})
\longrightarrow
G_{\mathbb{Z}_k, \ast}^{(b, \infty]}(F_t^{(j)})
\longrightarrow 0 \,,
\]
so $G_{\mathbb{Z}_k, \ast}^{(a, \infty]}(F_t^{(j)})
\cong G_{\mathbb{Z}_k, \ast}^{(b, \infty]}(F_t^{(j)})$.
By crossing one by one all the critical values
bigger than $a$
we thus obtain $G_{\mathbb{Z}_k, \ast} ^{(a, \infty]} (F_t^{(j)}) \cong 0$.

For $c_{j, l - 1} < a < c_{j, l}$
and $\ast = 2nl, \dots, 2n (l+1) - 2$
the long exact sequence in (i)
for $b_1 = a$ and $b_2 = b$ with $c_{j, l} < b < c_{j, l+1}$ is
\[
0 \longrightarrow
G_{\mathbb{Z}_k, \ast}^{(a, b]}(F_t^{(j)})
\longrightarrow
G_{\mathbb{Z}_k, \ast}^{(a, \infty]}(F_t^{(j)})
\longrightarrow 0 \,,
\]
and so by (ii) we have
$G_{\mathbb{Z}_k, \ast}^{(a, \infty]} (F_t^{(j)})
\cong G_{\mathbb{Z}_k, \ast}^{(a, b]} (F_t^{(j)})
\cong \mathbb{Z}_k$.
In particular,
$G_{\mathbb{Z}_k, 2nl}^{(a, b]} (F_t^{(j)}) \cong \mathbb{Z}_k$.
Moreover, 
for every $a, a'$ with $c_{j, l - 1} < a < a' < c_{j, l}$
the long exact sequence in (i) for $b_1 = a$ and $b_2 = a'$
shows that the map $G_{\mathbb{Z}_k, 2nl}^{(a, \infty]} (F_t^{(j)})
\rightarrow G_{\mathbb{Z}_k, 2nl}^{(a', \infty]} (F_t^{(j)})$
is an isomorphism.

Suppose now that $l > 1$,
and consider $a$ with $c_{j, l-2} < a < c_{j, l-1}$.
The long exact sequence in (i) for $b_1 = a$ and $b_2 = b$
with $c_{j, l-1} < b < c_{j, l}$
gives
\begin{equation}\label{equation: les}
0 \longrightarrow
G_{\mathbb{Z}_k, 2nl}^{(a, \infty]}(F_t^{(j)})
\longrightarrow
G_{\mathbb{Z}_k, 2nl}^{(b, \infty]}(F_t^{(j)})
\longrightarrow
G_{\mathbb{Z}_k, 2nl-1}^{(a, b]}(F_t^{(j)})
\longrightarrow
G_{\mathbb{Z}_k, 2nl-1}^{(a, \infty]}(F_t^{(j)})
\longrightarrow 0 \,.
\end{equation}
By (ii), $G_{\mathbb{Z}_k, 2nl-1}^{(a, b]}(F_t^{(j)}) \cong \mathbb{Z}_k$.
Thus either $G_{\mathbb{Z}_k, 2nl-1}^{(a, \infty]}(F_t^{(j)}) \cong \mathbb{Z}_k$
or $G_{\mathbb{Z}_k, 2nl-1}^{(a, \infty]}(F_t^{(j)}) \cong 0$.
Suppose that $G_{\mathbb{Z}_k, 2nl-1}^{(a, \infty]}(F_t^{(j)}) \cong 0$.
Then, since $G_{\mathbb{Z}_k, 2nl}^{(b, \infty]}(F_t^{(j)}) \cong \mathbb{Z}_k$,
the long exact sequence \eqref{equation: les} implies that
$G_{\mathbb{Z}_k, 2nl}^{(a, \infty]}(F_t^{(j)}) = 0$.
But then the long exact sequence in (i)
with $b_1 < 0$ and $b_2 = a$ gives
\[
0 \longrightarrow
G_{\mathbb{Z}_k, 2nl-1}^{(b_1, a]}(F_t^{(j)})
\longrightarrow
G_{\mathbb{Z}_k, 2nl-1}^{(b_1, \infty]}(F_t^{(j)})
\longrightarrow 0 \,.
\]
This is absurd.
Indeed, on the one hand
$G_{\mathbb{Z}_k, 2nl-1}^{(b_1, \infty]}(F_t^{(j)}) \cong \mathbb{Z}_k \oplus \mathbb{Z}_k$
by (ii) and (iv).
On the other hand,
crossing one by one the critical values between $a$ and $0$
and applying (ii) and the exact sequence of (i) at every step
we see that $G_{\mathbb{Z}_k, 2nl-1}^{(b_1, a]}(F_t^{(j)})
\cong G_{\mathbb{Z}_k, 2nl-1}^{(b_1, \epsilon]}(F_t^{(j)})$.
But, by (iii), $G_{\mathbb{Z}_k, 2nl-1}^{(b_1, \epsilon]}(F_t^{(j)}) \cong \mathbb{Z}_k$.
We conclude that $G_{\mathbb{Z}_k, 2nl-1}^{(a, \infty]}(F_t^{(j)}) \cong \mathbb{Z}_k$
and so, by the long exact sequence \eqref{equation: les},
$G_{\mathbb{Z}_k, 2nl}^{(a, \infty]}(F_t^{(j)}) \cong \mathbb{Z}_k$
and the map $G_{\mathbb{Z}_k, 2nl}^{(a, \infty]}(F_t^{(j)})
\rightarrow G_{\mathbb{Z}_k, 2nl}^{(b, \infty]}(F_t^{(j)})$
is an isomorphism.
Using again (i) and (ii) we then obtain
the same conclusions also for $0< a < c_{j, l-2}$.
\end{proof}

Using \autoref{proposition: calculation phi equivariant}
we obtain the following result.

\begin{prop}\label{proposition: calculation limit balls}
If $k$ is prime and $0 < l < k$
then for any $a > 0$
we have
\[
G_{\mathbb{Z}_k, 2nl}^{(a, \infty]} \big(B^{2n}(R)\big)
\cong
\begin{cases}
\mathbb{Z}_k \quad &\text{for }
a < l \pi R^2 \\
0 &\text{otherwise.}
\end{cases}
\]
\end{prop}

\begin{proof}
Since $(c_{j, l})_j$ is an increasing sequence
with $\lim_{j \to \infty} c_{j, l} = l \pi R^2$,
if $a \geq l \pi R^2$ then $a > c_{j, l}$ for all $j$.
It thus follows from \autoref{proposition: calculation phi equivariant}
that in this case we have
$G_{\mathbb{Z}_k, 2nl}^{(a, \infty]} \big(B^{2n}(R)\big) \cong 0$.
Suppose now that $a < l \pi R^2$.
Then for $i < j$ big enough we have
$a < c_{i, l} < c_{j, l} < l \pi R^2$.
Define $\rho^{(s)} = (1-s) \rho_i + s \rho_j$
and $H^{(s)} = \rho^{(s)} \circ H$,
and let $\{\varphi_t^{(s)}\}$ be the flow of $H^{(s)}$.
Since the family $s \mapsto H^{(s)}$ is increasing,
by \autoref{lemma: monotonicity special} there is a 2-parameter family $F_{s, t}$
of special generating functions quadratic at infinity for $\{\varphi_t^{(s)}\}$
such that $F_{s, 0}$ is a non-degenerate quadratic form for every $s$,
$(F_{s, t})_{\infty}$ does not depend on $s$ and $t$,
and $\frac{d}{ds} F_{s, t} \geq 0$.
In particular, $F_t^{(i)} := F_{0, t}$ and $F_t^{(j)} := F_{1, t}$
are 1-parameter families of special generating functions quadratic at infinity
for $\{ \varphi_t^{\rho_i} \}$ and $\{ \varphi_t^{\rho_j} \}$ respectively
with $F_1^{(i)} \leq F_1^{(j)}$,
and so $\overline{(F_1^{(i)})^{\sharp k}} \leq \overline{(F_1^{(j)})^{\sharp k}}$.
Recall that the monotonicity homomorphism
\[
\mu_i^j: G_{\mathbb{Z}_k, 2nl}^{(a, \infty]} ( \{\varphi_t^{\rho_j}\} )
\rightarrow G_{\mathbb{Z}_k, 2nl}^{(a, \infty]} ( \{\varphi_t^{\rho_i}\} )
\]
is given by $\mu_i^j = (i_{F_t^{(i)}})^{-1} \circ i_{\ast} \circ i_{F_t^{(j)}}$,
where $i_{\ast}: G_{\mathbb{Z}_k, 2nl}^{(a, \infty]}(F_t^{(j)})
\rightarrow G_{\mathbb{Z}_k, 2nl}^{(a, \infty]}(F_t^{(i)})$
is the homomorphism induced by the inclusion
of the sublevel sets of $\overline{(F_1^{(j)})^{\sharp k}}$
into those of $\overline{(F_1^{(i)})^{\sharp k}}$,
and $i_{F_t^{(i)}}: G_{\mathbb{Z}_k, 2nl}^{(a, \infty]} ( \{\varphi_t^{\rho_i}\} )
\rightarrow G_{\mathbb{Z}_k, 2nl}^{(a, \infty]}(F_t^{(i)})$
and $i_{F_t^{(j)}}: G_{\mathbb{Z}_k, 2nl}^{(a, \infty]} ( \{\varphi_t^{\rho_j}\} )
\rightarrow G_{\mathbb{Z}_k, 2nl}^{(a, \infty]}(F_t^{(j)})$
are the isomorphisms given by \autoref{proposition: commuting isomorphism}.
Our result thus follows if we prove that $i_{\ast}$ is an isomorphism.
Let $s \mapsto a_s$ for $s \in [0, 1]$
be an increasing path with $a_1 = a$
and such that, for every $s$,
$a_s$ is a regular value
of $\overline{( F_{s, 1} )^{\sharp k}}$.
Then $a_0 < a < c_{i, l}$.
Moreover,
since $(F_{s, 1})_{\infty}$ does not depend on $s$,
$\frac{d}{ds} F_{s, 1}$ is bounded
and so $\frac{d}{ds} \, \overline{(F_{s, 1})^{\sharp k}}$ is bounded.
We can thus apply \autoref{lemma: continuation}
to obtain a $\mathbb{Z}_k$-equivariant isotopy
mapping $\{ \overline{(F_{0, 1})^{\sharp k}} \leq a_0 \}$
to $\{ \overline{(F_{s, 1})^{\sharp k}} \leq a_s \}$.
In particular, we get an isomorphism
\[
G_{\mathbb{Z}_k, 2nl}^{(a, \infty]} (F_t^{(j)})
\xrightarrow{\cong}
G_{\mathbb{Z}_k, 2nl}^{(a_0, \infty]} (F_t^{(i)})
\]
fitting into the commutative diagram
\[
\begin{tikzcd}
G_{\mathbb{Z}_k, 2nl}^{(a, \infty]} (F_t^{(j)})
\arrow{r}{\cong} \arrow{rd}{i_{\ast}}
& G_{\mathbb{Z}_k, 2nl}^{(a_0, \infty]} (F_t^{(i)})
\arrow{d}\\
& G_{\mathbb{Z}_k, 2nl}^{(a, \infty]} (F_t^{(i)}) \,,
\end{tikzcd}
\]
where the vertical arrow is the map induced by the inclusion.
This map is an isomorphism by \autoref{proposition: calculation phi equivariant}.
We thus conclude that $i_{\ast}$ is an isomorphism.
\end{proof}

Finally we prove the last ingredient entering in the proof
of \autoref{theorem: main}.

\begin{prop}\label{proposition: inclusion balls}
If $k$ is prime, $0 < l < k$ and $0< a < l \pi R_2^2 < l \pi R_1^2$
then the homomorphism
\[
G_{\mathbb{Z}_k, 2nl}^{(a, \infty]} \big(B^{2n}(R_1)\big)
\rightarrow
G_{\mathbb{Z}_k, 2nl}^{(a, \infty]} \big(B^{2n}(R_2)\big)
\]
induced by the inclusion
of $B^{2n} (R_2)$ into $B^{2n} (R_1)$
is an isomorphism.
\end{prop}

\begin{proof}
Let $H^{(1)} (z_1, \dots, z_n)
= \sum_{j=1}^n \frac{\lvert z_j \rvert^2}{R_1^2}$
and $H^{(2)} (z_1, \dots, z_n)
= \sum_{j=1}^n \frac{\lvert z_j \rvert^2}{R_2^2}$,
and consider the cofinal sequences
$\{\varphi_t^{\rho_1, 1}\} \leq \{\varphi_t^{\rho_2, 1}\}
\leq \{\varphi_t^{\rho_3, 1}\} \leq \cdots$
and $\{\varphi_t^{\rho_1, 2}\} \leq \{\varphi_t^{\rho_2, 2}\}
\leq \{\varphi_t^{\rho_3, 2}\} \leq \cdots$
in $\mathcal{H}_{a, \infty}^k \big(B^{2n}(R_1)\big)$
and $\mathcal{H}_{a, \infty}^k \big(B^{2n}(R_2)\big)$
respectively,
where $\{\varphi_t^{\rho_j, 1}\}$ and $\{\varphi_t^{\rho_j, 2}\}$
are the flows of the Hamiltonian functions
$H_{\rho_j}^{(1)} = \rho_j \circ H^{(1)}$
and $H_{\rho_j}^{(2)} = \rho_j \circ H^{(2)}$.
For $s \in [1, 2]$ let $R_s = (2-s) R_1 + (s-1) R_2$,
and define
$H^{(s)} (z_1, \dots, z_n)
= \sum_{j=1}^n \frac{\lvert z_j \rvert^2}{R_s^2}$
and $H_{\rho_j}^{(s)} = \rho_j \circ H^{(s)}$.
Let $\{\varphi_t^{\rho_j, s}\}$ be the flow of $H_{\rho_j}^{(s)}$.
For every $j$,
the family $s \mapsto H_{\rho_j}^{(s)}$ is decreasing.
By \autoref{lemma: monotonicity special},
there is thus a 2-parameter family $F_t^{(j, s)}$
of special generating functions quadratic at infinity
for $\{\varphi_t^{\rho_j, s}\}$
such that $F_0^{(j, s)}$ is a non-degenerate quadratic form for every $s$,
$(F_t^{(j, s)})_{\infty}$ does not depend on $s$ and $t$,
and $\frac{d}{ds} \, F_t^{(j, s)} \leq 0$.
In particular,
$F_t^{(j, 1)}$ and $F_t^{(j, 2)}$
are 1-parameter families of special generating functions quadratic at infinity
for $\{\varphi_t^{\rho_{j},1}\}$ and $\{\varphi_t^{\rho_{j},2}\}$ respectively
with $F_t^{(j, 2)} \leq F_t^{(j, 1)}$.
Similarly to the proof of \autoref{proposition: calculation limit balls},
we have to show that for $j$ big enough
(so that $c_{j, l}^{(1)}$ and $c_{j, l}^{(2)}$ are bigger than $a$)
the homomorphism
\[
i_{\ast}: G_{\mathbb{Z}_k, 2nl}^{(a, \infty]} (F_t^{(j, 1)})
\longrightarrow 
G_{\mathbb{Z}_k, 2nl}^{(a, \infty]} (F_t^{(j, 2)})
\]
induced by the inclusion
is an isomorphism.
Let $s \mapsto a_s$ for $s \in [1, 2]$
be a decreasing path with $a_1 = a$
and such that, for every $s$,
$a_s$ is a regular value
of $\overline{( F_1^{(j, s)} )^{\sharp k}}$.
Then $a_2 < a < c_{j, l}^{(2)}$.
Moreover,
since $(F_1^{(j, s)})_{\infty}$ does not depend on $s$,
$\frac{d}{ds} F_1^{(j, s)}$ is bounded
and so $\frac{d}{ds} \, \overline{(F_1^{(j, s)})^{\sharp k}}$ is bounded.
We can thus apply \autoref{lemma: continuation}
to obtain a $\mathbb{Z}_k$-equivariant isotopy
mapping $\{ \overline{(F_1^{(j, 1)})^{\sharp k}} \leq a \}$
to $\{ \overline{(F_1^{(j, s)})^{\sharp k}} \leq a_s \}$.
In particular,
we get an isomorphism
\[
G_{\mathbb{Z}_k, 2nl}^{(a, \infty]} (F_t^{(j, 1)})
\xrightarrow{\cong}
G_{\mathbb{Z}_k, 2nl}^{(a_2, \infty]} (F_t^{(j, 2)})
\]
fitting into the commutative diagram
\[
\begin{tikzcd}
G_{\mathbb{Z}_k, 2nl}^{(a, \infty]} (F_t^{(j, 1)})
\arrow{r}{\cong} \arrow{rd}{i_{\ast}}
& G_{\mathbb{Z}_k, 2nl}^{(a_2, \infty]} (F_t^{(j, 2)})
\arrow{d}\\
& G_{\mathbb{Z}_k, 2nl}^{(a, \infty]} (F_t^{(j, 2)}) \,,
\end{tikzcd}
\]
where the vertical arrow is the map induced by the inclusion.
This map is an isomorphism by \autoref{proposition: calculation phi equivariant}.
We thus conclude that $i_{\ast}$ is an isomorphism.
\end{proof}


\end{document}